\documentclass{article}
\usepackage{amsmath,amssymb,amsthm,graphicx,epsfig,float,url}
\usepackage[colorlinks=true,citecolor={Plum},linkcolor={Periwinkle}]{hyperref}
\usepackage{pdfsync}
\usepackage[usenames, dvipsnames]{xcolor}
\usepackage{tikz}
\usepackage{subfig}
\usepackage{bbm}
\usepackage{stmaryrd}
\usepackage{amssymb}
\usepackage{mathrsfs}  
\usepackage{caption}
\usepackage{float}
\usepackage{esint}

\usepackage{dsfont}
\usepackage[makeroom]{cancel}
\usetikzlibrary{patterns}
\newcommand{\R}{\textnormal{I\kern-0.21emR}}
\newcommand{\N}{\textnormal{I\kern-0.21emN}}

\renewcommand{\geq}{\geqslant}
\renewcommand{\leq}{\leqslant}

\def\e{{\varepsilon}}

\def\tamu{{\theta_{\alpha,\mu}}}

\def\taa{{\theta_{\alpha_1,\alpha_2,\mu}}}

\allowdisplaybreaks

\def\YYint#1#2#3{{\setbox0=\hbox{$#1{#2#3}{\iint}$}
    \vcenter{\hbox{$#2#3$}}\kern-.51\wd0}}
 
\usepackage[dvipsnames]{xcolor}

\newtheorem*{theorem*}{Theorem}

\newtheorem{theorem}{Theorem}

\DeclareMathOperator*{\argmax}{arg\,max}

\usepackage{color}
\definecolor{lightgray}{gray}{0.85}

\newtheorem{material}{material}
\newtheorem{proposition}[material]{Proposition}

\newtheorem{definition}[material]{Definition}
\newtheorem{lemma}[material]{Lemma}

\newtheorem{remark}[material]{Remark}
\newtheorem{algo}{Algorithm}

\def\O{{\Omega}}

\def\n{{\nabla}}

\def\p{{\varphi}}

\usepackage{xargs}
 \usepackage[colorinlistoftodos,textsize=small]{todonotes}
 \newcommandx{\christian}[2][1=]{\todo[linecolor=red,backgroundcolor=red!25,bordercolor=red,#1]{#2}}
 \newcommandx{\laura}[2][1=]{\todo[linecolor=blue,backgroundcolor=blue!25,bordercolor=blue,#1]{#2}}
 \newcommandx{\info}[2][1=]{\todo[linecolor=green,backgroundcolor=green!25,bordercolor=green,#1]{#2}}
 \newcommandx{\improvement}[2][1=]{\todo[linecolor=yellow,backgroundcolor=yellow!25,bordercolor=yellow,#1]{#2}}
 
  \newcommandx{\biblio}[2][1=]{\todo[linecolor=blue,backgroundcolor=magenta!25,bordercolor=blue,#1]{#2}}

 \numberwithin{equation}{section}

\begin{document}
\title{Spatial ecology, optimal control and game theoretical fishing problems}


\author{Idriss Mazari\footnote{CEREMADE, UMR CNRS 7534, Universit\'e Paris-Dauphine, Universit\'e PSL, Place du Mar\'echal De Lattre De Tassigny, 75775 Paris cedex 16, France, (mazari@ceremade.dauphine.fr).}, \quad Dom\`enec Ruiz-Balet\footnote{ 
Department of Mathematics
Imperial College, South Kensington, London, UK, (d.ruiz-i-balet@imperial.ac.uk).
 }}
\date{\today}

\maketitle

\begin{abstract} 
Of paramount importance in both ecological systems and economic policies are the problems of harvesting of natural resources. A paradigmatic situation where this question is raised is that of fishing strategies. Indeed, overfishing is a well-known problem in the management of live-stocks, as being too greedy may lead to an overall dramatic depletion of the population we are harvesting. A closely related topic is that of Nash equilibria in the context of fishing policies. Namely, two players being in competition for the same pool of resources, is it possible for them to find an equilibrium situation? 

The goal of this paper is to provide a detailed analysis of these two queries (\emph{i.e} optimal fishing strategies for single-player models and study of Nash equilibria for multiple players games) by using a basic yet instructive mathematical model, the logistic-diffusive equation. In this framework,  the underlying model simply reads $-\mu\Delta \theta=\theta(K(x)-\alpha(x)-\theta)$ where $K$ accounts for natural resources, $\theta$ for the density of the population that is being harvested and $\alpha=\alpha(x)$ encodes either the single player fishing strategy or, when dealing with Nash equilibria, a combination of the fishing strategies of both players. This article consists of two main parts. The first one gives a very fine characterisation of the optimisers for the single-player game where one aims at solving $\sup_\alpha \int_\O \alpha \theta$, under $L^\infty$ and $L^1$ constraints on the fishing strategies $\alpha$. In particular, we show that, depending on the value of these constraints, this optimal control problem may behave like a convex or, conversely, concave problem. We also provide a detailed analysis of the large diffusivity limit of this problem. In the case where two players are involved, we rather write $\alpha$ as $\alpha_1+\alpha_2$ where $\alpha_i$, the fishing strategy of the $i$-th player, also satisfies $L^\infty$ and $L^1$ constraints. Defining $I_1:=\int_\O \alpha_i \theta$ we aim at finding a Nash equilibrium. We prove the existence of Nash equilibria in several different regimes \textcolor{black}{and investigate several related qualitative queries, for instance providing examples of the well-known \emph{tragedy of commons}}.

Our study is completed by a variety of numerical simulations that illustrate our results and allow us to formulate open questions and conjectures.
\end{abstract}


\medskip


\paragraph{Acknowledgment.} This work was started during a visit of D. Ruiz-Balet at CEREMADE. I. Mazari was partially supported by the French ANR Project ANR-18-CE40-0013 - SHAPO on Shape Optimization and by the Project ”Analysis and simulation of optimal shapes - application to life sciences” of the Paris City Hall.

\section{Introduction}
\subsection{Scope of the paper \& summary of the models}

In this paper, we study an optimal harvesting problem motivated by the ecological management of wild fisheries. One of the main ecological threats we currently face is the depletion of fish populations in oceans \cite{BBC2,davies2012extinction,pinsky2011unexpected}. While many factors can be held accountable for this situation, one of the overarching ones is overfishing and, more generally, the poor management of fisheries. The resulting very high strain that is exerted on fishing stocks  puts at risk the biomass \cite{costello2012status,pikitch2012risks}. While it is clear that this overfishing problem may arise when only one population of fishermen is present, the situation can be more dramatic when several populations of fishermen are competing for the same pool of resources. This is an example of the ubiquitous \emph{tragedy of commons} \cite{hardin2009tragedy}: the competition over finite common resources may lead to the extinction of said resources. But not only does this affect the fish population, it also endangers the fishing-based economies of several societies \cite{hamilton2001outport}. Consequently, the future of fisheries and the study of optimal fishing strategies is now a central topic both in the scientific community and in society \cite{BBC50,costello2012status,worm2012future,BBCwaste}.

In the present work, we aim at providing an in-depth analysis of a paradigmatic model of such (over)fishing problems from the perspective of optimal control of spatial ecology models and game theory. Using, as a basic building block, the logistic-diffusive equation, we offer several qualitative results that exemplify the intricate and rich qualitative behaviours of such queries, and provide theoretical illustrations of the aforementioned concepts in the management of fisheries (in particular, of the tragedy of commons).

\paragraph{Summary of the models}
Since the introduction is long let us for the sake of convenience summarise here the models and questions we investigate. In general, the fishes' population will be described using the standard logistic-diffusive equation (see section \ref{singleIntro} for more details): $\theta$ being the population density, we assume that $\theta$ solves
\begin{equation}\label{Imain}
 \begin{cases}
    -\mu\Delta \theta= \theta(K(x)-\theta)-\underbrace{\alpha(x)\theta}_{\text{harvested fish}}\quad &\text{ in }\Omega,\\
  \frac{\partial \theta}{\partial \nu } =0 \quad &\text{ on }\partial\Omega,
 \end{cases}
\end{equation}
where $\alpha$ described the fishing rate, $\mu>0$ is the diffusivity of the population and $K(x)$ accounts for the natural resources available in the environment. The optimisation problem we seek to understand is the maximisation of the fishing outcome:
\begin{equation*}
\max_{\alpha\in L^\infty (\Omega)}\int_\Omega \alpha(x)\theta_\alpha(x)dx,
\end{equation*}
 where by $\theta_\alpha$ we indicate the dependence of $\theta$ with the variable $\alpha$ in \eqref{Imain}. Of course we would need to specify which constraints we enforce on $\alpha$.  We shall make this precise in section \ref{singleIntro}. In certain cases, this problem can be solved explicitly; this is the case when $K$ is a constant, see Remark \ref{Re:KConstant}.
 However, when we consider a general capacity $K(x)$ the study becomes more intricate. The first part of this article is devoted to the study of this optimal fishing problem.
 
The second part is devoted to understanding a related game-theoretical problem. In this model, two populations are fishing in the same pool of natural resources. 
Considering two players, the state equation becomes:
\begin{equation}\label{Igame}
 \begin{cases}
    -\mu\Delta \theta= \theta(K(x)-\theta)-\underbrace{\alpha_1\theta}_{\text{Player 1}}-\underbrace{\alpha_2\theta}_{\text{Player 2}}\quad &\text{ in }\Omega,\\
  \frac{\partial \theta}{\partial \nu } =0 \quad &\text{ on }\partial\Omega,
 \end{cases}
\end{equation}
where each player wants to optimise their fishing output
\begin{equation*}
I_1(\alpha_1,\alpha_2)=\int_\Omega \alpha_1 \theta_{\alpha_1,\alpha_2}dx,\quad I_2(\alpha_1,\alpha_2)=\int_\Omega \alpha_2 \theta_{\alpha_1,\alpha_2}dx,
\end{equation*}
the outcome of one player depends on the strategy of the other player, since, both players have an impact on the total population $\theta$ through equation \eqref{Igame}. A pair of strategies $(\alpha_1^*,\alpha_2^*)$ is said to be a Nash equilibria if
\begin{equation}
\alpha_1^*\in \argmax_{\alpha_1} I_1(\alpha_1,\alpha_2^*),\quad \alpha_2^*\in\argmax_{\alpha_2}I_2(\alpha_1^*,\alpha_2).
\end{equation}
In general, Nash equilibria do not necessary exist, and obtaining their existence  is a core point of this article.
Additionally, we shall give some qualitative contributions to the study of the impact of competition on the total outcome: is it better, when two fishers' population are fishing, to be competing or cooperating? In particular, we will see that competition is sometimes detrimental to the total fishing outcome.
Furthermore, when considering $n$ players we will see a more devastating effect. We will see that as the number of players increase, there exist a Nash equilibrium such that the total harvested amount of fishes goes to $0$ as the number of players increases. Again, this result validates the principle known as \emph{tragedy of the commons}  \cite{hardin2009tragedy}: if one increases the number of players in the harvesting game, the total amount harvested may decrease dramatically.
%

In order to avoid overfishing, typically governments impose regulations on the fishing capacity of the players. Furthermore, players themselves might have limited fishing ability. In this paper we will model this by imposing an integral constraint on $\alpha$,
$\int_\Omega \alpha(x)dx \leq V_0$ or $\int_\Omega \alpha(x)dx = V_0$.

Throughout this study, we shall also cover several aspects of optimal control problems that are interesting in their own right, and that belong to a currently very active field of research devoted to the understanding of spatial heterogeneity in population dynamics and, more generally, in the study of spatial ecologoy \cite{BaiHeLi,BHR,DeAngelis2020,heo2021ratio,InoueKuto,KaoLouYanagida,LamLiuLou,LamboleyLaurainNadinPrivat,LiangLou,LiangZhang,LouInfluence,Lou2008,LouNagaharaYanagida,LouYanagida,MazariThese,MazariNadinPrivat,MNPChapter,MRBSIAP,SuTongYang}. Let us give a more mathematical point of view on our contributions:
\paragraph{From the applied mathematics perspective}
In this paper we investigate several \emph{optimal fishing problems} in spatial ecology. The first class of problem corresponds to a  \emph{single fisher problem}, while the other two deal with \emph{multiple players problems}. In the single fisher case, we mostly investigate the influence of the total fishing capacity on the  qualitative features of optimal fishing strategies, while in the other problems we provide some contribution to the existence of Nash equilibria. For multiple player games, we mostly consider the case of two players. Our approach can also be used for analysing games with more players. Our theoretical analysis is illustrated by several detailed numerical solutions.

\paragraph{From the optimal control perspective}
Another outlook on the results of this paper is to notice that we are investigating a \emph{non-monotonic bilinear optimal control problem}. By this we mean the following in the case of a single fisher problem: the population of fishes being modelled by its density $\theta$ and a fishing strategy being accounted for by a certain function $\alpha$, the equation features a loss term $-\alpha\theta$, while the player tries to optimise a criterion of the form $\fint \alpha \theta$.  Then it is clear that overfishing will be detrimental to the fisher, as it is going to be detrimental for the overall population. In this paper, we exemplify the shift this creates in the qualitative analysis; for instance, maximisers can saturate certain constraints, or not at all depending on the values of the parameters of the problem.

For further references and discussion, we refer to section \ref{Se:Bib} of the introduction.

\subsection{The single fisher problem}\label{singleIntro}
\paragraph{State equation}
Following the seminal papers \cite{Fisher,KPP}, we model our population of fishes according to the logistic diffusive equation: we assume that the population lives in a domain $\O\subset \R^d$, assumed to be bounded and with a $\mathscr C^2$ boundary. The population is modelled by a population density $\theta$ and depends on the characteristic dispersal rate $\mu>0$ of the species, on the resources available in the domain, which are accounted for by a function $K\in L^\infty(\O)$, and the fishing strategy $\alpha\in L^\infty(\O)$ of the single player. In general, we denote by  $\theta_{K,\alpha,\mu}$ the population density. In the course of this paper, when $K,\alpha$ or $\mu$ are fixed, we may drop certain of the subscripts and only use the notations $\theta_\alpha$ or $\theta_{K,\alpha}$ for instance. Overall, $\theta_{\alpha,\mu}$ solves the following logistic-diffusive equation:
\begin{equation}\label{Eq:MainSingle}
\begin{cases}
-\mu \Delta \theta_{K,\alpha,\mu}-\theta_{K,\alpha,\mu}\left(K(x)-\alpha(x)-\theta_{K,\alpha,\mu}\right)=0&\text{ in }\O\,, 
\\ \frac{\partial \theta_{K,\alpha,\mu}}{\partial \nu}=0&\text{ on }\partial \O\,, 
\\ \theta_{K,\alpha,\mu}\geq 0\,, \neq 0.\end{cases}\end{equation} We refer to \cite[Introduction]{MazariThese} and the references therein for more details on the modelling.
The question of existence and uniqueness of solutions of \eqref{Eq:MainSingle} can be tedious. It is known \cite{BHR} that for fixed $\alpha\,, K$ there exists a unique solution to \eqref{Eq:MainSingle} if and only if the first eigenvalue of the operator $-\mu \Delta -(K-\alpha)$ is negative. Since we work on optimisation problems, it is easier to ensure the existence and uniqueness of the solution for any control. As the first eigenvalue is bounded from above \cite{DockeryHutsonMischaikowPernarowskiEvolution} by $\fint_\O\left( \alpha-K\right)$ we will simply work with controls $\alpha$ satisfying
\begin{equation}\label{Contr:Integ} 0<\fint_\O \alpha< \fint_\O K.\end{equation} Under these conditions, classical results from \cite{BHR,31806255e7f648d5b65ff02c30c4f539} guarantee the existence and uniqueness of a solution of \eqref{Eq:MainSingle}. 

We introduce a parameter $K_0\in(0;1)$ and always assume that $K\in \mathcal K(\O)$, where $\mathcal K(\O)$ is defined as 

\begin{equation}\label{Eq:AdmK} \mathcal K(\O):=\left\{K\in L^\infty(\O)\,, 0\leq K\leq 1\,, \fint_\O	 K=K_0\right\},\end{equation}
where, for any $f\in L^1(\O)$ we use the notation 
\[ \fint_\O f=\frac1{\mathrm{Vol}(\O)}\int_\O f.\]

\paragraph{Single player functional}
The functional to optimise in the single player case is the \emph{total fishing output}
\[
J_\mu:\alpha \mapsto \fint_\O \alpha \theta_{\alpha,\mu},\]and the relevant optimisation problem is 
\[
\sup_{\alpha}J_\mu(\alpha).\] Of course, we need to specify which admissible fishing strategies $\alpha$ we consider.

\paragraph{Admissible controls}
Beyond the integral condition \eqref{Contr:Integ}, we  enforce a pointwise bound 
\[0\leq \alpha\leq  \kappa,\] where $\kappa>0$ is a fixed parameter: a single player has a limited fishing capacity at any given spots.

Second, we need to implement a global, $L^1$ constraint (the player has a globally limited fishing ability); in order to still satisfy \eqref{Contr:Integ}, we fixe a parameter $V_0\in (0;K_0)$ and we assume that either all controls satisfy 
\[
\fint_\O \alpha \leq V_0 \quad \text{(Inequality constraint)}\] or, on the other hand, that 
\[ \fint_\O \alpha=V_0\quad \text{(equality constraint)}.\]

Overall, we thus define, for these two fixed paramers $\kappa\,, V_0$, the two admissible classes of controls 
\begin{equation}\label{Adm:Ineq} \mathcal M_\leq(\kappa,V_0)=\left\{\alpha \in L^\infty(\O)\,, 0\leq \alpha\leq \kappa \text{ a.e., } \fint_\O\alpha \leq V_0\right\}\end{equation}
and 

\begin{equation}\label{Adm:Eq} \mathcal M_=(\kappa,V_0)=\left\{\alpha \in L^\infty(\O)\,, 0\leq \alpha\leq \kappa \text{ a.e., } \fint_\O \alpha=V_0\right\}\end{equation}

Working in one or the other of these admissible classes changes the features of the problem drastically. This is related to the  problem of overfishing: as we shall see throughout the proofs, depending on the value of $V_0$, the functional $J_\mu$ may be increasing (in the sense that $\alpha_1\leq \alpha_2\Rightarrow J_\mu(\alpha_1)\leq J_\mu(\alpha_2)$), in which case optimisers for the problem $\sup_{m\in \mathcal M_=(\kappa,v_0)} J_\mu(\alpha)$ are also optimisers for the problem $\sup_{m\in \mathcal M_\leq(\kappa,v_0)} J_\mu(\alpha)$, or loose this monotonicity, in which case the optimisers for the inequality case are strictly better than optimisers for the equality constraint: $\sup_{m\in \mathcal M_=(\kappa,v_0)} J_\mu(\alpha)<\sup_{m\in \mathcal M_\leq(\kappa,v_0)} J_\mu(\alpha)$. This is a first major difference between between the fishing problem and the problem of optimisation of the total population size, where the monotonicity of the functional is a stepping stone for further qualitative analysis of optimisers, see section \ref{Se:Bib}.

\paragraph{The main problem}
Thus, the first two optimisation problems to be considered  here and that are the main foci of the present contribution are:
\begin{equation}\label{Pv:1Ineq}\tag{$\bold{P}_{\leq,V_0}^{\mathrm{single}}$} \fbox{$\displaystyle \sup_{\alpha \in \mathcal M_\leq (\kappa,V_0)} J_\mu(\alpha)$}\end{equation} and 

\begin{equation}\label{Pv:1Eq}\tag{$\bold{P}_{=,V_0}^{\mathrm{single}}$} \fbox{$\displaystyle \sup_{\alpha \in \mathcal M_= (\kappa,V_0)} J_\mu(\alpha)$}\end{equation} 
For these two problems, we can provide a fine analysis in the case of low fishing abilities ($V_0\ll 1$) orin the large diffusivity asymptotic regime $\mu \to \infty$.  In particular, we will show that, in general (\emph{i.e.} for a fixed diffusivity), if $V_0\ll 1$, $J_\mu$  is increasing on $\mathcal M_\leq(\kappa,V_0)$ (Theorem \ref{Th:MonotonicityTrue}), and also concave (Theorem \ref{Th:Concave}) while, in the large diffusivity case $\mu\to \infty$, we can attain an explicit description of optimal strategies (Proposition \ref{Pr:Large0}, Theorem \ref{Th:AsymptoticSingle}).

\paragraph{A "large fishing ability" model to showcase the complexity of fishing problems }

To exemplify, however, the breadth of behaviours such fishing problems can display, we also propose a deep exploration of another asymptotic case, that of \emph{large fishing abilities}.

Let us make this more precise. What we mean here is that the fishing strategy is going to be a small perturbation  of the resources distribution $K$, \emph{i.e.} that any fishing strategy writes $\alpha=K+\delta m$ for a small parameter $\delta>0$. 

This leads us to introduce the auxiliary classes
\[ \mathcal N_{\leq}(\O):=\left\{m\in L^\infty(\O)\,, \Vert m\Vert_{L^\infty(\O)}\leq 1\,,-m_1\leq \fint_\O m\leq=-m_0\right\}\]
and 
\[ \mathcal N_{=}(\O):=\left\{m\in L^\infty(\O)\,, \Vert m\Vert_{L^\infty(\O)}\leq 1\,,\fint_\O m=-m_0\right\}\]  where $m_0$ is a fixed volume constraint, $m_1>-1$ and we define, for any $m\in \mathcal N(\O)$ and any $\delta>0$, the fishing strategy 
\[ \alpha_{\delta,m}:=K+\delta m.\] The parameter $\delta$ is destined to be small, so we are essentially, through this reparameterisation, assuming that fishing strategies  are close to natural resources distribution, and essentially lead to killing all the population off.

\begin{remark}
For any $m\in \mathcal N(\O)$, the zones $\{m<0\}$ correspond to zones where we are not exhausting the natural resources modelled by $K$.
\end{remark}

We define, for any $\delta>0$, the map 
\[ \overline J_{\delta,\mu}:\mathcal N(\O)\ni m\mapsto \fint_\O \alpha_{\delta,m}\theta_{\alpha_{\delta m},\mu}.\] The  related optimisation problems are
\begin{equation}\label{Pv:1IneqLargeFishing}\tag{$\bold{Q}_{\leq,\delta}^{\mathrm{single}}$} 
\fbox{$\displaystyle \sup_{m\in \mathcal N_\leq(\O)}\overline J_{\delta,\mu}(m)$}\end{equation} and 
\begin{equation}\label{Pv:1EqLargeFishing}\tag{$\bold{Q}_{=,\delta}^{\mathrm{single}}$} 
\fbox{$\displaystyle \sup_{m\in \mathcal N_=(\O)}\overline J_{\delta,\mu}(m)$}\end{equation} 
While these two problems seem extremely related to our original formulations \eqref{Pv:1Ineq}-\eqref{Pv:1Eq} the qualitative behaviours of \eqref{Pv:1IneqLargeFishing}-\eqref{Pv:1EqLargeFishing} are very different. For instance, we show in Theorem \ref{Th:MonotonicityTrue} that when $\delta\ll1$ the functional $\overline J_{\delta,\mu}$ is not monotonic, and that it even behaves like a  convex function, in the sense that its maximisers are extreme points of the admissible set (see Theorem \ref{Th:Convex}).

\paragraph{Structure of the statement of the results for single fisher models}
While it would seem natural to divide our presentation of the results in two batches, one devoted to \eqref{Pv:1Eq}-\eqref{Pv:1Ineq} and another to \eqref{Pv:1EqLargeFishing}-\eqref{Pv:1IneqLargeFishing}, the coherence of the methods of proofs used prompts us to rather present them in the following order:

\begin{enumerate}
\item \underline{Monotonicity properties:} in the first two theorems, Theorems \ref{Th:MonotonicityTrue} and \ref{Th:MonotonicityFalse}, we investigate the monotonicity of the functionals $J_\mu$ and $\overline J_{\delta,\mu}$. In Theorem \ref{Th:MonotonicityTrue} we show that \eqref{Pv:1Ineq} and \eqref{Pv:1Eq} coincide when $V_0\ll1$.  In Theorem \ref{Th:MonotonicityFalse} we prove that when $\delta\ll1$ the problems \eqref{Pv:1IneqLargeFishing} and \eqref{Pv:1EqLargeFishing} do not coincide. While such results can be obtained in a very straightforward manner when we consider the case of a constant resources distribution $K$ (see in particular Remark \ref{Re:KConstant}), it is not immediate at all in the case of varying $K$. The interest of Theorem \ref{Th:Concave} is twofold: first, it exemplifies the qualitative change of behaviour of the functional $J_\mu$ when the volume constraint is perturbed. Second, it is an essential building block to obtain concavity properties for the functional and, therefore, to derive the existence of Nash equilibria when we will, in the second part of the paper, study multiple players games.
\item\underline{Concavity and convexity properties:} In  Theorems \ref{Th:Concave}-\ref{Th:Convex}, we focus on the problems with equality constraints \eqref{Pv:1Eq}-\eqref{Pv:1EqLargeFishing}. We first show in Theorem \ref{Th:Concave} that, if $V_0$ is small enough and if $\O$ is one-dimensional then, regardless of the resources distribution $K$, $J_\mu$ is a concave functional, and we identify the maximising controls for particular values of $V_0$ or for particular resources distribution $K$. This relies on very fine properties of the one-dimensional logistic diffusive equation previously investigated in \cite{BaiHeLi}. We prove the same result in higher dimensions, provided $K$ remains close to a constant. We show in particular that if $K$ is constant, then the maximising controls are constant as well. Then, in Theorem \ref{Th:Convex} we show that, if $\delta>0$ is small enough, the functional $J_{\delta,\mu}$ behaves, conversely, like a convex function from the point of view of optimisation in $\mathcal N_=(\O)$:  all solutions of \eqref{Pv:1EqLargeFishing} are extreme points of the admissible sets and so they write $m^*=\kappa \mathds 1_{E^*}$ for some suitable subset $E^*$ of $\O$.
\item\underline{Precised behaviour in asymptotic regimes:} finally, to conclude the theoretical contributions to single player games, we offer an in-depth analysis of the large diffusivity limit $\mu \to \infty$ of the optimisation problem \eqref{Pv:1Eq}. Building on techniques of \cite{Mazari2020}, we give explicit maximisers in the one-dimensional case; we refer to Theorem \ref{Th:AsymptoticSingle}. Similarly, this result will be used to exhibit Nash equilibria in two-players games. 

\end{enumerate}

All these results are gathered in subsection \ref{Se:QualitativeSingle}.

In Section \ref{Se:NumericsSingleCommented}, we present and comment several numerical simulations.

\paragraph{Remark on the techniques used} Throughout this first part of the paper, especially for Theorems \ref{Th:MonotonicityTrue}-\ref{Th:Concave}-\ref{Th:Convex} one of the key ingredient is the second-order technique introduced in \cite{Mazari2021} to tackle the problem of optimising the total population size. While this method proved fruitful in a variety of other situations \cite{M2021,MPRobin}, it is here impossible to apply directly, and it needs to be coupled with some fine analytical study of the functions at hand. The characterisation of optimisers in the large diffusivity limit is on the obtained using rearrangement-like arguments and Talenti inequalities. Specifically, we shall use some results of \cite{Langford} and of \cite{Mazari2020}, the latter being used solely to derive the limit model.
\paragraph{Terminology: bang-bang functions}
We shall often refer in this paper to "bang-bang" functions. They are simply admissible controls that write 
\[ \alpha=\kappa\mathds 1_E.\] Such bang-bang functions are known to be important in the optimal control of reaction-diffusion equations (see in particular section \ref{Se:Bib} of this introduction), and, geometrically, are extreme points of the convex set $\mathcal M_=(\kappa,V_0)$.

\subsection{Qualitative properties for single player games: general diffusivities}\label{Se:QualitativeSingle}
\paragraph{Monotonicity of the fishing output}
We begin with the monotonicity of the fishing output functional and explain how the volume constraint may have an influence on the increasing character of $J_\mu$. Of course, this is a theoretical, optimal control formulation of the overfishing problem.  Before we state our result, let us explain in the following remark that such a result is very much expected when working in homogeneous environments ($K\equiv 1$) where explicit computations allow for an explicit characterisation of maximisers; this shows that monotonicity is not the general rule.

\begin{remark}[A standard example with loss of monotonicity]\label{Re:KConstant}
A simple yet instructive case to exemplify the loss of monotonicity is given by the case $K\equiv K_0$. In this case, for any strategy $\alpha\in \mathcal M_{\leq}(\kappa,V_0)$, $\theta_{\alpha,\mu}$ solves 
\[ -\mu \Delta \theta_{\alpha,\mu}-\theta_{\alpha,\mu}\left(K_0-\theta_{\alpha,\mu}\right)=\alpha \theta_{\alpha,\mu}.\] As $\theta_{\alpha,\mu}$ satisfies Neumann boundary conditions, this entails 
\[ J_\mu(\alpha)=\fint_\O \alpha \theta_{\alpha,\mu}=\fint_\O \tamu(K_0-\tamu).\] Besides, if we assume that $\kappa<2$, so that $\Vert 1-\alpha\Vert_{L^\infty}\leq 1$, the maximum principle implies $\tamu\leq 1$ almost everywhere. As the maximiser of $\p:x\mapsto x(K_0-x)$ on $[0;K_0]$ is reached at $x=\frac{K_0}2$ it follows that 
\[J_\mu(\alpha)\leq \p\left(\frac12\right),\] with equality if, and only if, $\tamu\equiv \frac{K_0}2$. However, $\tamu=\frac{K_0}2$ if and only if $\alpha \equiv\frac{K_0}2$. We thus obtain the following conclusion: for any $V_0\geq \frac{K_0}2$, $\alpha^*\equiv \frac{K_0}2$ is the unique maximiser of $J_\mu$ on $\mathcal M_\leq (\kappa,V_0)$. In particular, if $V_0>\frac{K_0}2$, the volume constraint is not saturated in \eqref{Pv:1Ineq}.
\end{remark}
We now state our main theorem:
\begin{theorem}\label{Th:MonotonicityTrue}
 Let $\kappa>0$ be fixed. There exists $\e_1>0$ such that, if $V_0\in (0;\e_1)$, the map $\alpha \mapsto J_\mu(\alpha)$ is monotonic on $\mathcal M_\leq(\kappa,V_0)$:
\[ \alpha_1\leq \alpha_2\Rightarrow J_\mu (\alpha_1)\leq J_\mu(\alpha_2).\] As a consequence, any solution $\alpha^*$ of \eqref{Pv:1Ineq} satisfies
\[ \fint_\O \alpha^*=V_0.\]
\end{theorem}
Our second theorem deals with \eqref{Pv:1IneqLargeFishing}-\eqref{Pv:1EqLargeFishing}:
\begin{theorem}\label{Th:MonotonicityFalse} There exists $\delta_1>0$ such that, for any $ \delta\in (0;\delta_1)$, the functional $\overline J_{\delta,\mu}$ is not increasing on $\mathcal N_\leq(\O)$; furthermore, for any solution $\alpha^*$ of \eqref{Pv:1IneqLargeFishing}, there holds 
\[ \fint_\O \alpha^*<V_0.\]
\end{theorem}

As was explained for example in \cite{Mazari2021}, the monotonicity is intimately linked to pointwise properties of optimisers. In \cite{Mazari2021,M2021} it is shown that for certain bilinear control problems, the monotonocity of the functional entails that optimisers are extreme points of the convex set under consideration, the aforementioned "bang-bang" functions. Here, we show related results, in that we obtain concavity and convexity-like properties. The first theorem deals with the "low fishing capacity" limit.

\begin{theorem}\label{Th:Concave}
\begin{enumerate}
\item Assume $\O=(0;1)$ i.e. that we are working in the one-dimensional case. There exists $\e_2>0$ such that, for any $V_0\in (0;\e_2)$, the map $J_\mu$ is strictly concave on $\mathcal M_\leq(\kappa,V_0)$. If $K$ is constant, and if $V_0\in (0; \e_2)$, the solution of \eqref{Pv:1Eq} and of \eqref{Pv:1Ineq} is $\overline \alpha \equiv V_0.$
\item In any dimension $d$, there exists $\e_2>0$ and $\e_3>0$ such that for any $V_0\in (0;\e_2)$ and for any $K\in \mathcal K(\O)$ such that, defining $\overline K:=K_0$, 
\[\Vert K-\overline K\Vert_{L^1(\O)}\leq \e_3\] then the map $J_\mu$ is strictly concave on $\mathcal M_\leq(\kappa,V_0)$. If $K$ is constant, and if $V_0\in (0; \e_2)$, the solution of \eqref{Pv:1Eq} and of \eqref{Pv:1Ineq} is $\overline \alpha \equiv V_0.$

\end{enumerate}
\end{theorem}

\begin{theorem}\label{Th:Convex}
There exists $\delta_2>0$ such that, for any $0<\delta<\delta_2$,  any solution $m^*$ of \eqref{Pv:1EqLargeFishing} is a bang-bang function: there exists a subset $E^*\subset \O$ such that 
\[ m^*=-\mathds 1_{E^*}.\]

\end{theorem}

As mentioned before we stated the Theorem, the parameters $\delta_1\,, \delta_2$ are linked to the monotonicity of the functional and it will be shown through the proof that 
\[ \delta_k\leq \e_k\quad (k=1,2).\]

\begin{remark}In Theorems \ref{Th:MonotonicityFalse} and \ref{Th:Convex} we have interpreted "large fishing capacity limit" in an $L^\infty$ sense, by requiring that the $L^\infty$ distance from $K$ to any fishing strategy be small. Another possibility would be to require that the $L^1$ distance of $K$ to the admissible controls is small. 
\end{remark}

\paragraph{Comment on the proofs}
The proofs of the three theorems above rely on the computation of first and second-orde Gateaux derivatives of the map $J_\mu$. The first order Gateaux-derivative of $J_\mu$ will be denoted by $\dot J_\mu$.
These computations can be used to determine whether or not certain configurations can be optimal, by checking whether or not they satisfy first order optimality conditions.

\paragraph{The large diffusivity limit for single player games: precised change of convexity}All the information above can be made much more precise in certain asymptotic limits. In this section, we analyse in depth the behaviour, as $\mu \to \infty$, of the optimisation problems \eqref{Pv:1Ineq}-\eqref{Pv:1Eq}. This interest of this part is two fold: first, it allows to make the change of regime of the functional $J_\mu$, from concave to convex, much more precise and, second, as the problem is linearised, this allows to gain a full characterisation of certain optimal configurations; this will be used at length in the section devoted to the analysis of Nash equilibria in two player games.

It should be noted that this approach is natural in the context of the spatial ecology: as the intricate nature of the problems at hand makes them hard to solve explicitly, it is hoped that such large diffusivity limits may provide meaningful simplifications of the problem at hand. For instance, we refer to \cite{HN,HENIII,HeNi,Mazari2020}, where such asymptotic regimes are used to tackle both the optimisation of the total population size and the study of stability of certain equilibria in Lotka-Volterra systems.

Recall from \cite{HN,HENIII,HeNi,Mazari2020} that uniformly in $\alpha \in \mathcal M_{\leq}{(\kappa,V_0)}(\O)$ there holds, in the $W^{1,2}(\O)$ sense, 
\begin{equation}\label{DefMA} \tamu =\underbrace{{\left(K_0-\fint_\O \alpha\right)}}_{=:M_\alpha}+\frac{v_{\alpha}}\mu+\underset{\mu \to \infty}O\left(\frac1{\mu^2}\right)\text{ where } \begin{cases}
-\Delta v_\alpha-M_\alpha\left(K-\alpha-M_\alpha\right)=0&\text{ in }\O\,, 
\\ \frac{\partial v_\alpha}{\partial \nu}=0\,, 
\\ \fint_\O v_\alpha=\frac1{M_\alpha^2}\fint_\O |\n v_\alpha|^2.
\end{cases}
\end{equation}
Also note that as we wish to investigate the monotonicity of the functional with respect to $\alpha$ in order to analyse whether or not the two formulations \eqref{Pv:1Ineq} and \eqref{Pv:1Eq} are equivalent, we keep $\fint_\O \alpha$ and do not replace it with $V_0$.

In particular we can already see the influence of the total fishing capacity on the first order asymptotic expansion of the functional: as in \cite{Mazari2020}, we obtain, uniformly in $\alpha \in \mathcal M_{\leq}{(\kappa,V_0)}(\O)$, the expansion
\[ J_\mu(\alpha)=J^0(\alpha)+\left(\frac1{\mu}\right)\text{ where }J^0:\alpha\mapsto \left(\fint_\O \alpha\right)\left(K_0-\fint_\O \alpha\right),
\]
and it is natural to invest the two asymptotic problems
\begin{equation}\label{Pv:1IneqLarge}\tag{$\bold{P}_{\leq,\mathrm{single},\mu\to\infty,0}$} \fbox{$\displaystyle \sup_{\alpha \in \mathcal M_\leq (\kappa,V_0)} J^0(\alpha)$}\end{equation} and 

\begin{equation}\label{Pv:1EqLarge}\tag{$\bold{P}_{=,\mathrm{single},\mu \to\infty,0}$} \fbox{$\displaystyle \sup_{\alpha \in \mathcal M_= (\kappa,V_0)} J^0(\alpha)$}\end{equation} 
Of course the particularly simple shape of the limit functional $J^0$ makes it amenable to an easy analysis and we have the following Proposition:
\begin{proposition}\label{Pr:Large0}
\begin{enumerate}
\item If $V_0<\frac{K_0}2$ then 
\[ \sup_{\alpha \in \mathcal M_= (\kappa,V_0)} J^0(\alpha)= \sup_{\alpha \in \mathcal M_\leq(\kappa,V_0)} J^0(\alpha).\] In particular the two problems \eqref{Pv:1IneqLarge} and \eqref{Pv:1EqLarge} coincide.
\item If $V_0>\frac{K_0}2$ then 
\[ \sup_{\alpha \in \mathcal M_= (\kappa,V_0)} J^0(\alpha)<\sup_{\alpha \in \mathcal M_\leq(\kappa,V_0)} J^0(\alpha).\]  In particular the two problems \eqref{Pv:1IneqLarge} and \eqref{Pv:1EqLarge} do not coincide.
\end{enumerate}
\end{proposition}
The content of this proposition is  that at the first order the asymptotic expansion of the functional selects an optimal fishing ability. However, it characterises neither its pointwise nor its geometric properties. This information is carried by the next order of this asymptotic expansion, and we will only work with equality constraints. To make this more precise we define the functional 
\[J^1:\alpha \mapsto \fint_\O \alpha v_\alpha\text{ where }\begin{cases}
-\Delta v_\alpha-M_\alpha\left(K-\alpha-M_\alpha\right)=0&\text{ in }\O\,, 
\\ \frac{\partial v_\alpha}{\partial \nu}=0\,, 
\\ \fint_\O v_\alpha=\frac1{M_\alpha^2}\fint_\O |\n v_\alpha|^2.
\end{cases}\] where we recall that $M_\alpha$ is defined in \eqref{DefMA} and, similarly to \cite{Mazari2020}, we obtain, uniformly in $\alpha \in \mathcal M_\leq(\kappa,V_0)$, 
\[ J_\mu(\alpha)=J^0(\alpha)+\frac{J^1(\alpha)}{\mu}+\underset{\mu \to \infty}O \left(\frac1{\mu^2}\right)\] so that the next order optimisation problem is 

\begin{equation}\label{Pv:1EqLarge2}\tag{$\bold{P}_{=,\mathrm{single},\mu \to\infty,1}$} \fbox{$\displaystyle \sup_{\alpha \in \mathcal M_= (\kappa,V_0)} J^1(\alpha)$.}\end{equation} 
We have a fairly good understanding of this optimisation problem, as showcased by the following theorem:

\begin{theorem}\label{Th:AsymptoticSingle}
We have the following  results:
\begin{enumerate}
\item\underline{Concavity for low fishing abilities:} if $V_0<\frac{K_0}2$, the functional $J^1$ is strictly concave on $\mathcal M_= (\kappa,V_0)$.
\item \underline{Convexity for large fishing abilities:} if $V_0>\frac{K_0}2$, the functional $J^1$ is strictly convex on $\mathcal M_= (\kappa,V_0)$. Consequently the solutions of \eqref{Pv:1EqLarge2} are bang-bang functions.
\item \underline{Characterisation in one dimension:} if $\O=(0;1)$, if $V_0>\frac{K_0}2$, if $K$ is non-increasing and non constant, the optimal fishing strategy $\alpha^*$ is equal to
\[ \alpha^*=\kappa \mathds 1_{[1-\ell;1]}\text{ with }\ell\kappa=V_0.\]
\end{enumerate}
\end{theorem}
The proof of this theorem relies on a rewriting of the functional $J^1$ and, for the characterisation of optimisers in the one-dimensional case, we use Talenti inequalities \cite{Langford}.

With the elements that will be used in the proof of Theorem \ref{Th:AsymptoticSingle} we also derive the following result that shows the particular role of the volume constraint $V_0=\frac{K_0}3$. Before we state it, let us simply recall that a critical point of $J^1$ is simply a fishing strategy $\alpha \in \mathcal M_=(\kappa,V_0)$ such that the Gateaux-derivative of $J^1$ at $\alpha$ in any admissible direction is zero. 
\begin{proposition}\label{Pr:K0/3}
Consider the constant fishing strategy $\overline\alpha\equiv V_0$. Then $\overline \alpha$ is a critical point of $J^1$ on $\mathcal M_=(\kappa,V_0)$  if, and only if, one of the following is satisfied:
\begin{center} Either $K$ is constant or $V_0=\frac{ K_0}3$.\end{center} In particular, for any $K$, if $V_0=\frac{K_0}3$, the only solution of \eqref{Pv:1EqLarge2} is $\overline \alpha$.
\end{proposition}

\subsection{Qualitative analysis of Nash equilibria for two-player games: general diffusivities}
In this section, we present the second facet of the fishing problems we laid out in the introduction, namely, the problem of existence and equilibria of Nash equilibria for multiple player games. For the sake of simplicity, we will only work on two-players games.
\paragraph{Set-up and definitions}
We consider two players; the first player plays a fishing strategy $\alpha_1$ and the second player uses a fishing strategy $\alpha_2$. We assume that the fish population still accesses resources modelled by the function $K:\O\to \R$, with $K_0=\fint_\O K$, and that there exists constant $\kappa_i\,, V_i$ ($i=1,2$) such that 
\[ \text{ for }i=1,2\,, \alpha_i\in \mathcal M_= (\kappa_i,V_i).\] Let us note that here we work with equality constraints. We refer to Remark \ref{Re:NashMonot} for additional comments about the constraints but simply note here that this simplifies our presentation.
If we assume that 
\[ V_1+V_2< K_0\] then we can define $\taa$ as the unique solution of 
\begin{equation}\label{Eq:ELDNash}\begin{cases}
-\mu \Delta \taa-\taa(K-\alpha_1-\alpha_2-\taa)=0&\text{ in }\O\,, 
\\ \frac{\partial \taa}{\partial \nu}=0&\text{ on }\partial \O\,, 
\\ \taa\geq 0\,, \taa\neq 0.\end{cases}\end{equation}
It should be noted that throughout this section we once again changed the subscript defining the solution $\taa$ in order to emphasise that our optimisation variables are $\alpha_1\,, \alpha_2$. 

For the $i$-th player ($i=1,2$) the fishing output is given by the functional 
\[I_{i,\mu}:(\alpha_1,\alpha_2)\mapsto \fint_\O \alpha_i\taa.\]
Each player wants to maximise its fishing outcome, so that we are typically in a situation where we want to investigate the existence of \emph{Nash equilibria}, defined as follows:

\begin{definition}\label{De:Nash}A Nash equilibrium for our two-players game is a couple of fishing strategies $(\alpha_1^*,\alpha_2^*)$ such that 
\[\begin{cases} I_{1,\mu}(\alpha_1^*,\alpha_2^*)=\displaystyle\max_{\alpha_1\in \mathcal M_{=}(\kappa_1,V_1)}I_{1,\mu}(\alpha_1,\alpha_2^*)\,,\\
I_{2,\mu}(\alpha_1^*,\alpha_2^*)=\displaystyle\max_{\alpha_2\in \mathcal M_{=}(\kappa_2,V_2)}I_{2,\mu}(\alpha_1^*,\alpha_2).\end{cases}
\]
\end{definition}

\begin{remark}\label{Re:NashMonot}[Equality vs. Inequality constraints]
Of course, the same type of results that we obtained in the single player case  (Theorem \ref{Th:MonotonicityTrue}) could be derived in the case of multiple players games when considering the influence of an (in)equality constraint in the set of admissible fishing strategies; as the results would be extremely similar, we do not detail the influence of an inequality constraint and this is why we work with an equality constraint.
\end{remark}
Our main research question here is:
\begin{center} Does there exist a Nash equilibrium for the two-players game described above?\end{center}Let us note here that, in general, establishing the existence of Nash equilibria is a delicate matter, that can usually be achieved using concavity or convexity properties of the functionals at hand \cite{Gonz_lez_D_az_2010}.      

Our first theorem shows that whenever the fishing abilities of both players are small enough, a Nash equilibrium exists.
\begin{theorem}\label{Th:NashExists}
In the one-dimensional case $\O=(0;1)$, the constants $\kappa_1\,, \kappa_2\,, \mu$ being fixed, there exists $\delta=\delta(\kappa_1,\kappa_2,\mu)>0$ such that, if 
\[ V_1+V_2\leq \delta\] there exists a Nash equilibrium.

In any dimension $d$, the constants $\kappa_1\,, \kappa_2\,, \mu$ being fixed, there exists $\delta_1>0\,, \delta_2>0$ such that, if 
\[ V_1+V_2\leq \delta_1\,, \Vert K-\overline K\Vert_{L^1(\O)}\leq \delta_2\] there exists a Nash equilibrium.
\end{theorem}

Of course this result is linked to Theorem \ref{Th:Concave} above, as, since the seminal paper \cite{Nash1951}, the concavity of the cost functionals is known to be of paramount importance to obtain the existence of equilibria. Nonetheless, the proof is not immediate.
%

The standard results do not enable us to obtain the existence of a Nash equilibrium when, on the other hand, $V_1+V_2$ is close to $K_0$, and we can not conclude in the general case. We can, however, pursuing our investigation of asymptotic regimes, show that, even in this case, in which the cost functionals can behave, from the point of view of optimal control, as convex functions, (see Theorem \ref{Th:Convex} above), 
there exists a Nash equilibrium when the diffusivity $\mu$ is large enough.
\subsection{Existence of Nash equilibria when the cost functionals are convex: asymptotic analysis}
Our final result deals with a slightly more complicated case, that of convex functionals. Here, we provide a result for the asymptotic expansion of the fishing functionals, and in the case where $K\equiv \overline K$ is constant. This problem corresponds to taking the limit $\mu\to \infty$. Following the analysis that was succinctly presented when introducing the problem \eqref{Pv:1EqLarge2} we define the two limiting functionals (in what follows, $M=K_0-V_1-V_2$)
\[ I_i^1:\mathcal M_=(\kappa_i,V_i)\ni \alpha_i\mapsto\fint_\O \alpha_i v_{\alpha_1,\alpha_2}\text{ where }v_{\alpha_1,\alpha_2} solves \begin{cases}-\Delta v_{\alpha_1,\alpha_2}-M(K_0-\alpha_1-\alpha_2-M)=0&\text{ in }\O\,, 
\\ \frac{\partial v_{\alpha_1,\alpha_2}}{\partial \nu}=0&\text{ on }\partial \O\,, 
\\ \fint_\O v_{\alpha_1,\alpha_2}=\frac1{M^2}\fint_\O |\n v_{\alpha_1,\alpha_2}|^2.
\end{cases}\] 
An asymptotic Nash equilibrium is then defined as follows:
\begin{definition}\label{De:NashAsymptotic}An asymptotic Nash equilibrium for our two-players game is a couple of fishing strategies $(\alpha_1^*,\alpha_2^*)$ such that 
\[\begin{cases} I_{1}^1(\alpha_1^*,\alpha_2^*)=\max_{\alpha_1\in \mathcal M_{=}(\kappa_1,V_1)}I_{1}^1(\alpha_1,\alpha_2^*)\,,\\
I_{2}^1(\alpha_1^*,\alpha_2^*)=\max_{\alpha_2\in \mathcal M_{=}(\kappa_2,V_2)}I_{2}^1(\alpha_1^*,\alpha_2).\end{cases}
\]
\end{definition}

\begin{theorem}\label{Th:NashLarge}
Assume $V_1\,, V_2>\frac{K_0}4$, assume $K$ is constant and let 
\[ \alpha_i^*=\kappa_i\mathds 1_{[0;\ell_i]} \text{ with }\kappa_i\ell_i=V_i \text{ ($i=1,2$)}.\] 
$(\alpha_1^*\,, \alpha_2^*)$ is a Nash equilibrium in the sense of Definition \ref{De:NashAsymptotic}.
\end{theorem}

\paragraph{Regarding "the price of anarchy" and the uniqueness of Nash equilibria}
We conclude with two remarks about Theorem \ref{Th:NashLarge}.
First, regarding the uniqueness of Nash equilibria, we can conclude that it does not hold in general. Indeed, consider the conclusion of Theorem \ref{Th:NashLarge} and then compare it with the following analysis: if we assume that 
\[ K_0=1\,, V_1=V_2=\frac{1}3\,, \kappa_1=\kappa_2=1\] and if we let 
\[ \overline \alpha_1=\overline\alpha_2\equiv \frac13\] then it is readily checked that $(\overline \alpha_1,\overline\alpha_2)$ is also a Nash equilibrium: indeed, this follows from the consideration of Remark \ref{Re:KConstant} and the fact that with these definitions we have $\overline \alpha_1=\frac{K_0-\overline\alpha_2}2$ whence the conclusion. We are thus left with two different Nash equilibria, the one given by Theorem \ref{Th:NashLarge} and the constant one $(\overline \alpha_1,\overline \alpha_2)$. In particular, we can not expect the uniqueness of Nash equilibria to hold.

 Second, we can use this particular example to illustrate a concept known, in economics, as the "price of anarchy". As we sketched briefly in the introduction to our paper, the price of anarchy quantifies the insufficiency of selfish strategies when compared to cooperative strategies. In other words, is it true that, in general, the two players would be better off collaborating and then sharing the common fishing output rather than competing in a selfish manner? Consider once again the Nash equilibrium $(\overline \alpha_1,\overline \alpha_2)$ and now assume that, instead of competing against each other, the two players united their strength, and decided to solve 
 \begin{equation}\label{Eq:Us} \max_{\alpha \in \mathcal M_\leq(\frac12,\frac23)}J^1(\alpha).\end{equation}
 In the end they would simply split the total fishing outcome associated with an optimal strategy $\alpha^*$. However, from Remark \ref{Re:KConstant}, the unique solution of \eqref{Eq:Us} is $\overline \alpha=\frac12\neq \alpha_1^*+\alpha_2^*$. Thus, 
 \[ I_1^1(\overline\alpha_1 ,\overline\alpha_2 )+I_2^1(\overline\alpha_1 ,\overline\alpha_2 )<J^1(\overline \alpha):\] the total fishing output is worse than if the players had convened a strategy before playing. 
 
%

\paragraph{Competition and cooperation: a drastic example of the "tragedy of commons" situation}
We can actually prove something stronger when the number of players goes to $\infty$.
If we consider a game between $n$ players, we can construct a sequence of Nash equilibria when $K$ is a constant.  Assuming that $K\equiv 1$, we can adapt the arguments above to observe that the configuration where all players have the same strategy, namely when
$$\forall i\in \{1,\dots n\}\,, \alpha_i^*:=\alpha^*=\frac{1}{n+1},$$ then $(\alpha_i^*)_{i=1,\dots,n}$ is a Nash equilibrium. Defining $\overline \alpha:=(\alpha_i^*)_{i=1,\dots,n}$, the associated steady state is
$$\theta_{\overline\alpha }=1-\frac{n}{n+1}\underset{n\to +\infty}{\longrightarrow} 0 .$$
We hence conclude that  all $\mu>0$, there exist a Nash equilibria $\overrightarrow{\alpha^*}=(\alpha_1^*,\alpha_2^*,...,\alpha_n^*)$ such that:
 \begin{equation*}
 \frac{1}{4}=\max_{\overrightarrow{\alpha}}\int_\Omega \left(\sum_{i=1}^n\alpha_i\right)\theta_{\overrightarrow{\alpha}}dx > \int_\Omega \left(\sum_{i=1}^n\alpha_i^*\right)\theta_{{\overrightarrow \alpha^*}}dx \underset{n\to +\infty}{\longrightarrow} 0
 \end{equation*}
 In particular, for this sequence of Nash equilibrium, \emph{the total harvested amount goes to zero as the number of player goes to $\infty$}: cooperation would have been better than competition. For further discussion of this concept of "price of anarchy", we refer to \cite{Johari,Roughgarden}.

\subsection{Bibliographical references}\label{Se:Bib}
As there are several bodies of literature the present work fits in, we split the detailed presentation of our references accordingly.
\paragraph{Optimisation problem in spatial ecology}
Over the past two decades, a wide range of efforts have been devoted to provide a better mathematical understanding of spatially heterogeneous phenomena. Indeed, after the pioneering works of Fisher, Kolmogorov, Petrovski and Piskunoff \cite{Fisher,KPP}, a wide body of literature was produced in an attempt to grasp fine propagation or invasion phenomena but, more recently, a new line of research has emerged that strongly emphasises the influence of heterogeneous reaction terms. After the works of Shigesada and Kawasaki, which provided a first qualitative insight into the influence of the geometry of environments \cite{ShigesadaKawaski} on the survival of populations, and several results of Cantrell and Cosner \cite{CantrellCosner1,31806255e7f648d5b65ff02c30c4f539,CantrellCosner,MR1105497},  \emph{optimising the spatial heterogeneity} became a fruitful point of view. In other words: which is the optimal heterogeneity from the point of population dynamics? Of course, we need to specify which criteria are considered when using the word "optimal", but let us point out that this way of looking at the question brought forth combinations of PDE or ODE techniques and of optimal control theory.  Let us also, on the topic of optimal control of biological models, point to the monograph \cite{lenhart2007optimal}. A typical instance of optimal control problem of the type under study in the present paper is that of the optimal survival ability. A spectral optimisation problem, it has sparked a wealth of scientific articles devoted to its understanding and is by now fairly well understood \cite{BHR,31806255e7f648d5b65ff02c30c4f539,henrot2006,KaoLouYanagida,LamboleyLaurainNadinPrivat,LouYanagida,MazariNadinPrivat}. Let us point out that, in studying this problem, \cite{BHR} features what is, to the best of our knowledge, the first use of rearrangement techniques and isoperimetric inequalities to spatial ecology problems.

More recently, a new question that has drawn a lot of attention from the mathematical community is that of the optimisation of the total population size. In other words, how should we spread resources in logistic-diffusive models in order to maximise the total population size? Originating in the works of Lou \cite{LouInfluence,Lou2008} this question was then explored in details in a series of works \cite{heo2021ratio,InoueKuto,LiangLou,LiangZhang,LouNagaharaYanagida,MRBSIAP,Mazari2020,Mazari2021,NagaharaYanagida}. Of particular relevance in the context of the total population size was the \emph{bang-bang property}: are optimisers for the total population size bang-bang functions? After several partial results \cite{Mazari2020,NagaharaYanagida} the answer was proved to be yes in \cite{Mazari2021}. It should be noted that in the proof of Theorem \ref{Th:Convex} we build on the techniques of \cite{Mazari2021} to prove a bang-bang property for optimal fishing strategies.

\paragraph{Optimal fishing problems} Of course, all the problems we described in the previous paragraphs describe, in a way, "nice" problems, in the following sense: since we are trying to optimise a criterion with respect to resources, it is expected that adding resources will prove beneficial. One of the conclusion of \cite{Mazari2021} is indeed that, for monotonic bilinear functionals (\emph{i.e.} that increasing the resources increases the criterion) the bang-bang property holds. However, the case under study in this paper is quite different since, as we already touched upon, the problem of overfishing makes it so that the functional we are considering is no longer monotonic: it makes no sense to fish as much as we can for we may risk killing all the population. In that regard, our paper can be seen as a first detailed analysis of an optimal control problem for spatially heterogeneous fishing problems.

Of course, several authors have considered many different aspects of optimal fishing problems before. While it is impossible to list all these contributions here, let us single out \cite{Cooke1986}, where a survey of the early works (e.g. one-dimensional harvesting models, stochastic harvesting models...) is           presented and \cite{Braverman2009} where several types of models are considered, including the logistic diffusive models, but where the diffusion operator would be (if we were to adopt our notations) $\Delta\left(\frac\cdot\alpha\right)$, which changes the qualitative behaviour of the optimisation problem dramatically. Notably, it is not possible to lift their results to the case of non-regular fishing strategies $\alpha$ (that may be discontinuous for instance).

\paragraph{Nash equilibria in optimal control theory}
Several recent contributions deal with the existence and computation of Nash equilibria in optimal control theory. Let us single out two of these works, namely, \cite{Carvalho_2018,Fern_ndez_Cara_2021} . In these works, the functionals one seeks Nash equilibria for are of tracking-type (in the sense that we seek to minimise the distance to certain objective functions) and, very importantly, consider linearly controlled PDEs with $L^2$ penalisations of the constraints. This changes the features of the optimisation problem drastically. In \cite{Cal_Campana_2020} on the other hand, the question of existence and computation of Nash equilibria in bilinear problems, but for ODE models. Our paper is, to the best of our knowledge, a first contribution to the qualitative analysis of $L^\infty-L^1$ constrained bilinear optimal control problems with a cost function that is not of tracking-type.

\subsection{Plan of the paper}
The proofs of the theorems of the paper are grouped by the tools  used in their proof. In section \ref{Se:1} we give the proof of Theorems \ref{Th:MonotonicityTrue}, \ref{Th:MonotonicityFalse}, \ref{Th:Concave} and \ref{Th:Convex} as they all rely strongly on the computation of first and second-order Gateaux derivatives of the functional. In section \ref{Se:2}, the proofs of the asymptotic behaviours described in Theorems \ref{Th:AsymptoticSingle}  are presented. Finally, we gatherd in section \ref{Se:3} the proofs of those results dealing with multiple player games, Theorems \ref{Th:NashExists} and \ref{Th:NashLarge}.

\section{Proofs of Theorems \ref{Th:MonotonicityTrue}, \ref{Th:MonotonicityFalse}, \ref{Th:Concave} and \ref{Th:Convex}}\label{Se:1}
\paragraph{Notational simplifications}Throughout this section we investigate the influence of the range of the parameter $V_0$ on the optimisation problems \eqref{Pv:1Ineq}-\eqref{Pv:1Eq}, and we thus drop the diffusivity $\mu$ from all subscripts. Henceforth, $\theta_\alpha$ denotes the solution of \eqref{Eq:MainSingle}, and we set 
\[ J:\alpha \mapsto \fint_\O \alpha \theta_\alpha.\]
Furthermore, the proofs of the three theorems under scrutiny derive from the computations of the first and second order Gateaux-derivatives of the functional $J$. We recall that, for an admissible fishing strategy $\alpha$, an admissible perturbation $h$ at $\alpha$ is a function $h\in L^2(\O)$ such that there exists two sequences $\{h_n\}_{n\in\N}\in L^2(\O)^\N$, $\{\e_n\}_{n\in \N}\in \left(\R_+\backslash\{0\}\right)^\N$ satisfying:
\[ \e_n\underset{n\to \infty}\rightarrow 0\,, h_n\underset{n\to \infty}\rightarrow h \text{ in }L^2(\O)\text{ and, for any } n\in \N\,, \alpha+\e_n h_n\text{ is admissible.}\]

Whenever $\fint_\O \alpha<K_0$, we can adapt in a straightforward manner the proof of \cite[Lemma 4.1]{Ding2010} and prove that the functional $J$ and the map $\alpha\mapsto \theta_\alpha$ are  twice Gateaux-differentiable. In the first part of this section we give these Gateaux-derivatives in expanded form, and analyse their specific features when proving our results.

\subsection{Computations of the first and second-order Gateaux-derivatives of the functional $J_\mu$}
We fix an admissible fishing strategy $\alpha$ and an admissible perturbation $h$ at $\alpha$. At this point, since we do not specify in which admissible set we work, an admissible perturbation is any $h\in L^2(\O)$. From the computations of \cite[Lemma 4.1]{Ding2010}, the first, respectively  second, order Gateaux-derivative of $\alpha \mapsto \theta_\alpha$ at $\alpha$ in the direction is the unique solution $\dot \theta_\alpha$ of the equation 
\begin{equation}\label{Eq:DotSingle}
\begin{cases}
-\mu \Delta \dot \theta_\alpha-\dot \theta_\alpha \left(K-\alpha-2\theta_\alpha\right)=-h  \theta_\alpha&\text{ in }\O\,, 
\\ \frac{\partial \dot \theta_\alpha}{\partial \nu}=0,\end{cases}\end{equation} respectively the unique solution $\ddot\theta_\alpha$ of the equation 
\begin{equation}\label{Eq:DdotSingle}
\begin{cases}
-\mu \Delta \ddot \theta_\alpha-\ddot\theta_\alpha \left(K-\alpha-2\theta_\alpha\right)=-2h\dot \theta_\alpha-2\dot \theta_\alpha^2&\text{ in }\O\,, 
\\ \frac{\partial \ddot \theta_\alpha}{\partial \nu}=0&\text{ on } \partial \O.\end{cases}\end{equation}
\begin{remark}\label{Re:ExistenceFredholm}
Existence and uniqueness of solutions of \eqref{Eq:DotSingle}-\eqref{Eq:DdotSingle} follow from the following crucial observation \cite[Comment after eq. (2.6)]{LouInfluence}: from \eqref{Eq:MainSingle}, the first eigenvalue of the operator $-\mu\Delta-(K-\alpha-\theta_\alpha)$ is zero. From the monotonicity of the eigenvalue, the first eigenvalue of $-\mu \Delta-(K-\alpha-2\theta_\alpha)$ is positive. The existence and uniqueness of solutions to \eqref{Eq:DotSingle}-\eqref{Eq:DdotSingle} then follow from a standard variational argument.
\end{remark}
Similarly the first, respectively second, order Gateaux-derivative of the map $J$ at $\alpha$ in the direction $h$ is given by the expression 
\begin{equation}\label{Eq:DotJ1}
\dot J(\alpha)[h]=\fint_\O h\theta_\alpha+\fint_\O \alpha \dot \theta_\alpha,
\end{equation}
respectively by 
\begin{equation}\label{Eq:DdotJ1}
\ddot J(\alpha)[h,h]=2\fint_\O h\dot\theta_\alpha+\fint_\O \alpha \ddot \theta_\alpha.\end{equation}
We need to introduce an adjoint state in order to give equations \eqref{Eq:DotJ1}-\eqref{Eq:DdotJ1}  tractable expressions.

We introduce the unique solution $p_\alpha$ of 
\begin{equation}\label{Eq:Adjoint}
\begin{cases}
-\mu \Delta p_\alpha -p_\alpha(K-\alpha-2\theta_\alpha)=\alpha&\text{ in }\O
\\ \frac{\partial p_\alpha}{\partial \nu}=0&\text{ on }\partial \O.\end{cases}\end{equation}
The following properties are obtained by adapting the reasoning of \cite[Lemma 13]{Mazari2021}, which simply relies on the aforementioned Remark \ref{Re:ExistenceFredholm} that the first eigenvalue of $-\mu \Delta-(K-\alpha-2\theta_\alpha)$ is positive:
\begin{lemma}\label{Le:AdjEx}
There exists a unique solution $p_\alpha$ of \eqref{Eq:Adjoint}. Furthermore, if $\alpha\geq 0\,, \alpha\neq 0$, 
\[ \inf_{\overline \O}p_\alpha>0.\]
\end{lemma}
Now, if we multiply \eqref{Eq:Adjoint} by $\dot\theta_\alpha$ and integrate by parts, and, similarly, multiply \eqref{Eq:DotSingle} by $p_\alpha$ and integrate by parts we derive the equality 
\begin{align*}
\fint_\O \alpha \dot \theta_\alpha&=\mu\fint_\O \langle \n \dot \theta_\alpha\,, \n p_\alpha\rangle-\fint_\O p_\alpha \dot \theta_\alpha \left(K-\alpha-2\theta_\alpha\right)
\\&=-\fint_\O hp_\alpha \theta_\alpha,
\end{align*}
so that 
\[ \dot J(\alpha)[h]=\fint_\O (1-p_\alpha)\theta_\alpha h.\]
Similarly, we obtain 
\[\frac12 \ddot J(\alpha)[h]=\fint_\O  h\dot \theta_\alpha-\fint_\O p_\alpha h\dot \theta_\alpha-\fint_\O p_\alpha \dot \theta_\alpha^2.\] We have thus proved the following lemma:
\begin{lemma}\label{Le:GD}
The first and second order Gateaux-derivative of the functional admit the following expressions:
\begin{equation}\label{Eq:DotJ}
\dot J(\alpha)[h]=\fint_\O (1-p_\alpha)\theta_\alpha h,\end{equation} and 
\begin{equation}\label{Eq:DdotJ}
\frac12\ddot J(\alpha)[h,h]=\fint_\O (1-p_\alpha)h\dot \theta_\alpha-\fint_\O p_\alpha \dot\theta_ \alpha^2.
\end{equation}
\end{lemma}

\subsection{Computation of the first and second-order Gateaux-derivatives of $\overline J_{\delta,\mu}$}
We can adapt the proofs of the previous section to $\overline J_{\delta,\mu}$. Similar to the notational conventions we adopted above, we now denote by $\theta_{K+\delta m}$ the solution of \eqref{Eq:MainSingle} with $\alpha=K+\delta m$. We define, for any $m \in \mathcal N(\O)$, the unique solution $q_{\delta,m}$ of 
\begin{equation}\label{Eq:AdjointLargeFishing}
\begin{cases}
-\mu \Delta q_{\delta,m}-q_{\delta,m}\left(-\delta m-2\theta_{K+\delta m}\right)=K+\delta m&\text{ in }\O\,, 
\\ \frac{\partial q_m}{\partial \nu}=0&\text{ on }\partial \O.
\end{cases}\end{equation}
Similarly to Lemma \ref{Le:GD} we obtain the following expression:
\begin{lemma}\label{Le:GDLargeFishing}
The first and second order Gateaux-derivative of the functional admit the following expressions:
\begin{equation}\label{Eq:DotJ}
\dot{ \overline J}_{\delta,\mu}(m)[h]=\delta\fint_\O (1-q_{\delta,m})\theta_{K+\delta m} h,\end{equation} and 
\begin{equation}\label{Eq:DdotJ}
\frac12\ddot{ \overline J}_{\delta,\mu}(m)[h,h]=\delta\fint_\O (1-q_{\delta,m})h\dot \theta_{K+\delta m}-\fint_\O q_{\delta,m} \dot\theta_ m^2,
\end{equation}
where $\dot \theta_{K+\delta m}$ satisfies
\begin{equation}
\begin{cases}
-\mu \Delta \dot \theta_{K+\delta m}-\dot \theta_{K+\delta m}\left(-\delta m-2\theta_{K+\delta m}\right)=-\delta h \theta_{K+\delta m}&\text{ in }\O\,, 
\\ \frac{\partial \dot\theta_{K+\delta m}}{\partial \nu}=0&\text{ on }\partial \O.\end{cases}\end{equation}

\end{lemma}

\subsection{Proof of Theorems \ref{Th:MonotonicityTrue}-\ref{Th:MonotonicityFalse}: monotonicity of the functionals}

\paragraph{Plan of the proofs}
We recall that monotonicity, for instance for $J_\mu$, means that 
\[ \forall \alpha_1\,, \alpha_2\in \mathcal M_\leq(\kappa,V_0)\,,\alpha_1\leq \alpha_2\text{ a.e. }\Rightarrow J_\mu(\alpha_1)\leq J_\mu(\alpha_2).\]
However, by the mean value theorem, we know that, for any $\alpha_1\,, \alpha_2\in \mathcal M_\leq(\kappa,V_0)$, there exists $\xi\in [0;1]$ such that 
\begin{equation} J_\mu(\alpha_2)-J_\mu(\alpha_1)=\dot J_\mu(\alpha_1+\xi (\alpha_2-\alpha_1))[\alpha_2-\alpha_1]\end{equation} and thus Lemma \ref{Le:GD} yields the existence of $\xi\in [0;1]$ such that  \begin{equation} \label{TAF}J_\mu(\alpha_2)-J_\mu(\alpha_1)=\fint_{\O} (1-p_{\alpha_1+\xi (\alpha_2-\alpha_1)})\theta_{\alpha_1+\xi(\alpha_2-\alpha_1)}(\alpha_2-\alpha_1).
\end{equation}
We can thus read the monotonicity of $J_\mu$ on \eqref{TAF}: if $\alpha_1\leq \alpha_2$ almost everywhere, and since $\theta_{\alpha_1+\xi(\alpha_2-\alpha_1)}$ is positive on $\overline \O$ for any $\xi \in [0;1]$, obtaining the monotonicity of the functional boils down to deriving the sign of $1-p_{\alpha_1+\xi(\alpha_2-\alpha_1)}$. Thus the proof of Theorem \ref{Th:MonotonicityTrue} is simply to show that under certain volume constraints we have $p_\alpha<1$.

Similarly, using Lemma \ref{Le:GDLargeFishing}, to show the non-monotonicity of $\overline J_{\delta,\mu}$ it suffices to prove that, for $\delta>0$ small enough, $q_{\delta,m}>1$ for any $m\in \mathcal N_=(\O)$. This will imply that the optimal values of the two problems \eqref{Pv:1IneqLargeFishing}-\eqref{Pv:1EqLargeFishing} differ.

\begin{proof}[Proof of Theorem \ref{Th:MonotonicityTrue}]
Following the general idea explained in the plan of the proof, it suffices to prove the following lemma:
\begin{proposition}\label{Pr:LFA1}
The two constants $\kappa,\mu$ being fixed, we have the following property: for any $\delta>0$,  there exists $\e_1=\e_1(\O,\kappa,\mu,\delta)>0$ such that, for any $V_0\in (0;\e_1)$, for any $\alpha \in \mathcal M_\leq(\kappa,V_0)$, we have 
\[0<\inf_{\overline \O}p_\alpha\leq \sup_{\overline \O} p_\alpha\leq 1-\delta.\]
\end{proposition}
With this result at hand it is easy to obtain the monotonicity property: fixing $\delta>0$ and choosing the $\e_1$ given by Proposition \ref{Pr:LFA1}, we obtain 
\[\forall V_0\in (0;\e_1)\,,  \forall \alpha \in \mathcal M_\leq (\kappa,V_0)\,,\inf_{\overline \O} \left(\Psi_\alpha:=(1-p_\alpha)\theta_\alpha\right)\geq \delta \inf_{\overline \O}\theta_\alpha>0\] so that \eqref{TAF} implies the conclusion: with $\delta=\frac12$, the functional is monotonically increasing if $V_0<\e_1$ where $\e_1$ is given by Proposition \ref{Pr:LFA1}.

This  implies that any solution $\alpha^*$ of \eqref{Pv:1Ineq} satisfies 
\[ \fint_\O \alpha^*=V_0\] and thus that $\alpha^*$ is a solution of \eqref{Pv:1Ineq}. Indeed, should we have $\fint_\O \alpha^*<V_0$ we simply take any positive function $h\in L^\infty(\O)$ such that $\alpha^*+h\in \mathcal M_=(\kappa,V_0)$. By monotonicity of the functional, 
\[ J_\mu(\alpha^*+h)>J_\mu(\alpha^*),\] a contradiction.

It remains to prove Proposition \ref{Pr:LFA1}.
\begin{proof}[Proof of Proposition \ref{Pr:LFA1}]
Let us note that, as $\kappa$ is fixed and as $K\in L^\infty(\O)$, a classical application of the maximum principle there holds
\begin{equation}\label{Eq:Estima} \forall \mu>0\,, \Vert \tamu\Vert_{L^\infty(\O)}\leq \Vert K\Vert_{L^\infty(\O)}+\Vert \alpha\Vert_{L^\infty(\O)}=:M_0.\end{equation}
We will prove that
\[\lim_{V_0\to 0^+}\sup_{\alpha \in \mathcal M_\leq(\kappa,V_0)}\Vert p_\alpha\Vert_{L^\infty(\O)}=0.\]
To control the $L^\infty$ norm of any $p_\alpha$, we need to use the first eigenvalue $\lambda(K-\alpha-2\theta_\alpha)$ of the operator 
\[ -\mu \Delta-(K-\alpha-2\theta_\alpha)\] endowed with Neumann boundary conditions. As this operator is symmetric, we know that 
\begin{equation}\label{Eq:Rayleigh1} \lambda(K-\alpha-2\theta_\alpha)=\inf_{u\in W^{1,2}(\O)\,, \fint_\O u^2=1}\left(\mu \fint_\O|\n u|^2-\fint_\O u^2\left(K-\alpha-2\theta_\alpha\right)\right).\end{equation}
As in \cite[Proof of Lemma 2.1]{LouInfluence} (see also Remark \ref{Re:ExistenceFredholm} above) we know that for any $V_0\in (0;K_0)$ and any $\alpha \in \mathcal M_\leq(\kappa,V_0)$ there holds
\[\lambda(K-\alpha-2\theta_\alpha)>0.\]To obtain uniform regularity estimates on $p_\alpha$ we need to obtain a uniform lower bound on $\lambda(K-\alpha-2\theta_\alpha)$ as $V_0\to 0^+$.
\begin{lemma}\label{Le:LFA1}
There exists $a_0,\e_0>0$ such that
\[ \forall V_0\in(0;\e_0)\,, \forall \alpha \in \mathcal M_\leq (\kappa,V_0)\,,\quad \lambda(K-\alpha-2\theta_\alpha)\geq a_0.\]
\end{lemma}

\begin{proof}[Proof of Lemma \ref{Le:LFA1}]
We observe that from \eqref{Eq:Estima} and standard $L^p$ elliptic regularity theory, for any $p\in [1;+\infty)$ there exists a constant $M_p=M_p(\mu,\Omega)>0$ such that uniformly in $V_0$ and uniformly in $\alpha \in \mathcal M_\leq (\kappa,V_0)$ there holds
\[ \Vert \theta_\alpha\Vert_{W^{2,p}(\O)}\leq M_p.\] Using Sobolev embeddings this implies that for any $s \in [0;1)$ there exists a constant $C_s$ such that 
uniformly in $V_0$ and uniformly in $\alpha \in \mathcal M_\leq (\kappa,V_0)$ there holds
\begin{equation}\label{Eq:Holder1}\Vert \theta_\alpha\Vert_{\mathscr C^{1,s}(\O)}\leq C_s.\end{equation}
It is expected that as $V_0\to 0$ we should have $\theta_\alpha \to \overline \theta$ where $\overline\theta$ is the solution of 
\begin{equation}\label{Eq:LFAEq}
\begin{cases}
-\mu \Delta \overline\theta-\overline\theta\left(K-\overline \theta\right)=0&\text{ in }\O\,, 
\\ \frac{\partial \overline \theta}{\partial \nu}=0&\text{ on }\partial \O\,, 
\\ \overline \theta\geq 0\,, \overline\theta\neq 0.\end{cases}\end{equation}
Let us show that this convergence is uniform in the following sense:
\begin{equation}\label{Eq:CvUni}
\lim_{V_0\to 0^+}\sup_{\alpha\in \mathcal M_\leq (\kappa,V_0)}\Vert \theta_\alpha-\overline \theta\Vert_{\mathscr C^0(\O)}=0.
\end{equation}
Argue by contradiction and assume there exists a sequence $\{V_k\}_{k\in \N}$, $c_1>0$ and, such that, for any $k\in \N$, there exists $\alpha_k\in \mathcal M_\leq(\kappa,V_k)$ such that 
\[ \Vert \theta_{\alpha_k}-\overline \theta\Vert_{\mathscr C^0(\O)}\geq c_1.\] From \eqref{Eq:Holder1}, we extract from $\{\theta_{\alpha_k}\}_{k\in \N}$ a $\mathscr C^1$ converging subsequence, still labeled $\{\theta_{\alpha_k}\}_{k\in \N}$ and its $\mathscr C^1$ limit $\theta_\infty$.
From \cite[Equation (2.4)]{LouInfluence} there exists a constant $c_0$ uniform in $V_0$ such that 
\[ \Vert \theta_{\alpha_k}-\overline\theta\Vert_{L^1(\O)}\leq c_0\Vert \alpha_k\Vert_{L^1(\O)}^{\frac13}.\]
We thus conclude that $\theta_\infty=\overline \theta$, a contradiction.

From this uniform convergence and the simplicity of the first eigenvalue $\lambda(K-\alpha-2\theta_\alpha)$, we  deduce that 
\begin{equation}\label{Eq:Ode} \lim_{V_0\to 0^+}\inf_{\alpha \in \mathcal M_\leq(\kappa,V_0)}\lambda(K-\alpha-2\theta_\alpha)=\lambda(K-2\overline \theta)>0\end{equation} where the last inequality comes from the aforementioned \cite[Proof of Lemma 2.1]{LouInfluence}. The proof of \eqref{Eq:Ode} is standard and we postpone it to appendix \ref{Ap:Ode}.

Lemma \ref{Le:LFA1} is proved.
\end{proof}
We can go back to the proof of Proposition \ref{Pr:LFA1}. We argue \emph{via} a standard bootstrap method, as follows: using $p_\alpha$ as a test function in \eqref{Eq:Adjoint} we obtain 
\[ \mu \fint_\O|\n p_\alpha|^2-\fint_\O p_\alpha^2\left(K-\alpha-2\theta_\alpha\right)\leq{\mathrm{Vol}(\O)^{-1}} \Vert \alpha\Vert_{L^2(\O)}\Vert p_\alpha\Vert_{L^2(\O)}\leq {\mathrm{Vol}(\O)^{-1}}\sqrt{\kappa V_0}\Vert p_\alpha\Vert_{L^2(\O)}.
\]
From the Rayleigh quotient formulation of eigenvalue \eqref{Eq:Rayleigh1} and the lower estimate of Lemma \ref{Le:LFA1} we deduce that 
\[ \Vert p_\alpha\Vert_{L^2(\O)}\leq \frac{ \sqrt{\kappa V_0}}{a_0\mathrm{Vol}(\O)},\] which in turn yields a uniform $W^{1,2}(\O)$ bound on the family $\{p_\alpha\}$. Using Sobolev embeddings, 
 the bootstrap method implies the following uniform bound: there exists $\e_{0}'>0$ such that, for any $p\in [1;+\infty)$, for any $V_0\in (0;\e_0')$, there exists $M_p$ such that for any $\alpha \in \mathcal M_\leq(\kappa,V_0)$, there holds
 \[ \Vert p_\alpha\Vert_{W^{2,p}(\O)}\leq M_p.\]
 It is then clear that for any sequence $\{V_k\}_{k\in \N}$ converging to zero and for any $\{\alpha_k\}_{k\in \N}\in \prod_{k\in \N} \mathcal M_\leq (\kappa,V_k)$, the sequence $\{p_{\alpha_k}\}_{k\in \N}$ converges in $\mathscr C^1(\O)$ to the solution $\overline p$ of 
 \[ \begin{cases}
 -\mu \Delta \overline p-\overline p(K-2\overline \theta)=0&\text{ in }\O\,, 
 \\ \frac{\partial \overline p}{\partial \nu}=0&\text{ on }\partial \O.\end{cases}\] As $\lambda(K-2\overline \theta)>0$, $\overline p=0$. Adapting the arguments of the proof of Lemma \ref{Le:LFA1}, it is easily shown that such convergence is uniform and that 
 \begin{equation}\label{Bloodborne} \lim_{V_0\to 0}\sup_{\alpha \in \mathcal M_\leq(\kappa,V_0)}\Vert p_\alpha\Vert_{\mathscr C^1(\O)}=0.\end{equation}
 The proof of Proposition \ref{Pr:LFA1} is finished.

\end{proof}
Thus, as we explained how Proposition \ref{Pr:LFA1} implied Theorem \ref{Th:MonotonicityTrue}, the proof of Theorem \ref{Th:MonotonicityTrue} is complete.

\end{proof}

\begin{proof}[Proof of Theorem \ref{Th:MonotonicityFalse}]

For large fishing abilities, on the contrary, we will prove that 
\[ \forall \eta>0\,,  \exists \delta_1>0\,, \forall 0<\delta<\delta_1 \,, \forall m \in \mathcal N_\leq(\O)\,, \inf_{\overline \O}q_{\delta,m}\geq \frac1\eta.\]
To do so, we need to investigate the asymptotic behaviour of $q_{\delta,m}$ as $\delta \to 0$. Given that $q_{\delta,m}$ solves \eqref{Eq:AdjointLargeFishing}, this requires a knowledge of the behaviour of $\theta_{K+\delta m}$ as $\delta \to0$. This is the object of the following proposition:
\begin{proposition}\label{Pr:LargeFA1}
Uniformly in $m\in \mathcal N(\O)$ the following asymptotic expansion holds in $\mathscr C^1(\O)$: 
\[ \theta_{K+\delta m}=\delta m_0+\underset{\delta \to 0}O(\delta^2).\]\end{proposition}
\begin{proof}[Proof of Proposition \ref{Pr:LargeFA1}]
We set $z_{\delta,m}:=\frac{\theta_{K+\delta m}}\delta.$ Direct computations show that $z_{\delta,m}$ solves
\begin{equation}
\begin{cases}
-\frac\mu\delta\Delta z_{\delta,m}-z_{\delta,m}\left(-m-z_{\delta,m}\right)=0&\text{ in }\O\,, 
\\ \frac{\partial z_{\delta,m}}{\partial \nu}=0&\text{ on }\partial \O\,, 
\\ z_{\delta,m}\geq 0\,, z_{\delta,m}\neq 0.
\end{cases}
\end{equation}
Thus the large fishing ability limit corresponds to a large-diffusivity limit for a standard logistic diffusive equation. We can now apply \cite[Appendix A-Convergence of the series]{Mazari2020}. Let us simply recall the main steps: first it is proved that, uniformly in $m$,  the asymptotic expansion 
\begin{equation}\label{Eq:Op}
z_{\delta,m}=m_0+\underset{\mu \to \infty}O\left(\delta\right)
\end{equation}
holds in $W^{1,2}(\O)$. We  then use a standard bootstrap argument, to obtain that \eqref{Eq:Op} holds in any $W^{2,p}(\O)$.
From the definition of $z_{\delta,m}$ we infer that $\theta_{K+\delta m}$ admits the expansion 
\begin{equation}\label{Eq:DA1}
\theta_{K+\delta m}=\delta m_0+\underset{\mu \to \infty}O\left(\delta^2\right)
\end{equation}in $\mathscr C^1(\O)$.
This concludes the proof.
\end{proof}
From this proposition we obtain an asymptotic expansion of the adjoint state $q_{\delta,m}$:

\begin{proposition}\label{Pr:LargeFA2}
Uniformly in $m\in \mathcal N(\O)$, the following asymptotic expansion holds in $\mathscr C^1(\O)$:
\[ q_{\delta, m}=\frac1\delta\cdot \frac{K_0}{m_0}+\underset{\delta \to 0}o\left(\frac1\delta\right).\]
\end{proposition}
\begin{proof}[Proof of Proposition \ref{Pr:LargeFA2}]
From Proposition \ref{Pr:LargeFA1} the function $q_{m,\delta}$ solves
\begin{equation}
-\mu \Delta q_{m,\delta}-\delta q_{m,\delta}\left(-m-2m_0+\underset{\delta \to 0}O(\delta)\right)=K+\delta m.\end{equation}
We set $z_{\delta,m}:=\delta q_{\delta,m}.$ Direct computations show that $z_{\delta,m}$ solves
\begin{equation}
\begin{cases}
-\frac\mu\delta\Delta z_{\delta,m}-z_{\delta,m}\left(-m-2m_0+\underset{\delta \to 0}O(\delta)\right)=K+\delta m&\text{ in }\O\,, 
\\ \frac{\partial z_{\delta,m}}{\partial \nu}=0&\text{ on }\partial \O.
\end{cases}
\end{equation}
We can apply exactly the same reasoning as in \cite{Mazari2020} to obtain that, in $\mathscr C^1(\O)$, we have 
\[ z_{\delta,m}=\frac{K_0}{m_0}+\underset{\delta \to 0}o\left(1\right), 
\]
whence the conclusion.\end{proof}

From Proposition \ref{Pr:LargeFA2} we obtain that, for any $\eta>0$ there exists $\delta_1>0$ such that, for any $0<\delta<\delta_1$, for any $m\in \mathcal N(\O)$, 
\[ 1-q_{\delta,m}\leq 1-\frac{K_0}{2m_0\delta}\leq -\frac1\eta.\]
However, we may proceed as in the proof of Theorem \ref{Th:MonotonicityTrue}: for any $m_1\leq m_2\,, m_1\neq m_2$,  $\overline J_{\delta,\mu}(m_2)<\overline J_{\delta,\mu}(m_1)$, so that the functional is no longer monotonic.

\end{proof}

\subsection{Proofs of Theorems \ref{Th:Concave}-\ref{Th:Convex}}

\paragraph{Reformulation of the second-order Gateaux-derivatives}
The proofs of the concavity of $J_\mu$ and the characterisation of maximisers of $\overline{J}_{\delta,\mu}$ as extreme points of the admissible sets rely on the type of computations carried out in \cite{M2021,Mazari2021,MPRobin} and in particular on a certain reformulation of the second-order Gateaux-derivatives of the functionals under consideration.

\paragraph{Reformulation of $\ddot{J}_\mu$}
We start with $\ddot J_\mu (\alpha)[h,h]$, which for notational convenience we write here $\ddot J(\alpha)[h,h]$ (in other words we have dropped the subscript $\mu$). Throughout the computations that follow we work with a fixed $\alpha \in \mathcal M_=(\kappa,V_0)$ and a fixed admissible perturbation $h$ at $\alpha$. We recall that from Lemma \ref{Le:GD} we have the expression 
\[ \frac12\ddot J(\alpha)[h,h]=\fint_\O(1-p_\alpha)h \dot \theta_\alpha-\fint_\O p_\alpha \dot \theta_\alpha^2.\] Now observe that we may rewrite 
\[ h:=\frac{\mu \Delta \dot \theta_\alpha+\dot \theta_\alpha\left(K-\alpha-2\theta_\alpha\right)}{\theta_\alpha},\]
whence, defining 
\[ \psi_\alpha:=\frac{1-p_\alpha}{\theta_\alpha},\] we derive 
\begin{align*}
\fint_\O (1-p_\alpha)h\dot\theta_\alpha&=\fint_\O \psi_\alpha\left(\mu\dot \theta_\alpha \Delta \dot\theta_\alpha+\dot\theta_\alpha^2(K-\alpha-2\theta_\alpha)\right)
\\&=\fint_\O \psi_\alpha\left(\frac\mu2\Delta\left(\dot \theta_\alpha^2\right)-\mu\left|\n \dot \theta_\alpha \right|^2+\dot\theta_\alpha^2(K-\alpha-2\theta_\alpha)\right)
\\&=-\mu\fint_\O \psi_\alpha \left|\n \dot \theta_\alpha \right|^2
\\&+\fint_\O \dot\theta_\alpha^2\left(\frac\mu2\Delta \psi_\alpha+\psi_\alpha(K-\alpha-2\theta_\alpha)\right).
\end{align*}
We have thus established the following lemma:
\begin{lemma}\label{Le:2GD}
For any $\alpha \in \mathcal M_=(\kappa,V_0)$, for any admissible perturbation $h$ at $\alpha$, there holds
\[ \frac12\ddot J(\alpha)[h,h]=-\mu\fint_\O \frac{1-p_\alpha}{\theta_\alpha}|\n \dot \theta_\alpha|^2+\fint_\O \dot \theta_\alpha^2\left(\frac\mu2\Delta\left(\frac{1-p_\alpha}{\theta_\alpha}\right)+\frac{1-p_\alpha}{\theta_\alpha}(K-\alpha-2\theta_\alpha)-p_\alpha\right).
\]
\end{lemma}

\paragraph{Reformulation of $\overline J_{\delta,\mu}$}
We can carry the same type of computations for the second-order Gateaux derivative of $\overline J_{\delta,\mu}$: let $m\in \mathcal N(\O)$ be fixed and $h$ be an admissible perturbation at $h$. We know from Lemma \ref{Le:GDLargeFishing} that 
\[\frac12\ddot{\overline J}_{\delta,\mu}(m)[h,h]=\delta\fint_\O (1-q_{\delta,m})h\dot \theta_{K+\delta m}-\fint_\O q_{\delta,m}\dot\theta_{K+\delta m}^2.\]
However, we may rewrite 
\[h=\frac{\mu \Delta \dot \theta_{K+\delta m}+\dot  \theta_{K+\delta m}\left(-\delta m-2\theta_{K+\delta m}\right)}{\delta \theta_{K+\delta m}}\] and thus obtain, defining 
\[ \p_{\delta,m}:=\frac{q_{\delta,m}-1}{\theta_{K+\delta m}} \]
\begin{align*}
\delta\fint_\O (1-q_{\delta,m})h\dot \theta_{K+\delta m}&=\mu\fint_\O \p_{\delta,m}\left|\n \dot \theta_{K+\delta m}\right|^2+\fint_\O \dot\theta_{K+\delta m}^2\left(\frac\mu2\Delta \p_{\delta,m}-(-\delta m-2\theta_{K+\delta m})\right).
\end{align*}
Hence the following lemma holds:
\begin{lemma}\label{Le:2GDLargeFishing}
For any $m \in \mathcal N_=(\O)$, for any admissible perturbation $h$ at $m$, there holds
\begin{multline*} \frac12\ddot{\overline J}_{\delta,\mu}(m)[h,h]=\mu\fint_\O {\frac{q_{\delta,m}-1}{\theta_{K+\delta m}}}\left|\n \dot \theta_{K+\delta m}\right|^2\\+\fint_\O \dot\theta_{K+\delta m}^2\left(\frac\mu 2\Delta\left( {\frac{1-q_{\delta,m}}{\theta_{K+\delta m}}}\right)-(-\delta m-2\theta_{K+\delta m})\right)-\fint_\O q_{\delta,m}\dot \theta_{K+\delta m}^2.
\end{multline*}
\end{lemma}

\paragraph{Proofs of Theorems \ref{Th:Concave}-\ref{Th:Convex}}
We now get to the core of the proofs.
\begin{proof}[Proof of Theorem \ref{Th:Concave}]
Theorem \ref{Th:Concave} contains two statements, one dealing with the one-dimensional case, the other one dealing with the multi-dimensional case. Both rely on the same estimate of the expression of the second order gateaux derivative give by Lemma \ref{Le:2GD}.

From the proof Lemma \ref{Le:LFA1}  we lift estimate \eqref{Bloodborne}, which ascertains that \[\lim_{V_0\to 0}\sup_{\alpha \in \mathcal M_\leq (\kappa,V_0)}\Vert p_\alpha \Vert_{\mathscr C^1(\O)}=0.\]

Let us introduce, for any $\alpha \in \mathcal M_\leq(\kappa,V_0)$, the potential
\[ W_\alpha:=\left(\frac\mu2\Delta\left(\frac{1-p_\alpha}{\theta_\alpha}\right)+\frac{1-p_\alpha}{\theta_\alpha}(K-\alpha-2\theta_\alpha)-p_\alpha\right)
\]
as well as, for any $\alpha\in \mathcal M_{\leq}(\kappa,V_0)$, with $V_0$ small enough to ensure that for any $\alpha\in \mathcal M_\leq(\kappa,V_0)$ we have $1-p_\alpha\geq \frac12$, the operator 
\[ \mathcal L_\alpha:=-\mu\nabla\cdot\left(\frac{1-p_\alpha}{\theta_\alpha}\nabla\right)-W_\alpha.\]
Let $\xi(\alpha)$ be the first eigenvalue of $\mathcal L_\alpha$. $\xi(\alpha)$ is defined, by its Rayleigh quotient, as 
\begin{equation}\label{Eq:RyleighAlpha}
\xi(\alpha):=\inf_{u\in W^{1,2}(\O)\,, \fint_\O u^2=1}\left(\mu\fint_\O \frac{1-p_\alpha}{\theta_\alpha}|\n u|^2-\fint_\O W_\alpha u^2.\right).\end{equation}

From Lemma \ref{Le:2GD}, there holds, for any $\alpha \in \mathcal M_\leq(\kappa,V_0)$ and any admissible perturbation $h$ at $\alpha$, 
\begin{equation}
\frac12 \ddot J(\alpha)[h,h]\leq -\xi(\alpha)\fint_\O \dot \theta_\alpha^2.
\end{equation}
The goal is now to get the asymptotic behaviour of $\xi(\alpha)$ as $V_0\to 0$ and, more precisely, to obtain that \begin{equation}\label{Eq:Monsieur}\lim_{V_0\to 0^+}\inf_{\alpha \in \mathcal M_\leq(\kappa,V_0)}\xi(\alpha)>0,\end{equation} which would suffice to prove the concavity of the functional. To do so, a first step is to understand the behaviour of the potential $W_\alpha$ as $V_0 \to 0$ .

As
\begin{align*}
\mu\Delta \left(\frac{1-p_\alpha}{\theta_\alpha}\right)&=-\mu \frac{\Delta p_\alpha}{\theta_\alpha}+2\mu \frac{\langle \n p_\alpha,\n \theta_\alpha\rangle}{\theta_\alpha^2}-(1-p_\alpha)\mu \frac{\Delta \theta_\alpha}{\theta_\alpha^2}+2(1-p_\alpha)\frac{|\n \theta_\alpha|^2}{\theta_\alpha^3}
\end{align*}
and as 
\[ -\mu \Delta p_\alpha=\alpha+p_\alpha(K-\alpha-2\theta_\alpha)\]
we deduce that, if we define the limit potential 
\[ W_0:=\frac{\mu}2\Delta\left(\frac1{\overline\theta}\right)+\frac1{\overline \theta}(K-2\overline \theta),\] then it follows from \eqref{Eq:CvUni}-\eqref{Bloodborne} that 
\begin{equation}\label{Eq:CvPotentiels}
\forall p\in [1,+\infty)\,, \lim_{V_0\to 0}\sup_{\alpha \in \mathcal M_\leq(\kappa,V_0)} \Vert W_\alpha-\overline{W}\Vert_{L^p}=0.\end{equation}
Let $\overline \xi$ be the first eigenvalue of the operator
\[ \overline{\mathcal L}:=-\mu \nabla\cdot\left(\frac1{\overline \theta}\right)-\overline W.
\]
By a standard method that we detail in Appendix \ref{Ap:ConvEigen} this implies 

\begin{equation}\label{Eq:CvEi1}
\lim_{V_0\to 0}\sup_{\alpha \in \mathcal M_\leq(\kappa,V_0)}\left\vert \xi(\alpha)-\overline\xi\right\vert=0.\end{equation}

In particular, the proof of the Theorem is complete, provided we can prove that 
\[ \overline \xi>0.\] 

\paragraph{First analysis of $\overline \xi$}
Let us first observe that we can expand $\overline W$ as follows:
\begin{align*}
\overline W&=\frac{\mu}2\Delta\left(\frac1{\overline\theta}\right)+\frac1{\overline \theta}(K-2\overline \theta)
\\&=\frac{\mu}2\left(-\frac{\Delta \overline\theta}{\overline \theta^2} +2\frac{\left|\n \overline \theta\right|^2}{\theta^3}\right)+\frac1{\overline \theta}(K-2\overline \theta)
\\&=\frac1{2\overline \theta}(K-\overline \theta)&\text{ since $\overline \theta$ solves \eqref{Eq:LFAEq}}
\\&+\mu\frac{\left|\n \overline \theta\right|^2}{\overline \theta^3}+\frac1{\overline \theta}(K-2\overline \theta)
\\&=\frac32\cdot \frac{K-\overline \theta}{\overline \theta}+\mu\frac{\left|\n \overline \theta\right|^2}{\overline \theta^3}-1
\\&=\frac32\cdot \frac{K-\overline \theta}{\overline \theta}+\frac{3}4\mu\frac{\left|\n \overline \theta\right|^2}{\overline \theta^3}
\\&+\left(\frac{\mu}4\cdot \frac{\left|\n \overline\theta\right|^2}{\overline \theta^3}-1\right).
\end{align*}
Our last rewriting may seem mysterious at first, but it is justified by the following fact: if we define 
\begin{equation}\label{DeZ} \overline Z:=\frac32\cdot \frac{K-\overline \theta}{\overline \theta}+\frac{3}4\mu\frac{\left|\n \overline \theta\right|^2}{\overline \theta^3}\end{equation}
we can actually prove that the first eigenvalue $\overline A$ of the operator 
\begin{equation}\label{DefF} \overline{\mathcal F}:=-\mu \nabla \cdot\left(\frac1{\overline \theta}\nabla \right)-\overline Z\end{equation}is equal to 0. We will then use a monotonicity principle for eigenvalues. We start with the fact we just claimed:
\begin{lemma}
$\overline Z$ being defined by \eqref{DeZ} and 
$\overline{\mathcal F}$ being defined by \eqref{DefF}, the first eigenvalue $\overline A$ of $\overline{\mathcal F}$ is zero, and its  associated eigenfunction is $\p=\overline \theta^{\frac32}$. 
\end{lemma}

\begin{proof}
Let $\p:=\overline \theta^{\frac32}$. We have 
\[ \nabla \p=\frac32\cdot \sqrt{\overline \theta}\n \overline \theta\] and so
\[
\frac{\nabla \p}{\overline \theta}=\frac32\cdot\frac{\n \overline \theta}{\sqrt{\overline \theta}}.\]
Thus
\begin{align*}
 -\mu\nabla \cdot\left(\frac{\nabla \p}{\overline \theta}\right)&=\frac32\cdot\frac{-\mu \Delta \overline\theta}{\sqrt{\overline \theta}}-\mu\frac{3}4\frac{\left|\n \overline \theta\right|^2}{\sqrt{\overline \theta}}
 \\&=\frac32\cdot\frac{ \overline\theta(K-\overline \theta)}{\sqrt{\overline \theta}}-\mu\frac{3}4\frac{\left|\n \overline \theta\right|^2}{\sqrt{\overline \theta}}
\\&=\frac32\cdot{\sqrt{ \overline\theta}(K-\overline \theta)}-\mu\frac{3}4\frac{\left|\n \overline \theta\right|^2}{\sqrt{\overline \theta}}
\\&=\varphi \left(\frac{3}2\cdot \frac{K-\overline \theta}{\overline \theta}-\mu\frac34\cdot\frac{\left|\n \overline \theta\right|^2}{\overline \theta^3}\right).
\end{align*}
Thus $\p$ is an eigenfunction of $\overline F$ associated with the eigenvalue 0. As $\p=\overline \theta^{\frac32}>0$ and as the first eigenvalue of $\overline F$ is the only eigenvalue whose associated eigenfunctions have constant signs, we deduce that $\p$ is a principal eigenfunction and that the first eigenvalue of $\overline F$ is 0.
\end{proof}
Now, if we can ensure that $\overline W\leq \overline Z\,, \overline W\neq \overline Z$ then by virtue of the monotonicity of the first eigenvalue \cite[Lemma 2.1]{DockeryHutsonMischaikowPernarowskiEvolution} we have 
\[ \overline \xi>0\] so that \eqref{Eq:Monsieur}. Thus the proof would be complete.

We now notice that 
\[\overline W-\overline Z=\left(\frac{\mu}4\cdot \frac{\left|\n \overline\theta\right|^2}{\overline \theta^3}-1\right).\]Proving that $\overline\xi>0$ boils down to investigating whether or not $\left(\frac{\mu}4\cdot \frac{\left|\n \overline\theta\right|^2}{\overline \theta^3}-1\right)<0$. We do that in the one-dimensional case and, in the higher dimensional case, for resources distributions that are close to a constant.

\begin{enumerate}
\item \underline{In the one-dimensional case:} here we use an estimate of Bai, He and Li \cite[Estimate (2.2)]{BaiHeLi}, namely, that, in the one-dimensional case, provided $K$ is bounded (which is the case here by assumption) there holds 
\[\frac{\mu}2\left(\overline \theta'(x)\right)^2\leq \frac{\overline\theta(x)^3}3.\]

\begin{remark} It should be noted that in \cite[Estimate (2.2)]{BaiHeLi} this estimate is proved when $\overline\theta$ is monotonic, and that they then integrate this identity on such an interval to obtain an integral estimate. Then, they present, in \cite[Steps 2 and 3, proof of Theorem 2.2]{BaiHeLi}, a way to glue these integral estimates. The very same strategy works to prove that \cite[Estimate (2.2)]{BaiHeLi} is valid on the entire interval.\end{remark}
In particular, 
\[ 1-\frac{\mu}4\cdot \frac{\left(\overline \theta'(x)\right)^2}{\overline \theta^3}\geq 1-\frac16=\frac56
\] so that the proof is concluded.
\item\underline{In the higher-dimensional case:} in that second case, since we work with variable $K$, let us add the subscript $K$ to the notation $\overline \theta$ and denote by $\overline\theta_K$ the solution of \eqref{Eq:LFAEq}.
In this case the only thing that should be noted is that, when $\overline K$ is constant, $\overline \theta_{\overline K}\equiv \frac1{\mathrm{Vol}(\O)}K_0=\overline K$. In that case, 
\[ \overline W-\overline Z=-1<0.\] 
However, a simple adaptation of the arguments of \eqref{Eq:CvUni} proves that for any $\delta'>0$ there exists a constant $\e_3>0$ such that, for any $K\in \mathcal K(\O)$, if $\Vert K-\overline K\Vert_{L^1(\O)}\leq \e_3$ then 
\[ \Vert \overline \theta_K-\overline\theta_{\overline K}\Vert_{\mathscr C^1(\O)}\leq \delta'.\]
If $\delta'$ is small enough, this implies that  for any $K\in \mathcal K(\O)$ such that $\Vert K-\overline K\Vert_{L^1(\O)}\leq \e_3$ we have 
\[ -1+\frac14\frac{|\n \overline\theta_K|^2}{\overline \theta_K^3}< -\frac12.\] The conclusion follow in exactly the same way.
\end{enumerate}

\end{proof}

\begin{proof}[Proof of Theorem \ref{Th:Convex}]
For the proof of Theorem \ref{Th:Convex} we follow the same type of strategy as the one used for the proof of Theorem \ref{Th:Concave}. We start with the expression of the second-order Gateaux derivative given in Lemma \ref{Le:2GDLargeFishing}: for any $m\in \mathcal N_=(\O)$ and any admissible perturbation $h$ at $m$ we have 
\begin{multline*}
\frac12\ddot{\overline J}_{\delta,\mu}(m)[h,h]=\mu\fint_\O {\frac{q_{\delta,m}-1}{\theta_{K+\delta m}}}\left|\n \dot \theta_{K+\delta m}\right|^2\\+\fint_\O \dot\theta_{K+\delta m}^2\left(\frac\mu 2\Delta\left( {\frac{1-q_{\delta,m}}{\theta_{K+\delta m}}}\right)-(-\delta m-2\theta_{K+\delta m})\right)-\fint_\O q_{\delta,m}\dot \theta_{K+\delta m}^2.
 \end{multline*}
 Recall that, from Proposition \ref{Pr:LargeFA2}, there exists $\delta_2>0$ small enough such that, for any $\delta\leq \delta_2$ and for any $m\in \mathcal N_=(\O)$ there holds 
 \[q_{\delta,m}-1\geq \frac{\sup_{m\in \mathcal N_=(\O)}\Vert \theta_{K+\delta m}\Vert_{L^\infty(\O)}}2.\]
  For $\delta\leq \delta_2$ we can thus bound the second-order derivative as 
 \[ \frac12\ddot{\overline J}_{\delta,\mu}(m)[h,h]\geq \frac{\mu}2\fint_\O \left|\n \dot \theta_{K+\delta m}\right|^2+\fint_\O Y_{\delta,m}\dot\theta_{K+\delta m}^2,\]
  where the potential $Y_{\delta,m}$ is defined as
 \[ Y_{\delta,m}:=\left(\frac\mu2\Delta\left( {\frac{q_{\delta,m}-\delta}{\theta_{K+\delta m}}}\right)-(-\delta m-2\theta_{K+\delta m})\right)-q_{\delta,m}.\]
However, expanding $Y_{\delta,m}$ as was done in the proof of Theorem \ref{Th:Concave} for $\Delta\left(\frac{1-p_\alpha}{\theta_{\alpha}}\right)$ we obtain the existence of a constant $\beta=\beta(\delta)$ such that 
\[ \forall m\in \mathcal N_=(\O)\,, \Vert Y_m\Vert_{L^\infty(\O)}\leq \beta.\] Defining $\gamma:=\frac{\mu}2$ we thus have, for the second-order Gateaux-derivative, the lower estimate
 \[ \frac12\ddot{\overline J}_{\delta,\mu}(m)[h,h]\geq \gamma\fint_\O \left|\n \dot \theta_{K+\delta m}\right|^2-\beta\fint_\O \dot\theta_{K+\delta m}^2.\]However, we are now exactly in the proper situation to mimic the proof of \cite[Theorem 1]{Mazari2021}: argue by contradiction and assume that there exists a non-bang-bang solution $m^*$ of \eqref{Pv:1EqLargeFishing}. In particular the set $\omega^*:=\{0<m^*<1\}$ has a positive measure. Let $M>0$ be arbitrarily large. Following \cite[Proof of Theorem 1, Eq. (2.20) and below]{Mazari2021} there exists an admissible perturbation $h_M$ at $m^*$ supported in $\omega^*$ such that 
 \[ \fint_\O \left|\n \dot\theta_{K+\delta m^*}\right|^2>M\fint_\O \dot\theta_{K+\delta m^*}^2.\] Taking $M:=\frac\beta\gamma+1$ we obtain the required contradiction: for the perturbation $h_{\frac\beta\gamma}$ there holds 
 \[ \frac12\ddot{\overline J}_{\mu,\delta}(m^*)\left[h_{\frac\beta\gamma},h_{\frac\beta\gamma}\right]\geq \gamma\fint_\O \dot\theta_{K+\delta m^*}^2>0,\]
 in contradiction with the optimality of $m^*$.
\end{proof}

\section{Proofs of Theorem \ref{Th:AsymptoticSingle} }\label{Se:2}
\subsection{Proof of Proposition \ref{Pr:Large0}}
Before we prove Theorem \ref{Th:AsymptoticSingle}  we prove Proposition \ref{Pr:Large0}.
\begin{proof}[Proof of Proposition \ref{Pr:Large0}]
We recall that 
\[J^0:\alpha\mapsto\left(\fint_\O \alpha\right)\left(K_0-\fint_\O \alpha\right).\] Clearly, $J^0$ is twice Gateaux-differentiable at every $\alpha$ and, for any $\alpha \in \mathcal M_\leq(\kappa,V_0)$ and any admissible perturbation $h$ at $\alpha$ there holds 
\[ \dot J^0(\alpha)[h,h]=\left(\fint_\O h\right)\left(K_0-2\fint_\O \alpha\right).
\]
In particular, if $\fint_\Omega \alpha\leq V_0<\frac{K_0}2$ the functional $J^0$ is increasing on $\mathcal M_\leq(\kappa,V_0)$, so that any solution $\alpha^*$ of \eqref{Pv:1IneqLarge} satisfies $\fint_\O \alpha^*=V_0$. Thus, $\alpha^*$  is also a solution of \eqref{Pv:1EqLarge}.

If, on the contrary, we assume that $V_0>\frac{K_0}2$, consider a solution $\alpha^*$ of \eqref{Pv:1IneqLarge}. Let us prove that we necessarily have $\fint_\O \alpha^*<V_0$. If, by contradiction, we had
\[ \fint_\O \alpha^*=V_0\] then, for any non-positive, non-zero perturbation $h$, we have 
\[ \dot J^0(\alpha^*)[h]=\left(\fint_\O h\right)\left(K_0-2\fint_\O \alpha^*\right)>0,\] in contradiction with the optimality of $\alpha^*$. In particular, $\fint_\O \alpha^*<V_0$ and so the two problem \eqref{Pv:1EqLarge} and \eqref{Pv:1IneqLarge} do not coincide.

\end{proof}

\subsection{Proof of Theorem \ref{Th:AsymptoticSingle}}
\begin{proof}[Proof of Theorem \ref{Th:AsymptoticSingle}]
\textbf{Reformulation of $J^1$: }
To prove Theorem \ref{Th:AsymptoticSingle}, we need a tractable rewriting of the function $J^1$. Let us recall that we defined the constant \[M_\alpha:=K_0-\fint_\O \alpha.\] As we are working with an equality constraint we may drop the subscript $\alpha$ and simply define 
\[ M_0:=K_0-V_0.\] The functional $J^1$ is defined as  
\[ J^1(\alpha)=\fint_\O \alpha v_\alpha\text{ where }\begin{cases}-\Delta v_\alpha-M_0(K-\alpha-M_0)=0&\text{ in }\O\,, 
\\ \frac{\partial v_\alpha}{\partial\nu}=0&\text{ on }\partial \O\,, 
\\ \fint_\O v_\alpha=\frac1{M_0^2}\fint_\O |\n v_\alpha|^2.
\end{cases}\]
Let us introduce, for any $\alpha \in \mathcal M_\leq(\kappa,V_0)$, the solution $\hat v_\alpha$ of 
\begin{equation}\label{Eq:HatV}
\begin{cases}
-\Delta \hat v_\alpha-M_0\left(K-\alpha-M_0\right)=0&\text{ in }\O,
\\ \frac{\partial \hat v_\alpha}{\partial \nu}=0&\text{ on }\partial \O\,, 
\\ \fint_\O \hat v_\alpha=0.
\end{cases}
\end{equation}
Clearly we have 
\[ v_\alpha=\hat v_\alpha+\frac1{M_0^2}\fint_\O |\n\hat v_\alpha|^2,\]
so that 
\begin{align*}
J^1(\alpha)&=\fint_\O \alpha v_\alpha
\\&=\fint_\O \alpha \hat v_\alpha
\\&+\frac1{M_0^2}\fint_\O |\n\hat v_\alpha|^2\fint_\O\alpha
\\&=\fint_\O (\alpha+M_0-K)\hat v_\alpha
\\&-\fint_\O (M_0-K)\hat v_\alpha
\\&+\frac{V_0}{M_0^2}\fint_\O |\n\hat v_\alpha|^2
\\&=-\frac1{M_0}\fint_\O |\n \hat v_\alpha|^2
\\&+\fint_\O K\hat v_\alpha\quad \left(\text{as $\fint_\O \hat v_\alpha=0$}\right)
\\&+\frac{V_0}{M_0^2}\fint_\O |\n\hat v_\alpha|^2
\\&=\left(\frac{2V_0-K_0}{M_0^2}\right)\fint_\O |\n \hat v_\alpha|^2+\fint_\O K\hat v_\alpha.
\end{align*}

\textbf{Analysis of the second order derivative of $J^1$: }
But now observe that, if we define 
\[j_1:\alpha \mapsto \left(\frac{2V_0-K_0}{M_0^2}\right)\fint_\O |\n \hat v_\alpha|^2\,, j_2:\alpha\mapsto \fint_\O K\hat v_\alpha,\]
then $j_2$ is linear in $\alpha$ as the map $\alpha \mapsto \hat v_\alpha$ is linear. As 
\[ J^1=j_1+j_2\] the second order derivative of $J^1$ is determined by the second-order derivative of $j_1$. However, it is straightforward to see, mimicking the computations of \cite[Proof of Theorem 1, Step 1]{Mazari2020}, that, for any $\alpha\in \mathcal M_=(\kappa,V_0)$ and any admissible perturbation $h$ at $\alpha$, we have 
\[ \ddot j_1(\alpha)[h,h]=\left(\frac{2V_0-K_0}{M_0^2}\right)\fint_\O \left|\n \dot{\hat v}_\alpha\right|^2\text{ where }\begin{cases}
-\Delta \dot{\hat v}_\alpha+M_0h=0&\text{ in }\O\,, 
\\ \frac{\partial \dot{\hat v}_\alpha}{\partial \nu}=0&\text{ on }\partial \O\,, 
\\ \fint_\O \dot{\hat v}_\alpha=0.
\end{cases}
\]
In particular, if $2V_0>K_0$ the functional is (strictly) convex. Thus, any solution of \eqref{Pv:1EqLarge2} is an extreme point of $\mathcal M_=(\kappa,V_0)$ that is, any solution is a bang-bang function. Conversely, if $2V_0<K_0$, the functional is (strictly) concave. This ends the proof of the two first-points of the theorem.

Now let us move to the characterisation of optimisers in the convex regime (point 3 of the theorem). Assume $\O=(0;1)$, assume that $2V_0>K_0$ and that $K$ is a non-increasing, non constant function. To give an explicit description of the maximiser $\alpha$ we need to use the notion of non-increasing rearrangement. Let us recall the following definition:
\begin{definition}
For any non-negative function $f\in L^1(0;1)$ there exists a unique non-increasing, non-negative function $f^\#\in L^\infty(0;1)$ such that 
\[ \forall t\geq 0\,, \mathrm{Vol}\left(\{f\geq t\}\right)=\mathrm{Vol}(\{f^\#\geq t\}).\] Similarly, there exists a unique non-decreasing, non-negative function $f_\#\in L^\infty(0;1)$ such that \[ \forall t\geq 0\,, \mathrm{Vol}\left(\{f\geq t\}\right)=\mathrm{Vol}(\{f_\#\geq t\}).\]
\end{definition}
Two inequalities are of paramount importance when dealing with rearrangements:
\begin{enumerate}
\item 
The celebrated P\'{o}lya-Szeg\"{o} inequality: it states that, if $f\in W^{1,2}(0;1)$, then $f^\#\in W^{1,2}(\O)$ and, furthermore, that we have 
\begin{equation}\label{Eq:PS}\fint_0^1\left| \left(f^\#\right)'\right|^2\leq \fint_0^1\left| f'^2\right|.\end{equation}
\item The Hardy-Littlewood inequality: it states that, if $f\,, g\in L^1(\O)$ are bounded functions then 
\begin{equation}\label{Eq:HL}\fint_0^1 f_\# g^\#\leq\fint_0^1 fg\leq \fint_0^1 f^\#g^\#.
\end{equation}
\end{enumerate}
While rearrangements are central in the calculus of variations (we refer for instance to \cite{Baernstein,Bandle,BHR,Kawohl,Kesavan,Rakotoson}) and has wide ranging applications, we focus here on Talenti inequalities. Originating in the seminal \cite{Talenti}, in the case of the Schwarz rearrangement for Dirichlet boundary conditions, these inequalities aim at comparing the solution $u$ of a Poisson equation of the form $-\Delta u=f$ with Dirichlet boundary conditions with the solution $v$ of a symmetrised equation. Among the many results related to possible extensions and to the qualitative analysis of these inequalities to other operators \cite{Alvino1990,AlvinoNitschTrombetti,Bandle,Mazari2022,RakotosonMossino,Sannipoli2022} let us focus on the results of \cite{Langford}. To use them, we need to recall the rearrangement order on $L^1(0;1)$: for any two non-negative functions $f,g\in L^1(0;1)$, we say that $f$ dominates $g$ in the sense of rearrangements and we write \[
f\prec g
\] if, and only if, 
\[
\forall r \in [0;R]\,, \fint_0^r f^\#\leq \fint_0^r g^\#.
\]
Our goal is to show that minor adaptation of \cite[Chapter 5]{Langford} yields the following result: defining, for any $f\in L^1(\O)$ such that $\fint_0^1f=0$, the solution $u_f$ of 
\begin{equation}
\begin{cases}
-(u_f)''=f&\text{ in }(0;1)\,, 
\\ u_f'(0)=u_f'(1)=0\,, 
\\ \fint_0^1 u_f=0.\end{cases}\end{equation}
we claim that, for any $g$ such that $f\prec g$, there holds
\begin{equation}\label{Eq:TalentiLangford} u_f\prec u_{g^\#}.\end{equation}
Before we prove \eqref{Eq:TalentiLangford}, let us investigate why this yields the required result.
\begin{lemma}\label{Le:TalentiImplies} If estimate \eqref{Eq:TalentiLangford} holds for any non-negative $f\in L^1(\O)$, if $K=K^\#$ is not constant and if $V_0>\frac{K_0}2$ then the unique solution of \eqref{Pv:1EqLarge2} is given by
\[\alpha^*=\kappa \mathds 1_{[1-\ell;1]}
\] where $\kappa\ell=1$.
\end{lemma}

\begin{proof}[Proof of Lemma \ref{Le:TalentiImplies}]
The proof of this Lemma rests upon a rewriting of $J^1$ in terms of natural energy functional associated with $\hat v_\alpha$.
\paragraph{Rewriting of $J^1$ in terms of an energy functional}
We start from the fact that for any $\alpha$ we have
\[ J^1(\alpha)=\frac{2V_0-K_0}{M_0^2}\fint_0^1 |\hat v_\alpha'|^2+\fint_0^1 K\hat v_\alpha
\]where $\hat v_\alpha$ satisfies \eqref{Eq:HatV}
To alleviate notations, define 
\[C_0:=2\frac{2V_0-K_0}{M_0^2}>0,\] so that 
\[ \forall \alpha\,, J^1(\alpha)=C_0\fint_0^1 |\hat v_\alpha'|^2+\fint_0^1 K\hat v_\alpha.\]
 
 However, \eqref{Eq:HatV} admits a natural variational formulation: introduce the space
\[\mathcal X:=\left\{ u\in W^{1,2}(0;1)\,, \fint_0^1 u=0\right\}\] and define the energy functional 
\[ \mathcal E_\alpha:\mathcal X\ni u\mapsto \frac12\fint_0^1|u'|^2-M_0\fint_0^1(K-\alpha-M_0)u.\]Then $\hat v_\alpha$ is the unique solution of 
\[ \min_{u\in \mathcal X}\mathcal E_\alpha(u).\]
Now observe that from the weak formulation of \eqref{Eq:HatV} we have 
\begin{equation}
\fint_0^1 |\hat v_\alpha'|^2=M_0\fint_0^1(K-\alpha-M_0)\hat v_\alpha=-\mathcal E_\alpha(\hat v_\alpha)+\frac12\fint_0^1|\hat v_\alpha'|^2
\end{equation}
so that in the end
\[ \fint_0^1 |\hat v_\alpha'|^2=-2\mathcal E_\alpha(\hat v_\alpha).\]
This allows us to rewrite  $ J^1$ as
\begin{equation}\label{Eq:Reexpression}J^1(\alpha)=-2C_0\mathcal E_\alpha(\hat v_\alpha)+\fint_0^1 K \hat v_\alpha.\end{equation}

We will prove that rearranging the coefficients of the equation increases each term appearing in the right-hand-side of \eqref{Eq:Reexpression}.

\paragraph{Rearranging $\alpha$ increases $-\mathcal E_\alpha(\hat v_\alpha)$}
Let us start with the energy functional. From  the P\'{o}lya-Szeg\"{o} inequality \eqref{Eq:PS} we know that 
\[ \fint_0^1 |\hat v_\alpha'|^2\geq \fint_0^1 \left|\left(\hat v_\alpha^\#\right)'\right|^2.\]
Furthermore, from equimeasurability of the rearrangement, we have 
\[ \fint_0^1 M_0 \hat v_\alpha=\fint_0^1 M_0(\hat v_\alpha^\#).\]
Finally, from the Hardy-Littlewood inequality \eqref{Eq:HL} there holds
\[ \fint_0^1 K\hat v_\alpha\leq \fint_0^1 K^\#\hat v_\alpha^\#=\fint_0^1 K\hat v_\alpha^\#\text{ and }\fint_0^1 \alpha_\# \hat v_\alpha^\#\leq \fint_0^1 \alpha \hat v_\alpha.
\]
This gives
\[ \mathcal E_{\alpha_\#}\left(\hat v_{\alpha_\#}\right)\leq \mathcal E_{\alpha_\#}\left(\hat v_\alpha^\#\right)\leq \mathcal E_\alpha(\hat v_\alpha).\]

\paragraph{Rearranging increases $\fint _0^1 K \hat v_\alpha$}
Let us now observe the effect of rearrangement on the equation satisfied by $\hat v_\alpha$. Assume that the Talenti inequality \eqref{Eq:TalentiLangford} holds. Then we know (taking $f=g$ in \eqref{Eq:TalentiLangford}) that 
\[ \hat v_\alpha \prec z\] where $z$ solves
\[
\begin{cases}
-z''=M_0\left(K-\alpha-m_0\right)^\#&\text{ in }(0;1)\,, 
\\ z'(0)=z'(1)=0\,, 
\\ \fint_0^1 z=0.\end{cases}\]
In general it is not true that $(K-\alpha)^\#=K^\#-\alpha_\#$. However, we always have the inequality 
\[ (K-\alpha)^\#\prec  K^\#-\alpha_\#=K-\alpha_\#.\]See, for instance, \cite[Proposition 3]{alvino1991}. Thus, applying \eqref{Eq:TalentiLangford} with $f=M_0(K-\alpha-M_0)$ and $g=M_0(K^\#-\alpha_\#-M_0)$ yields 
\[z\prec \hat v_{\alpha_\#},\] whence 
\[ \hat v_\alpha \prec\hat v_{\alpha_\#}.\] 
From the Hardy-Littlewood inequality \eqref{Eq:HL} and the definition of the order relation $\prec$ this gives 
\[ \fint_0^1 K\hat v_\alpha\leq \fint_0^1K^\# \hat v_{\alpha_\#}=\fint_0^1 K\hat v_{\alpha_\#}.\]

\paragraph{Conclusion}
In conclusion, we have established that 
\[ J^1(\alpha)=-2C_0\mathcal E_\alpha(\hat v_\alpha)+\fint_0^1 K\hat v_\alpha\leq -2C_0\mathcal E_{\alpha_\#}(\hat v_{\alpha_\#})+\fint_0^1 K\hat v_{\alpha_\#}=J^1(\alpha_\#),\] whence the conclusion. To guarantee uniqueness, it suffices to check that equality holds in the  P\'{o}lya-Szeg\"{o} inequality if and only if $\hat v_\alpha=\hat v_\alpha^\#$ or $\hat v_\alpha=(\hat v_\alpha)_{\#}$. This, however, follows from \cite{Ferone2003}. We then conclude that either $\alpha=\alpha^\#$ or $\alpha=\alpha_\#$. However, as $K=K^\#$ is not constant, the only possibility to also achieve equality in the Hardy-Littlewood inequality is to have $\alpha=\alpha_\#$.

\end{proof}

It remains to prove the Talenti inequality \eqref{Eq:TalentiLangford}. As we said, the proof can be quickly derived from the considerations of Langford in \cite{Langford}. For the sake of completeness, we sketch the details of the proof of \cite{Langford} here.

\begin{proof}[Proof of the Talenti inequality \eqref{Eq:TalentiLangford}]

We may extend $f$  by parity to an  even function (still denoted $f$ in the following)  on $(-1;1)$. Similarly, since $u_f$ satisfies Neumann boundary conditions at 0, it may be extended by parity to  $(-1;1)$. Thus extended, the function $u_f$ satisfies
\begin{equation}
\begin{cases}
-(u_f)''=f&\text{ in } (-1,1)\,, 
\\ u_f'(-1)=u_f'(1)=0\,, 
\\ \fint_{-1}^1 u_f=0.\end{cases}\end{equation} Furthermore, by parity of $f$, we have 
\[
\fint_{-1}^1 \xi f(\xi)d\xi=0.\] We will establish a comparison inequality on this new problem. To this end, let us introduce the fundamental solution of the Neumann Laplacian on $(-1,1)$. It is the function $K$ defined by 
\[
G(x):=\frac12x^2-|x|+\frac13.
\] We extend it to $\R$ by $2$-periodicity.
Consequently (see \cite[Proposition 5.2]{Langford}), we have  an explicit formula for $u_f$:
\[ u_f=G\star f:x\mapsto \fint_{-1}^1 K(x-y)f(y)dy.
\]
Furthermore, $G$ is, on $(0,1)$, a decreasing function, and so it is equal to its decreasing rearrangement. Now, by the Riesz convolution inequality \cite[Theorem 1]{Baernstein_1989}), for any $E\subset (-1,1)$ of volume $2r<2$,
\[
\fint_E u_f=\fint_{-1}^1 \mathds 1_{E}\left(G\star f\right)\leq \fint_{-1}^1 \mathds 1_{(-r,r)}\left(K\star f^*\right)\leq \fint_{-1}^1 \mathds 1_{(-r;r)} (G\star g^\#)=\fint_{-r}^r u_{g^\#},
\]
which gives the required result. However, $u_{g^\#}$ which is, a priori, only defined on $(-1,1)$, is necessarily symmetric with respect to 0, as $g^\#$ is. Thus, its restriction to $(0,1)$ is the solution of the Neumann problem with datum $g^\#$. This concludes the proof.
\end{proof}

\end{proof}

We note that the computations we carried out in the course of proving this theorem also provides an efficient way to reach Proposition \ref{Pr:K0/3}.

\begin{proof}[Proof of Proposition \ref{Pr:K0/3}]
We use the optimality conditions for the problem \eqref{Pv:1EqLarge2}. Note that $\overline \alpha$ is an interior point of $\mathcal M_=(\kappa,V_0)$, so that $\overline \alpha$ is a critical point if, and only if, for any admissible perturbation $h$ at $\overline \alpha$, 
\[ \dot J^1(\overline \alpha)[h]=0.\]

Adapting the proof of Lemma \ref{Le:GD} we obtain the following expression for the first order Gateaux-Derivative of $J^1$ at  any $\alpha $ in any admissible (at $\alpha$) direction $h$:
\[\dot J^1(\overline \alpha)[h]=\underbrace{2\left(\frac{2V_0-K_0}{M_0^2}\right)}_{=:C_1}\fint_\O \langle \n \hat v_\alpha\,, \n \dot{\hat v}_\alpha\rangle+\fint_\O K\dot{\hat v}_\alpha\text{ with }\begin{cases}-\Delta \dot{\hat v}_\alpha+M_0h=0&\text{ in }\O\,, 
\\ \frac{\partial \dot{\hat v}_\alpha}{\partial \nu}=0&\text{ on }\partial \O\,, 
\\ \fint\O \dot{\hat v}_\alpha=0.\end{cases}
\]
Introduce the adjoint state $q$ as the solution of 
\begin{equation}
\begin{cases}
-\Delta q=K-{K_0}{}&\text{ in }\O\,, 
\\\frac{\partial q}{\partial \nu}=0&\text{ on }\partial \O,
\\ \fint_\O q=0.\end{cases}\end{equation} This allows to rewrite $\dot{J}^1(\alpha)[h]$ as \begin{align*}
\dot{ J}^1(\overline \alpha)[h]&=C_1\fint_\O \langle \n \hat v_\alpha,\n \dot{\hat v}_\alpha \rangle+\fint_\O K\dot{\hat v}_\alpha
\\&=C_1\fint_\O \hat v_\alpha h-M_0\fint_\O qh.
\end{align*}
Thus, if $\overline \alpha$ is a critical point of $J^1$, we must have
\[ C_1\hat v_{\overline \alpha}-M_0q=\mu\] where $\mu$ is a real constant.
Taking the Laplacian on both sides of this equality, this implies that 
\[ C_1M_0(K-\overline \alpha-M_0)=K-{K_0}{}.\]
However,
\begin{align*}
C_1 M_0(K-\overline \alpha-M_0)=K-{K_0}{}&\Leftrightarrow K-\overline \alpha-M_0=\frac{K-{K_0}}{C_1 M_0}
\\&\Leftrightarrow \left(K-{K_0}{}\right)\left(1-\frac1{C_1 M_0}\right)=0.
\end{align*}
We develop
\[ 1-\frac1{C_1 M_0}=1-\frac{M_0}{2{K_0}-4\overline \alpha}=\frac{2{K_0}-4\overline\alpha-{K_0}+\overline \alpha}{2{K_0}-4\overline \alpha}=\frac{{K_0}-3\overline \alpha}{2{K_0}-4\overline \alpha}.
\] and we thus derive the conclusion: for $\overline \alpha$ to be a critical point, we must either have $K$ constant, or 
\[ K_0-3V_0=0.\]

\end{proof}

\section{Proofs of Theorems \ref{Th:NashExists} and  \ref{Th:NashLarge}}\label{Se:3}
\begin{proof}[Proof of Theorem \ref{Th:NashExists}]

The proofs of Theorems \ref{Th:NashExists} follow in an almost straightforward manner from the previous considerations on single player games.

Indeed, observe the following fact: from Theorem \ref{Th:Concave}, $\mu>0$ being fixed, in the one-dimensional case, there exists $\overline \delta_1\,, \overline \delta_2>0$ such that
\[ V_1<\overline \delta_1\,, V_2<\overline \delta_2 \Rightarrow \forall \alpha_2 \in \mathcal M_=(\kappa_2,V_2)\,,  I_{1,\mu}(\cdot,\alpha_2) \text{ is concave in $\alpha_1$},\]and
\[ V_1<\overline \delta_1\,, V_2<\overline \delta_2 \Rightarrow \forall \alpha_1 \in \mathcal M_=(\kappa_1,V_1)\,,  I_{2,\mu}(\alpha_1,\cdot) \text{ is concave in $\alpha_2$}.\] Indeed, it suffices, for the concavity of $I_{1,\mu}$, to apply Theorem \ref{Th:Concave} with $K-\alpha_2$ as a resources distribution, and similarly for the concavity of $I_{2,\mu}$.

Similarly, in any dimension $d$, we obtain $\delta_1\,, \delta_2>0$ such that, if 
\[ V_1+V_2\leq \delta_1\,, \Vert K-\overline K\Vert_{L^1(\O)}\leq \delta_2\] then  the maps $I_{1,\mu}(\cdot,\alpha_2)$ and $I_{2,\mu}(\alpha_1,\cdot)$ are concave in their respective variables.

So what matters about the assumptions of smallness of $V_1,V_2$ (and $\Vert K-\overline K\Vert_{L^1(\O)}$) is that the functionals for which we are seeking a Nash equilibrium are concave. The rest of the proof does not depend in any way on the dimension.

This concavity property is the natural one in the context of existence of Nash equilibria. Indeed, let us recall \cite{Glicksberg_1952,Nash1951}: if $\Delta_i\subset \R^{d}$ ($i=1,2$) is a convex, compact set, and if $L_i=\Delta_1\times \Delta_2\to \R$ is a concave, continuous function ($i=1,2$) then the game
\[\text{ find $x_i^*\in \Delta_i$ ($i=1,2$) such that } \begin{cases}L_1(x_1^*,x_2^*)=\max_{x_1\in \Delta_1}L_1(x_1,x_2^*)\,,\\
L_2(x_1^*,x_2^*)=\max_{x_2\in \Delta_2}L_2(x_1^*,x_2)\end{cases}
\]
has a Nash equilibrium $(x_1^*,x_2^*)$.

To apply this result to the situation under investigation in the present paper, we need to approximate our infinite dimensional problem by a finite dimensional one.

\paragraph{Reduction to the finite-dimensional setting}
Let us explain how this reduction is carried out: consider, a fixed integer $N$ being fixed, a measurable partition of $\O$ as 
\begin{equation}\label{Eq:Partition}\O=\sqcup_{k=0}^{n(N)}\omega_{k,N}\end{equation} where, for any $k\in \{0,\dots,n(N)\}$ we have 
\[ \mathrm{Vol}(\omega_{k,N})\leq 2^{-N}.\]

We consider the auxiliary admissible sets
\[ 
\mathcal M^N_{=}(\kappa_i,V_i):=\left\{\sum_{k=0}^N a_k\mathds 1_{\omega_{k,N}}\,, 0\leq a_k\leq \kappa_i\,,\sum_{k=0}^{n(N)} a_k\mathrm{Vol}(\omega_{k,N})= V_i\right\} \quad (i=1,2),
\]and we define $\Delta_i:=\mathcal M^N_=(\kappa_i,V_i)$ ($i=1,2$).

Of course, for any $\alpha_2\in \Delta_2$, the map $I_{1,\mu}(\cdot,\alpha_2)$ is concave on $\Delta_1$ and, similarly,  for any $\alpha_1\in \Delta_1$, the map $I_{2,\mu}(\alpha_1,\cdot)$ is concave on $\Delta_2$. The continuity of $I_{i,\mu}$ ($i=1,2$) on $\Delta_1\times \Delta_2$ is obvious. Thus, by the existence theorem for pure Nash equilibria, we conclude that there exists a Nash equilibrium $(\alpha_{1,N}^*,\alpha_{2,N}^*)$ for the problem
\[\text{ find $\alpha_i^*\in \Delta_i$ ($i=1,2$) such that } \begin{cases}I_{1,\mu}(\alpha_1^*,\alpha_2^*)=\max_{\alpha_1\in \Delta_1}I_{1,\mu}(\alpha_1,\alpha_2^*)\,,\\
I_{2,\mu}(\alpha_1^*,\alpha_2^*)=\max_{\alpha_2\in \Delta_2}I_{2,\mu}(\alpha_1^*,\alpha_2).\end{cases}
\]

\paragraph{Conclusion of the proof}
We fix, for any $N\in \N$, a Nash equilibrium $(\alpha_{1,N}^*,\alpha_{2,N}^*)$. Up to a (non-relabelled) subsequence, there exists a couple $(\alpha_1^*,\alpha_2^*)\in \mathcal M_=(\kappa_1,V_1)\times \mathcal M_=(\kappa_2,V_2)$ such that 
\[ \alpha_{i,N}^*\underset{N\to \infty}\rightharpoonup \alpha_i^*\quad (i=1,2)\] where the convergence holds weakly in $L^\infty$-*.

However, this weak convergence implies that, weakly in $W^{2,2}(\O)$ (in particular, strongly in $L^2(\O)$), there holds 
\[ \theta_{\alpha_{1,N}^*,\alpha_{2,N}^*,\mu}\underset{N\to \infty}\rightarrow \theta_{\alpha_1^*,\alpha_2^*}.\] Let us check that $(\alpha_1^*,\alpha_2^*)$ is a Nash equilibrium for our initial problem.

To this end, let $\alpha_1\in \mathcal M_=(\kappa_1,V_1)$ and let us prove that 
\[ I_{1,\mu}(\alpha_1^*,\alpha_2^*)\geq I_{1,\mu}({\alpha_1,\alpha_2^*}).\] 
By \eqref{Eq:Partition}, there exists a sequence $\{\alpha_{1,N}\}_{N\in \N}$ such that, for any $N\in \N$, $\alpha_{1,N}\in \mathcal M_=^N(\kappa_1,V_1)$ and such that, strongly in $L^1(\O)$, 
\[ \alpha_{1,N}\underset{N\to \infty}\rightarrow \alpha_1.\]
By definition of $\alpha_{1,N}^*$ we have, for any $N\in \N$, 
\[ I_{1,\mu}(\alpha_{1,N}^*,\alpha_{2,N}^*)\geq I_{1,\mu}(\alpha_{1,N},\alpha_{2,N}^*).
\]
Passing to the limit as $N\to \infty$ we obtain 
\[I_{1,\mu}(\alpha_1^*,\alpha_2^*)\geq I_{1,\mu}(\alpha_1,\alpha_2^*).
\]
As the symmetric property for $I_2$ (\emph{i.e} that $\alpha_2^*$ is a maximiser of $I_{2,\mu}(\alpha_1^*,\cdot)$ over $\mathcal M_=(\kappa_2,V_2)$) is proved in the very same fashion we omit it here.
The conclusion follows: $(\alpha_1^*,\alpha_2^*)$ is indeed a Nash equilibrium.
\end{proof}

\begin{proof}[Proof of Theorem \ref{Th:NashLarge}]
The proof of this theorem follows from Theorem \ref{Th:AsymptoticSingle}. Indeed, observe that, if $V_1\,, V_2\geq \frac{K_0}4$, then we have 
\[ V_1\geq \frac{K_0-V_2}2\,, V_2\geq \frac{K_0-V_1}2.\] Consequently, for any $\alpha_2\in \mathcal M_=(\kappa_2,V_2)$, it follows from Theorem \ref{Th:AsymptoticSingle} that, for any $\alpha_2\in \mathcal M_=(\kappa_2,V_2)$ fixed, the map
\[ \alpha_1\mapsto I_1^1(\alpha_1,\alpha_2)\] is strictly convex on $\mathcal M_=(\kappa_1,V_1)$. Similarly, for any fixed $\alpha_1\in \mathcal M_=(\kappa_1,V_1)$, the map
\[ \alpha_2\mapsto I_1^2(\alpha_1,\alpha_2)\] is strictly convex on $\mathcal M_=(\kappa_2,V_2)$.

Now let us take $\alpha_1^*\,, \alpha_2^*$ as defined in the statement of the theorem. We apply Theorem \ref{Th:AsymptoticSingle}: taking as a resources distribution $K=K_0-\alpha_2$, which is a non-constant, non-decreasing function,  we deduce (from the convexity of the functional) that any solution of 
\[ \max_{\mathcal M_=(\kappa_1,V_1)}I_1^1(\alpha_1,\alpha_2^*)\] is a bang-bang function and (from Theorem \ref{Th:AsymptoticSingle}) that the solution is a non-increasing function. Thus, the solution is exactly $\alpha_1^*$.

Similarly, $\alpha_2^*$ is the solution of 
\[ \max_{\mathcal M_=(\kappa_2,V_2)}I_2^1(\alpha_1^*,\alpha_2).\] Thus, $(\alpha_1^*,\alpha_2^*)$ is indeed a Nash equilibrium in the sense of Definition \ref{De:NashAsymptotic}.
\end{proof}
\section{Numerical simulations and comments}\label{Se:NumericsSingleCommented}

\subsection{Simulations of the optimal harvesting problem}
{
In this section we  consider a random\footnote{In the sense that we pick random Fourier coefficients}  positive continuous function $K:[0,1]^2\mapsto\mathbb{R}$ represented in Figure \ref{Fig:capacity} and we consider the optimal harvesting problem \eqref{Pv:1Ineq}. Figure \ref{Fig:optimalH} exemplifies the richness of different qualitative behaviours  this simple problem can exhibit under modifications of the constraints. More specifically we notice the following facts:
\begin{enumerate}
\item The switch function $\theta(1-p)$, can be constant and hence the optimal controls are not necessarily bang-bang. This is a case already emphasised in the particular case in which $K(x)=1$. 
\item For certain parameters, the switch function does not have any flat region and it is uniformly positive. In this case, we have a bang-bang strategies due to the well known bathtub principle: the optimal policy is a characteristic function $\kappa\mathbbm{1}_\omega,$ where $\omega$ is the level set of the switch function with volume $\frac{ V_0}{\kappa}$.
\item The switch function can combine both aspects, it can have flat region and a nonflat one. This is the case of the third column in Figure \ref{Fig:optimalH}, where a qualitative mixture of the phenomenology described in the previous two points is observed. In this case, one observes that the flat region is at the maximum of the switch function. This is the reason why the optimal strategy does not saturate the upper bound $\alpha\leq \kappa$ and shows a non-bang-bang structure. In addition, as in the previous case, one can observe that the support of $\alpha$ is not the whole square $[0,1]^2$. This can be seen also with the same philosophy of the bathtub principle, observing that the optimal strategies have to be supported in a subset of a level set of the switch function. Indeed, if the integral constraint $V_0$ satisfied that $ V_0/\kappa=|\{x\in \mathbb{R}^2:\quad \theta(x)(1-p(x))=\max_{x} \theta(x)(1-p(x))\}|$ we would nonetheless observe a bang-bang strategy. However, for this simulation, the above property is not satisfied and one has the $ V_0/\kappa$ is smaller than the volume of the level set corresponding to the maximum of the switch function. Hence, what one observes is $\mathrm{supp}(\alpha)\subset\{x\in \mathbb{R}^2:\quad \theta(x)(1-p(x))=\max_{x} \theta(x)(1-p(x))\}$.
\end{enumerate}

\begin{figure}
\centering
\includegraphics[scale=0.25]{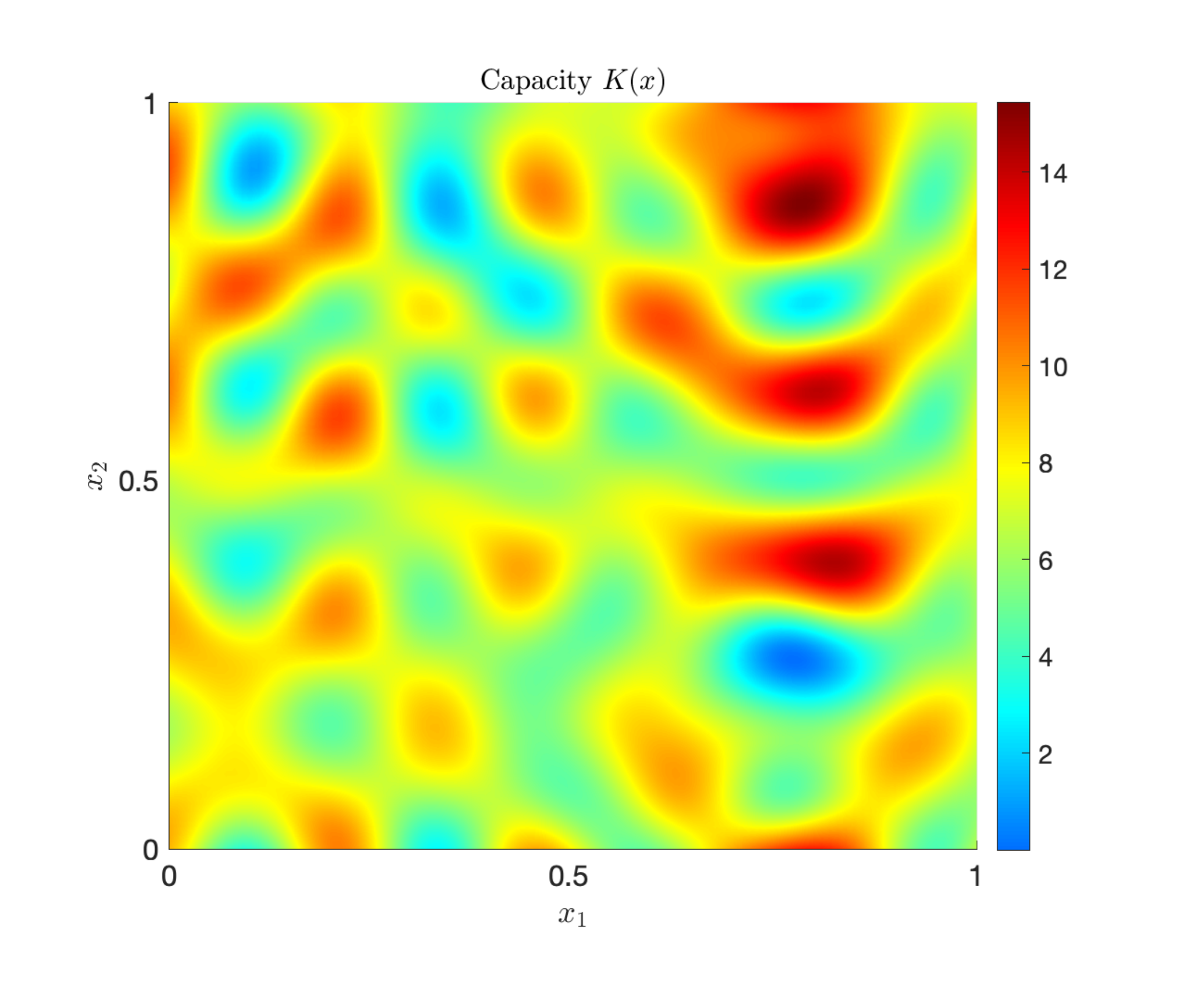}
\caption{Capacity $K(x)$ used for the simulations shown in Figure \ref{Fig:optimalH}.}\label{Fig:capacity}
\end{figure}

}

\begin{figure}
\centering
\includegraphics[scale=0.22]{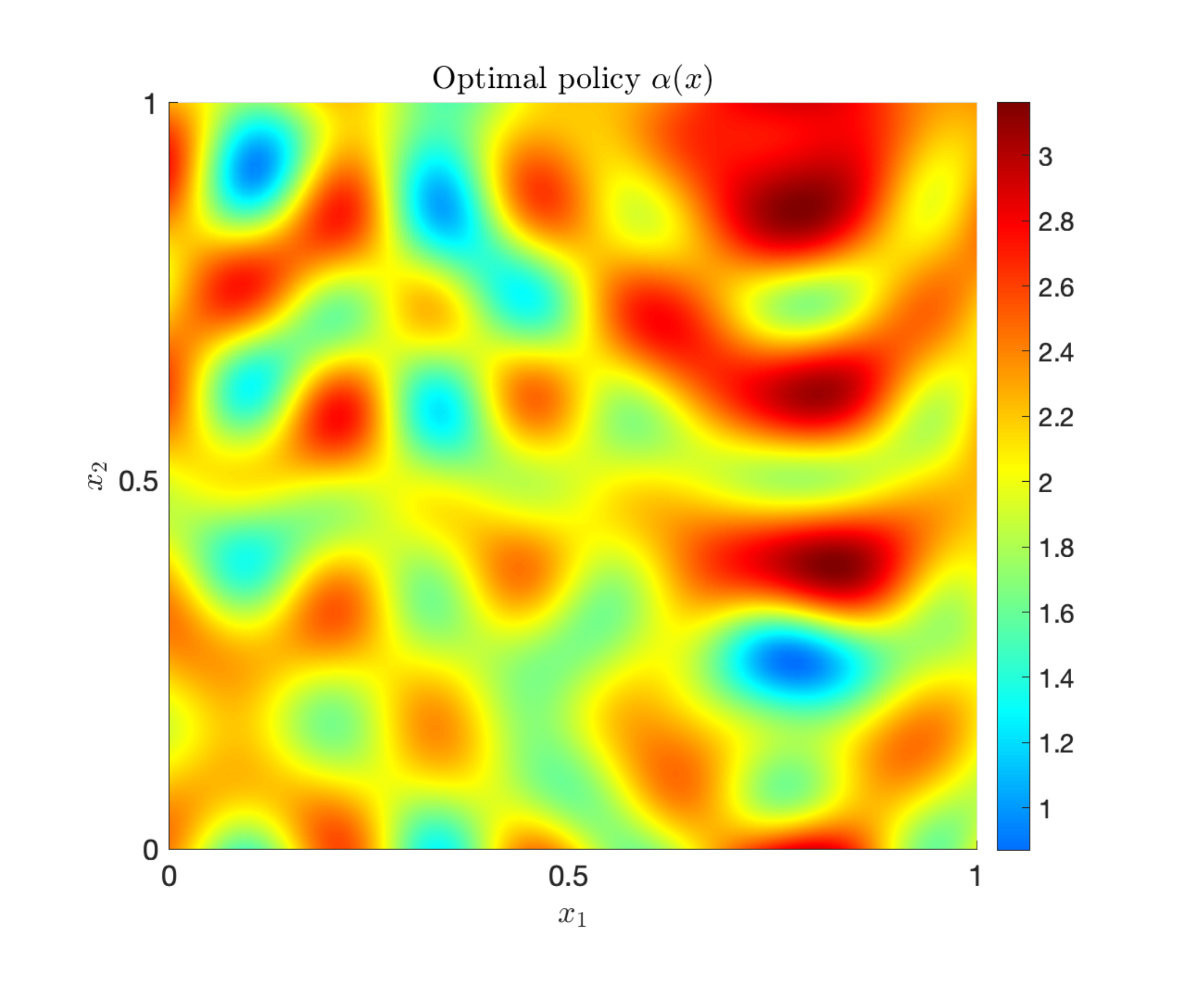}
\includegraphics[scale=0.22]{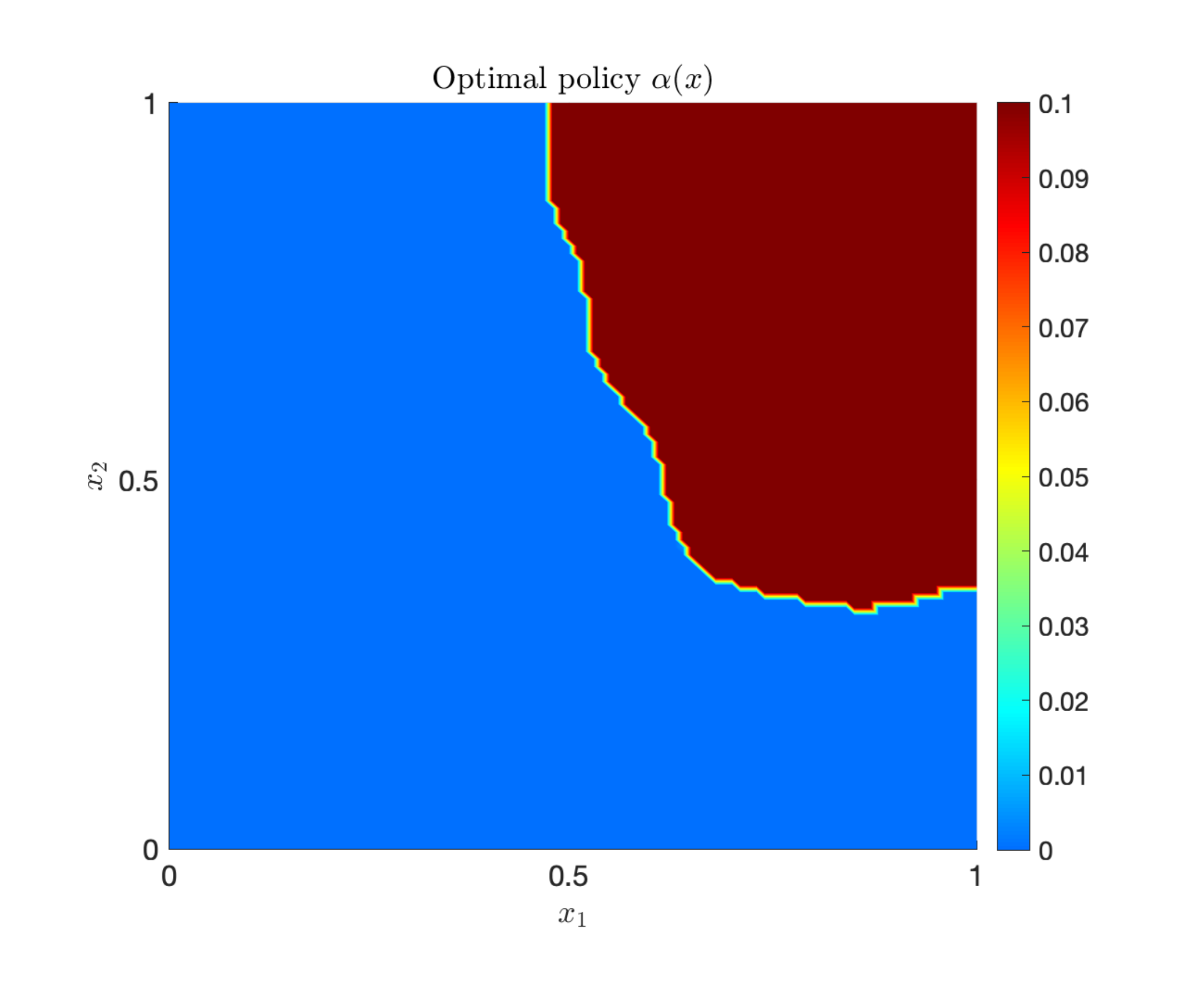}
\includegraphics[scale=0.22]{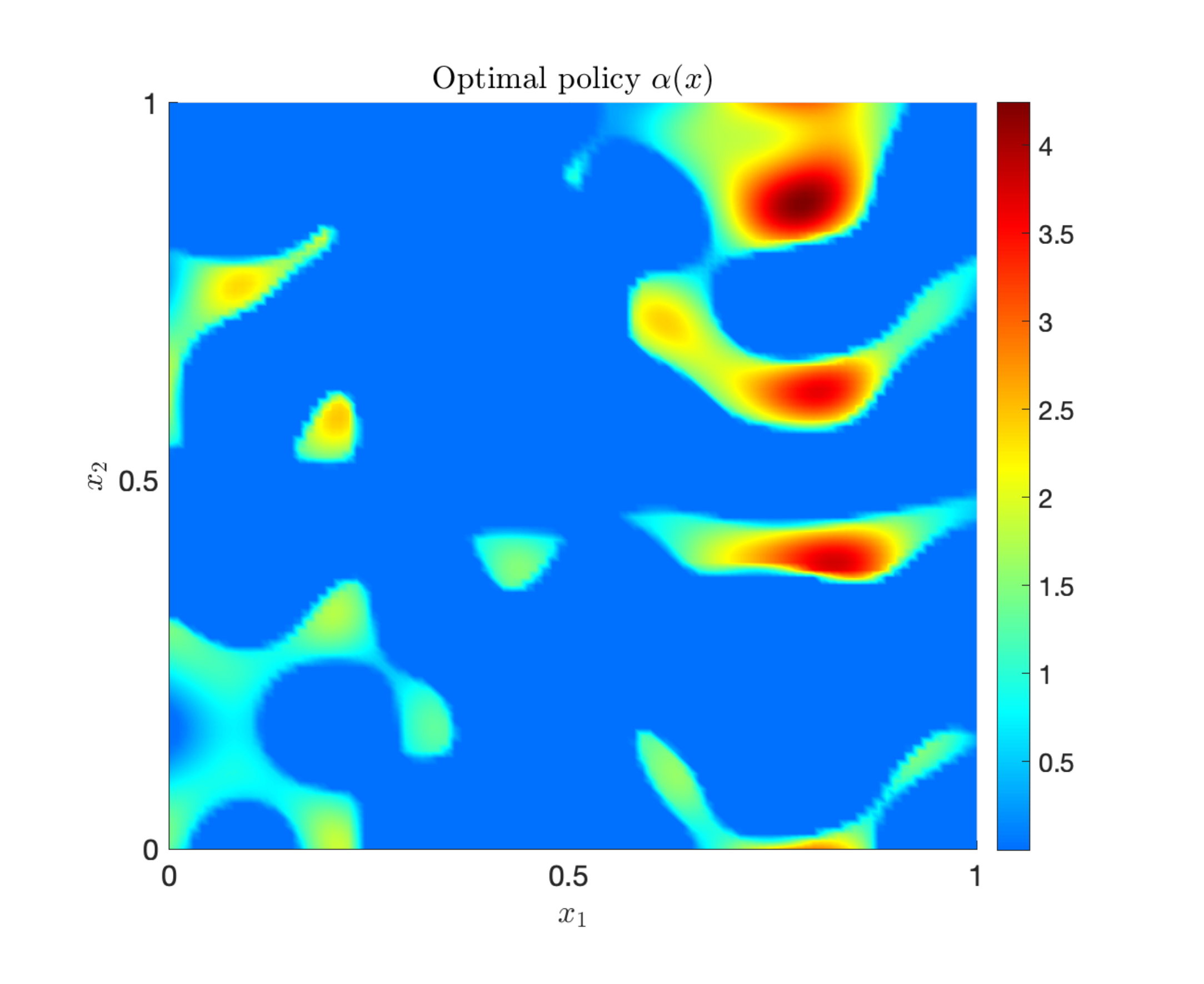}

\includegraphics[scale=0.22]{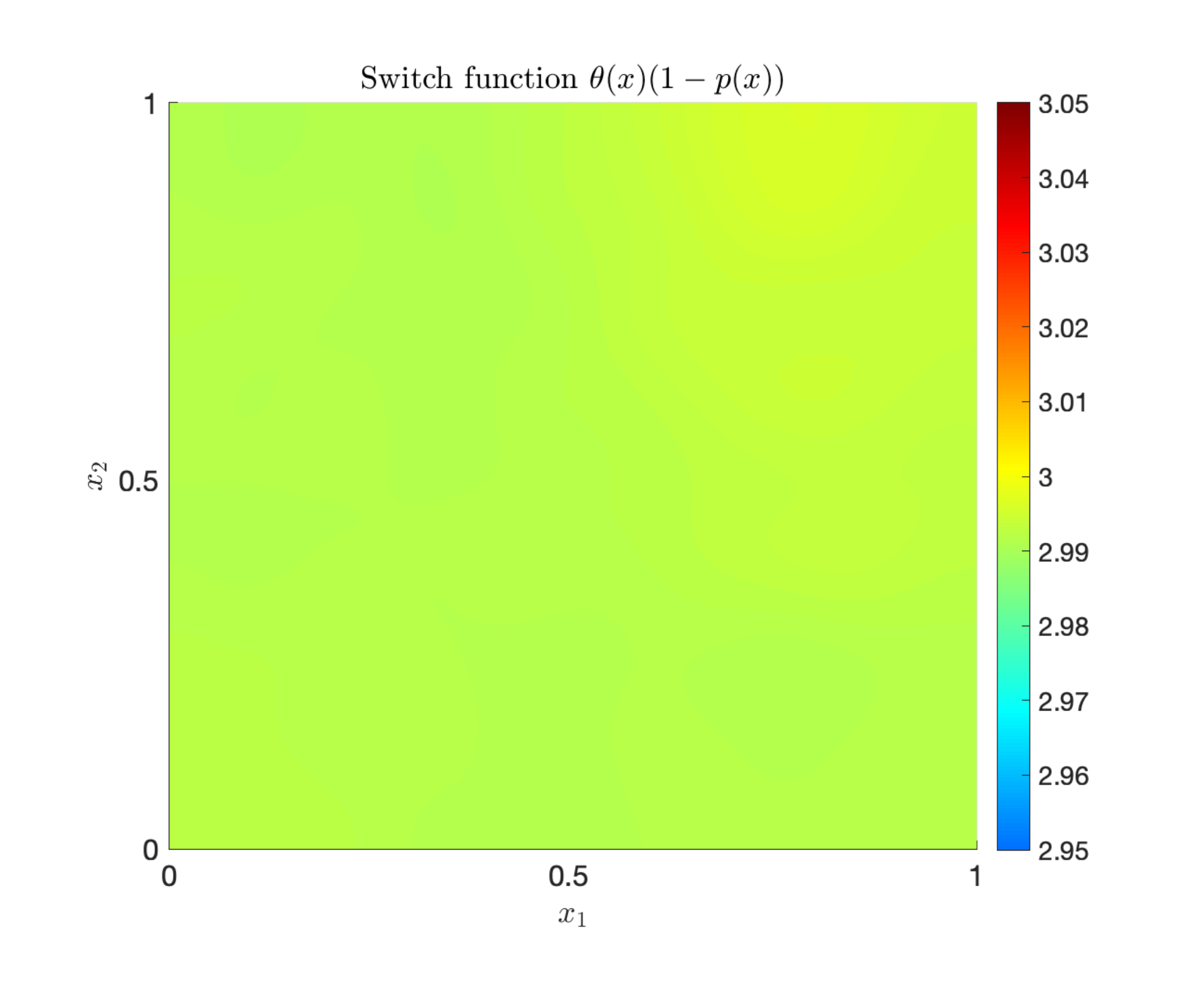}
\includegraphics[scale=0.22]{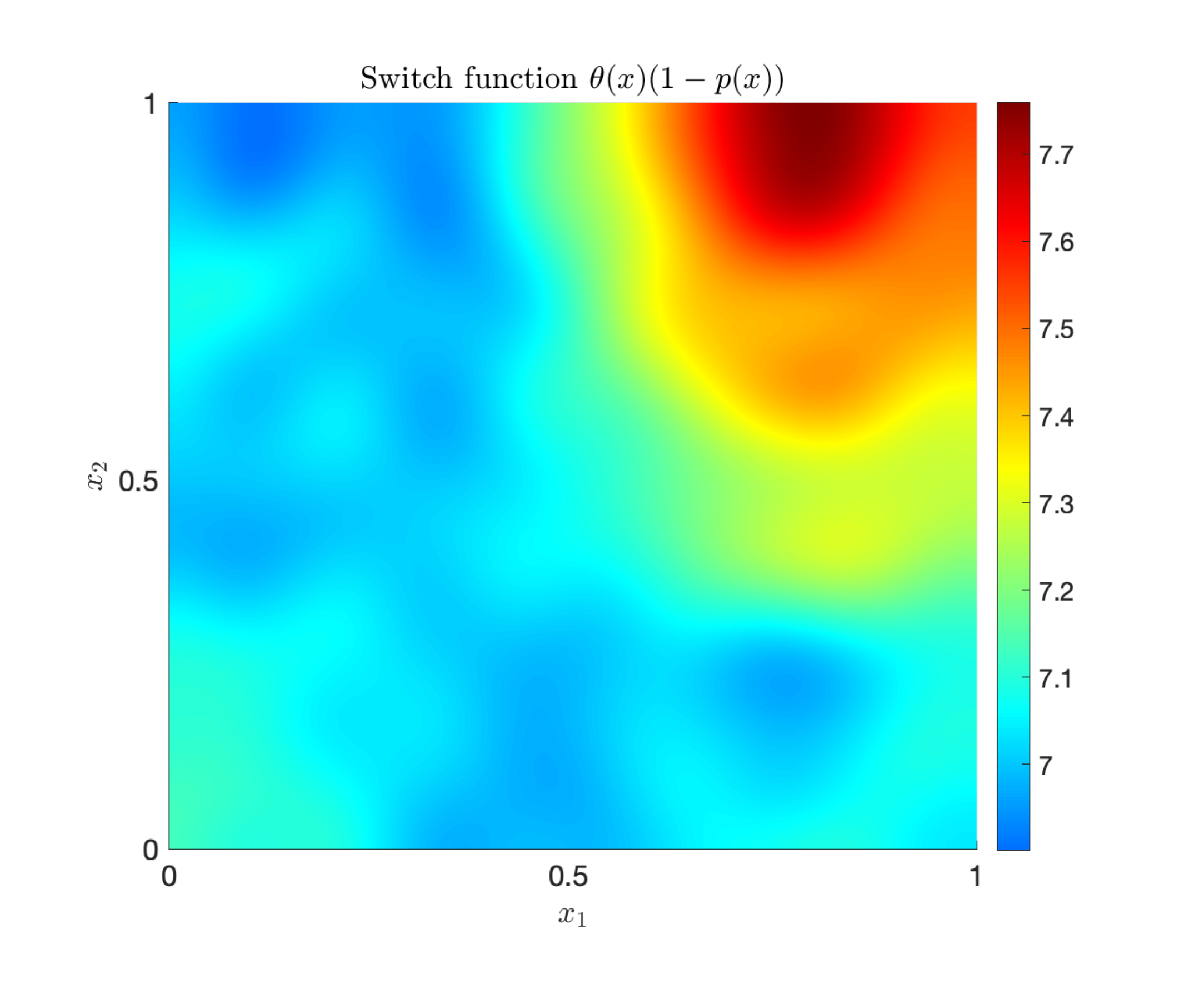}
\includegraphics[scale=0.22]{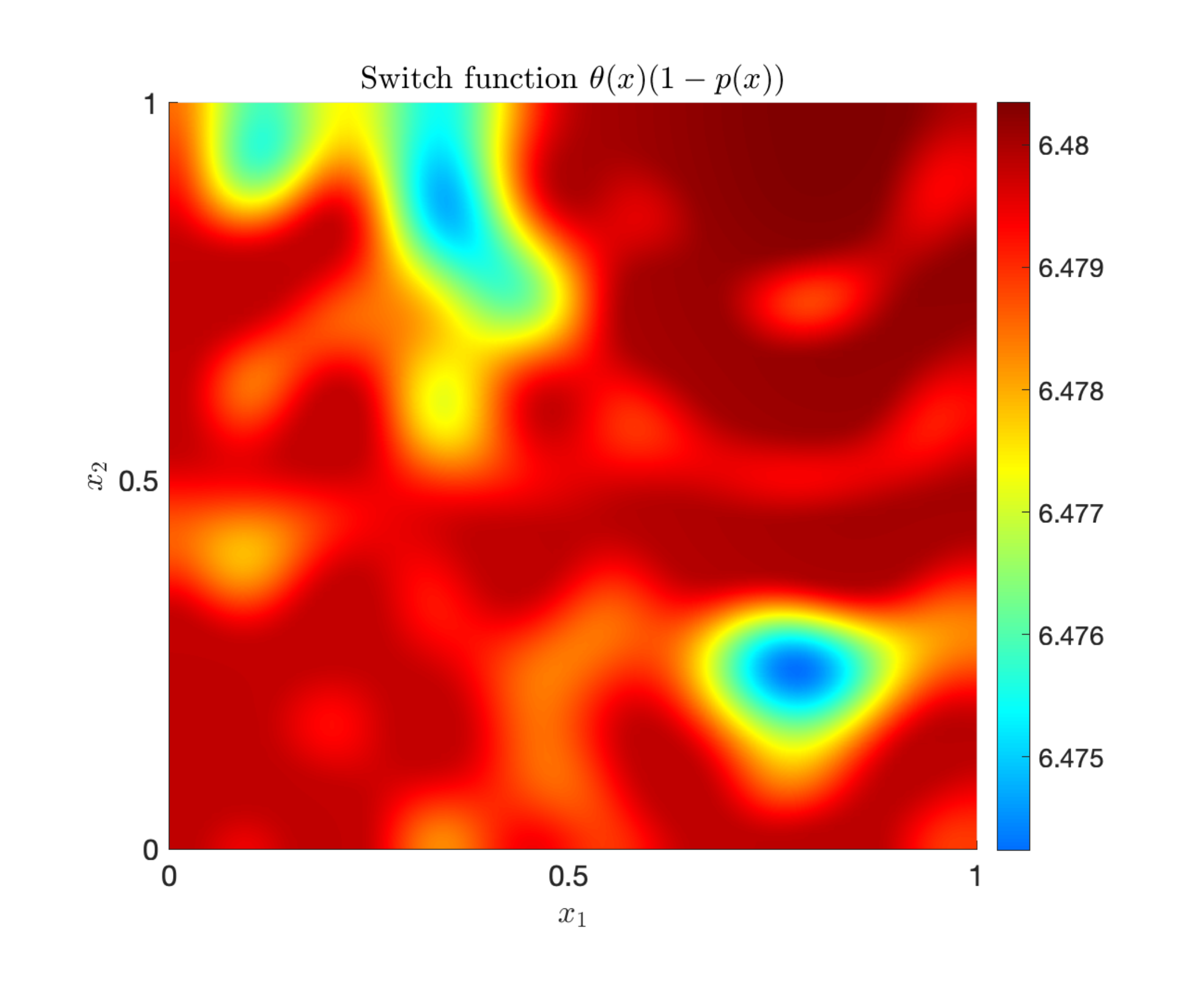}
\caption{In this figure, several optimal policies $\alpha$  (first row) are depicted along with the corresponding switch functions (second row). All the simulations have been done using the same capacity $K=K(x)$ of Figure \ref{Fig:capacity} and same diffusivity $\mu=1$ although with different control limitations. In the first and second column the integral constraint is $V_0=0.3$, but $\kappa=7$ in the first column while $\kappa=0.1$ in the second column. In the third column, $V_0=0.05$ and $\kappa=7$.}\label{Fig:optimalH}
\end{figure}

\subsection{Simulations of the Nash equilibria}

{
In this section, we  provide several numerical simulations that illustrate some of the phenomena described in the article. Moreover, it allows us to formulate open problems that may lead to further research on the topic. For all the simulations we have employed a fixed-point algorithm to find Nash equilibria. The algorithm used is the following
\begin{algo}\label{algo1}
\textcolor{white}{.}\newline
\vspace{-0.6cm}
\begin{enumerate}
 \item \textbf{Initialization}: Take a pair of strategies, $\alpha_1^{(0)},\alpha_2^{(0)}\in \mathcal{M}_\leq(\kappa,V_0)$.
 \item \textbf{Recursion}: For every $k\in \mathbb{N}$, solve sequentially the optimization problems
 \begin{equation*}
  \max_{\alpha_2^{(k)}\in\mathcal{M}_\leq (\kappa,V_0)} \int_{\Omega} \alpha_2^{(k)} \theta dx
 \quad\text{ restricted to }\quad
 \begin{cases}
    -\mu\Delta \theta= \theta(K(x)-\theta)-\alpha_1^{(k-1)}(x)\theta-\alpha_2^{(k)}(x)\theta,\\
  \text{+ Boundary conditions.}
 \end{cases}
 \end{equation*}
 and then
  \begin{equation*}
  \max_{\alpha_1^{(k)}\in\mathcal{M}_\leq(\kappa,V_0)} \int_{\Omega} \alpha_1^{(k)} \theta dx
\quad\text{ restricted to }
 \begin{cases}
    -\mu\Delta \theta= \theta(K(x)-\theta)-\alpha_1^{(k)}(x)\theta-\alpha_2^{(k-1)}(x)\theta,\\
  \text{+ Boundary conditions}.
 \end{cases}
 \end{equation*}
\end{enumerate}
\end{algo}
If the algorithm \eqref{algo1} converges, \emph{i.e.} if there holds
$$ \alpha_1^{(k)}\to \alpha_1^*\text{ and } \alpha_2^{(k)}\to \alpha_2^*$$
then the pair $(\alpha_1^*,\alpha_2^*)$ is a Nash equilibrium by definition. We do not have a proof of convergence of the above algorithm. The proof itself would imply the existence of Nash equilibria (but not the other way around). In the case of potential games, Algorithm \ref{algo1} always converges. However, it can be seen, using a contradiction argument, that our game is not a potential game.

 We will employ this algorithm numerically to try to discover if a Nash equilibrium exists and we will use the stopping condition $\|\alpha_i^{(k+1)}-\alpha_i^{(k)}\|_{L^2}\leq \texttt{tol}$ for $i=1,2$. Since the algorithm above is forced to stop given a tolerance, one cannot guarantee that the convergence is at a Nash equilibria, but rather at an $\epsilon$-Nash equilibria.

\begin{definition}[$\epsilon$-Nash equilibria]
Fix $\epsilon\geq 0$. A pair of strategies $(\alpha_1,\alpha_2)\in \mathcal{M}_\leq(\kappa,V_0)^2$ is an $\epsilon$-Nash equilibria if

$$ \forall \alpha\in \mathcal{M}_{\leq}(\kappa,V_0)\quad I_1(\alpha_1,\alpha_2)\geq I_1(\alpha,\alpha_2)-\epsilon,\qquad  \forall \alpha\in \mathcal{M}_{\leq}(\kappa,V_0)\quad I_2(\alpha_1,\alpha_2)\geq I_2(\alpha_1,\alpha)-\epsilon $$

\end{definition}

Note that if $\epsilon=0$ one has the definition of a Nash equilibria. Furthermore, it is important  to observe that an $\epsilon$-Nash equilibria (with $\epsilon>0$) \emph{does not need to be close to a Nash equilibria}. Moreover, it is worth noting that an $\epsilon$-Nash equilibria can exist without a Nash equilibria existing. If it converges, Algorithm \ref{algo1} converges to an $\epsilon$-Nash equilibria (see Proposition \ref{prop-enash} in the Appendix).

\subsubsection{Symmetric bounds for both players}

}

\begin{figure}
\centering
\includegraphics[scale=0.2]{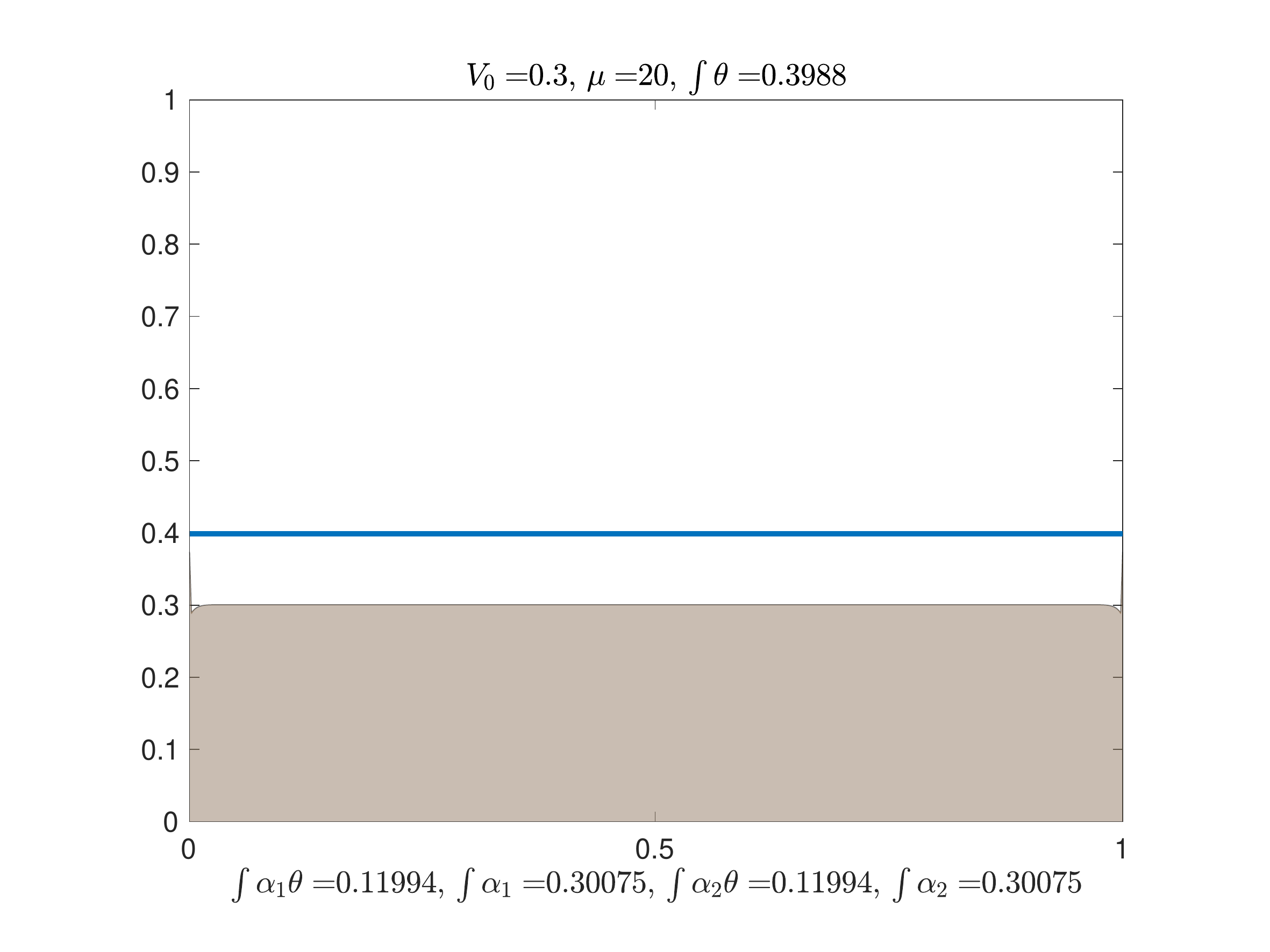}
\includegraphics[scale=0.2]{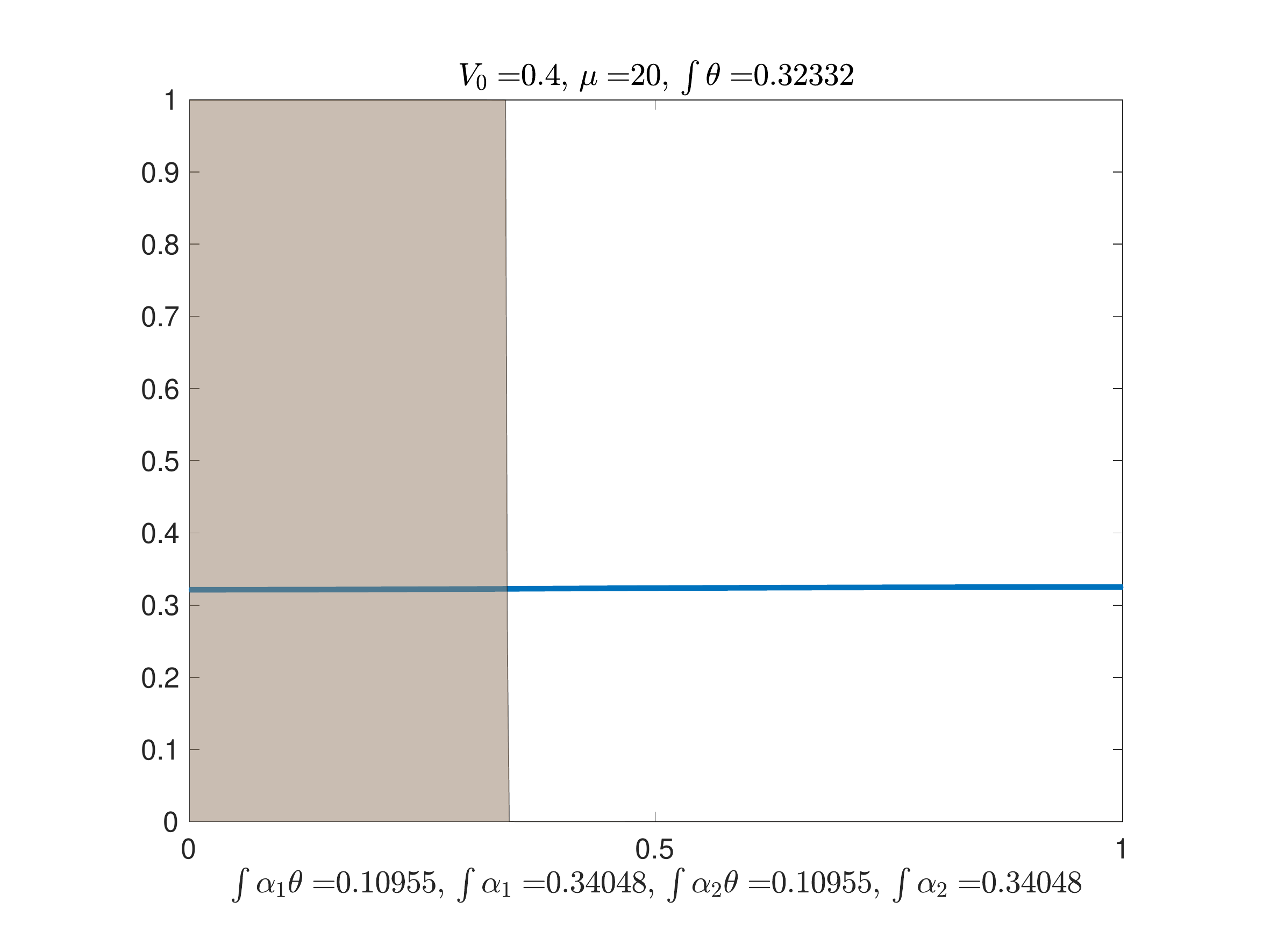}
\includegraphics[scale=0.2]{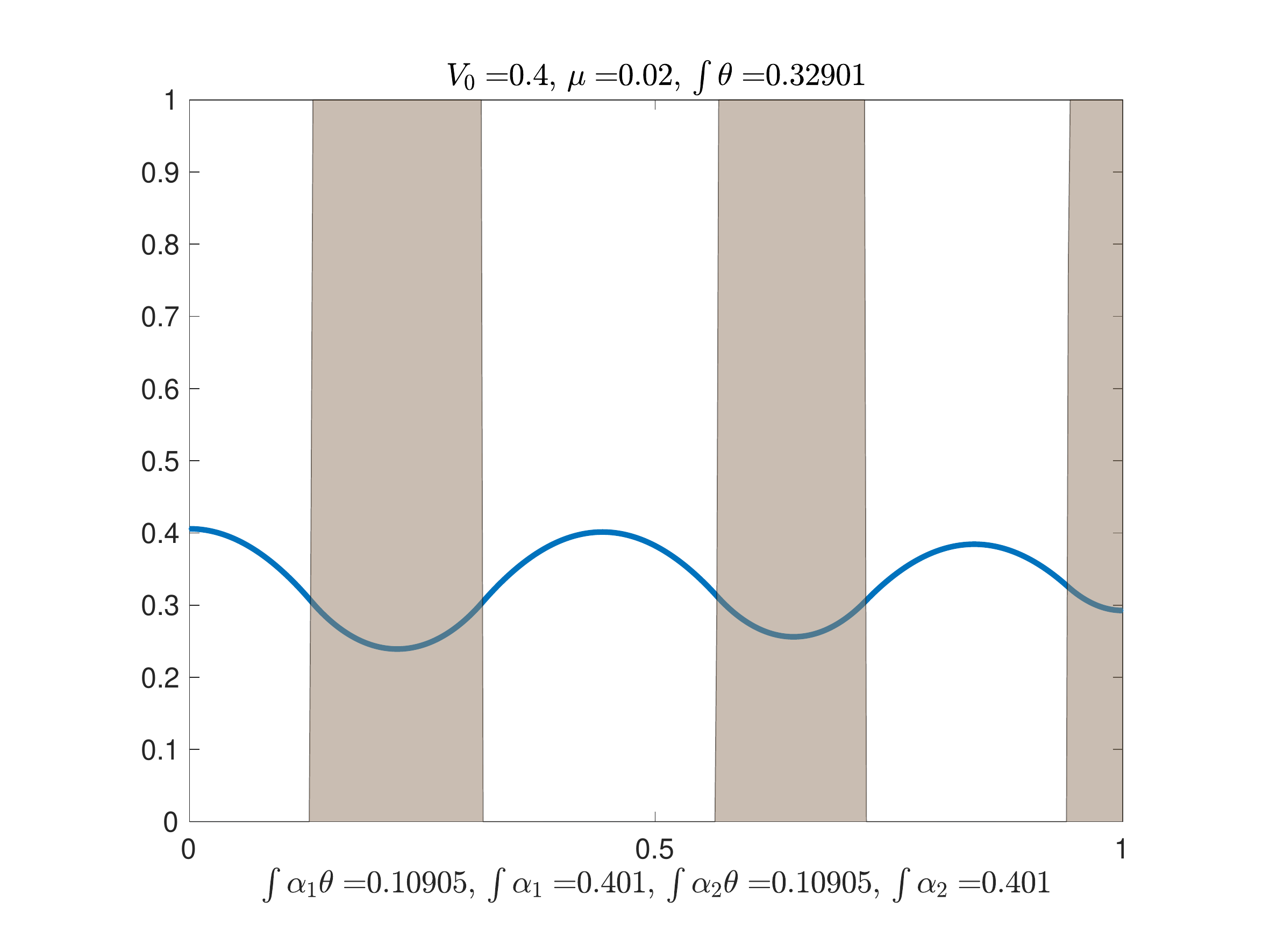}

\includegraphics[scale=0.2]{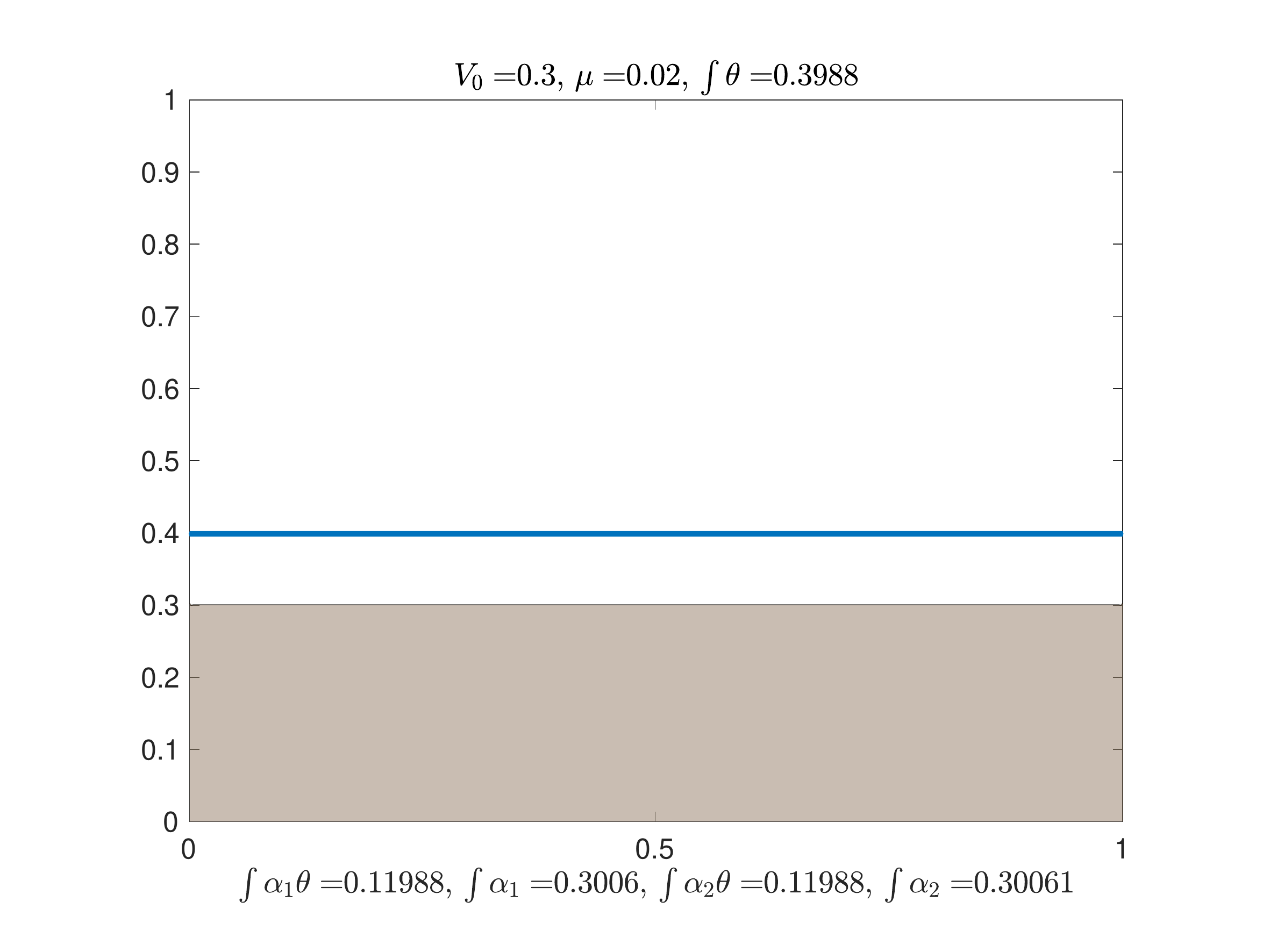}
\includegraphics[scale=0.2]{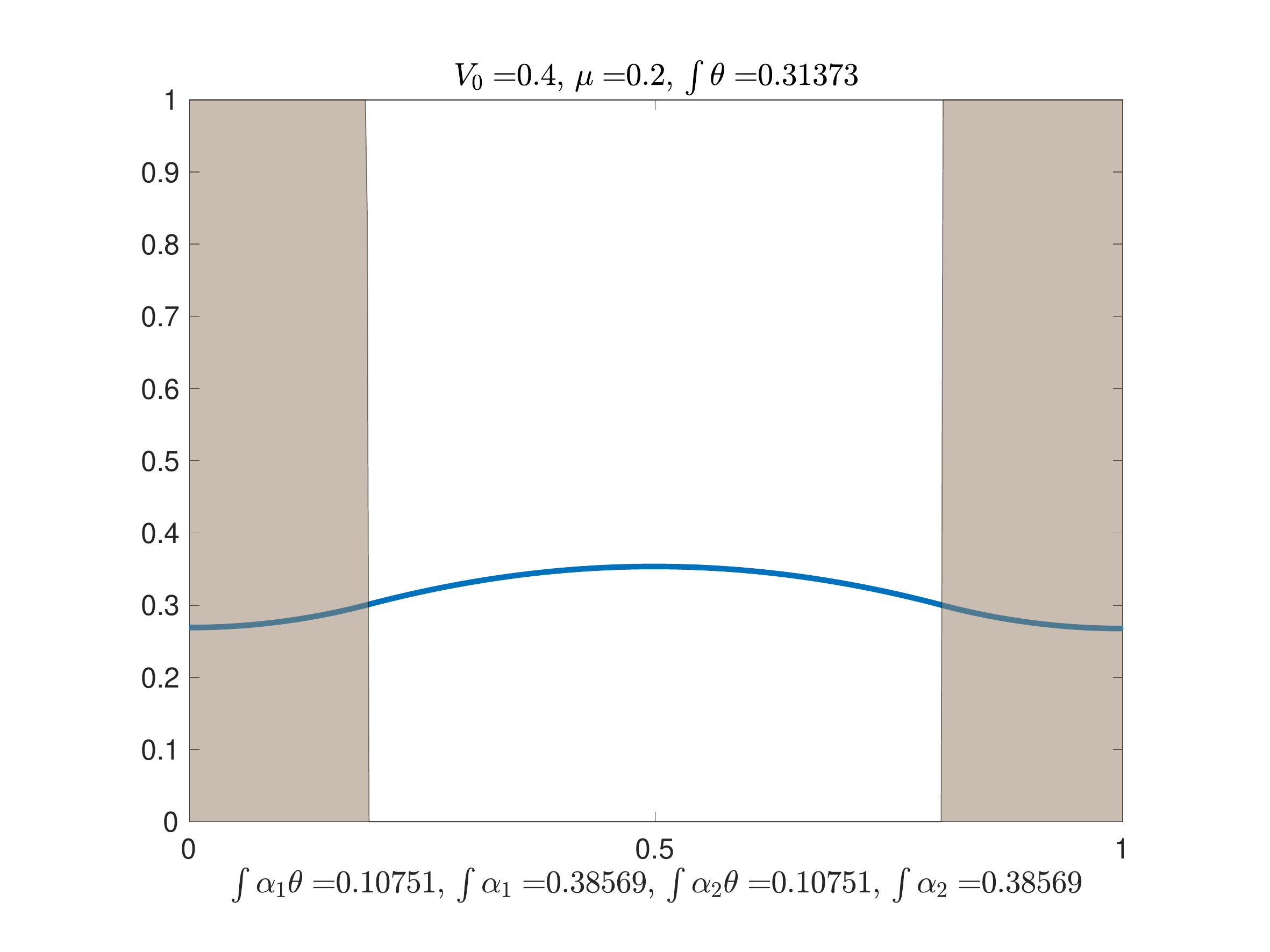}
\includegraphics[scale=0.2]{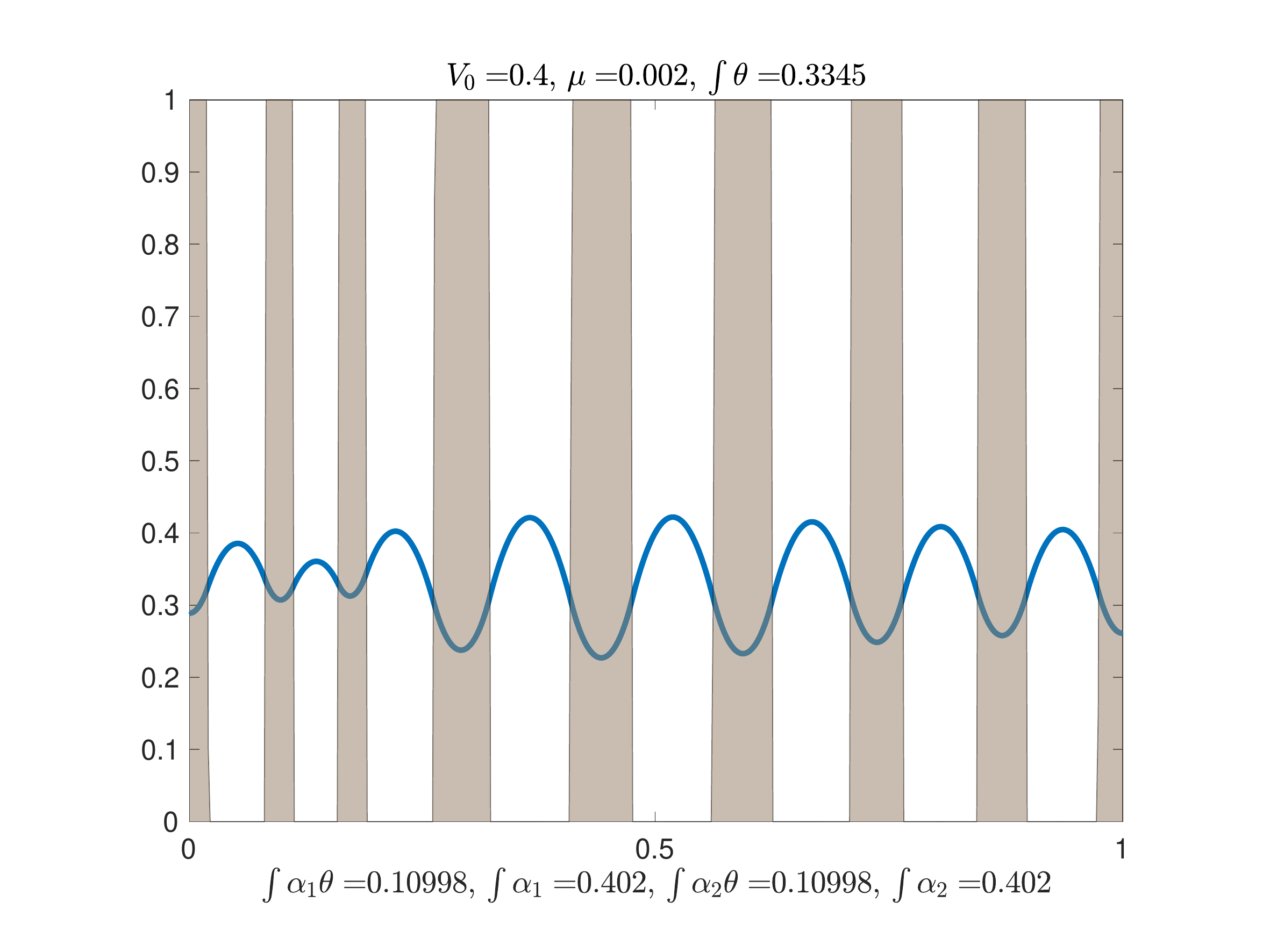}

\caption{The blue line is the state $\theta_{\alpha_1,\alpha_2}$. The grey area indicates the are the subgraphs of the strategy of the players (both players play the same strategy). Both players have the same capacity and $K(x)=1$.}\label{Fig:1dNash}

\end{figure}

\begin{figure}
\begin{center}
\includegraphics[scale=0.22]{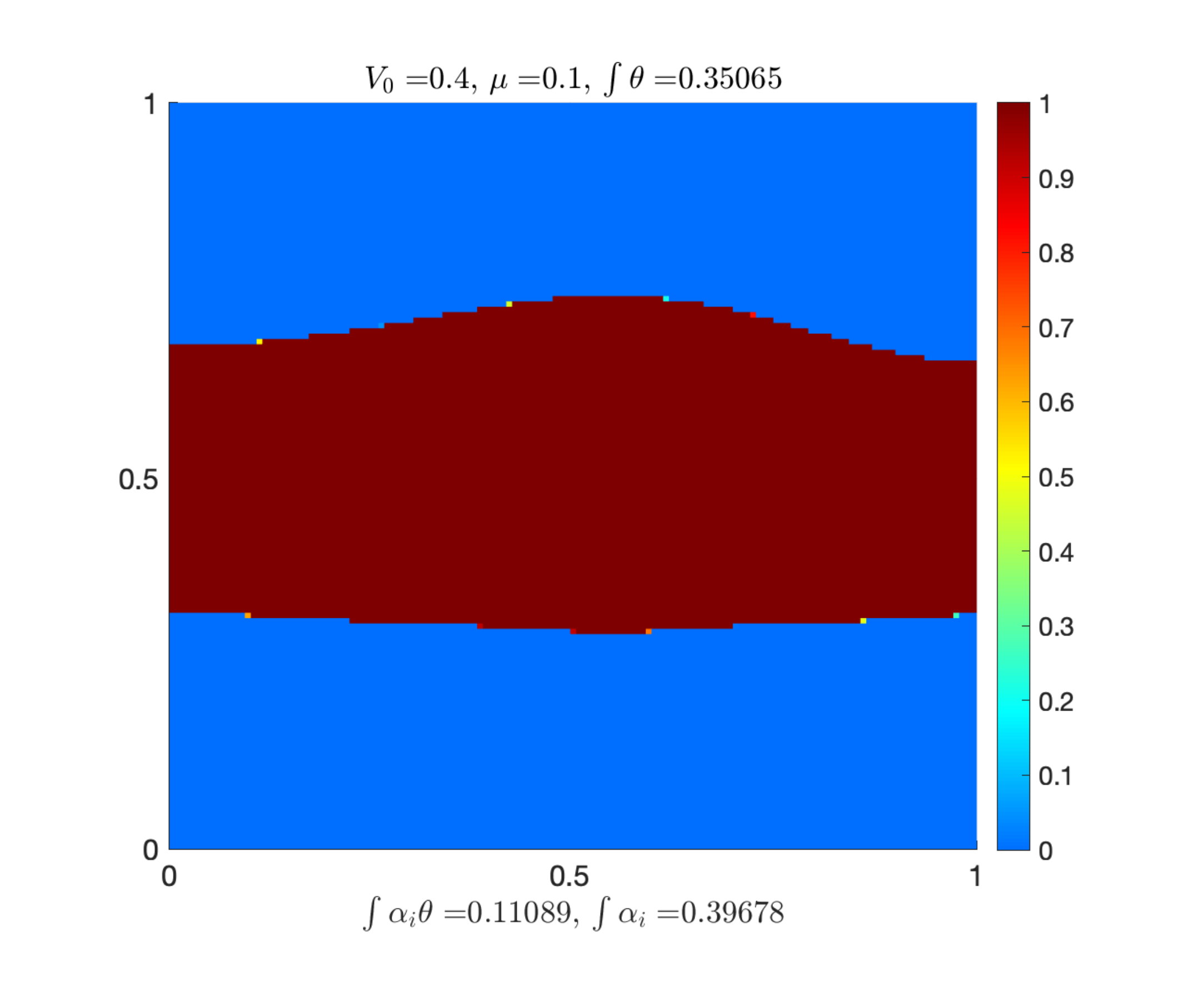}
\includegraphics[scale=0.22]{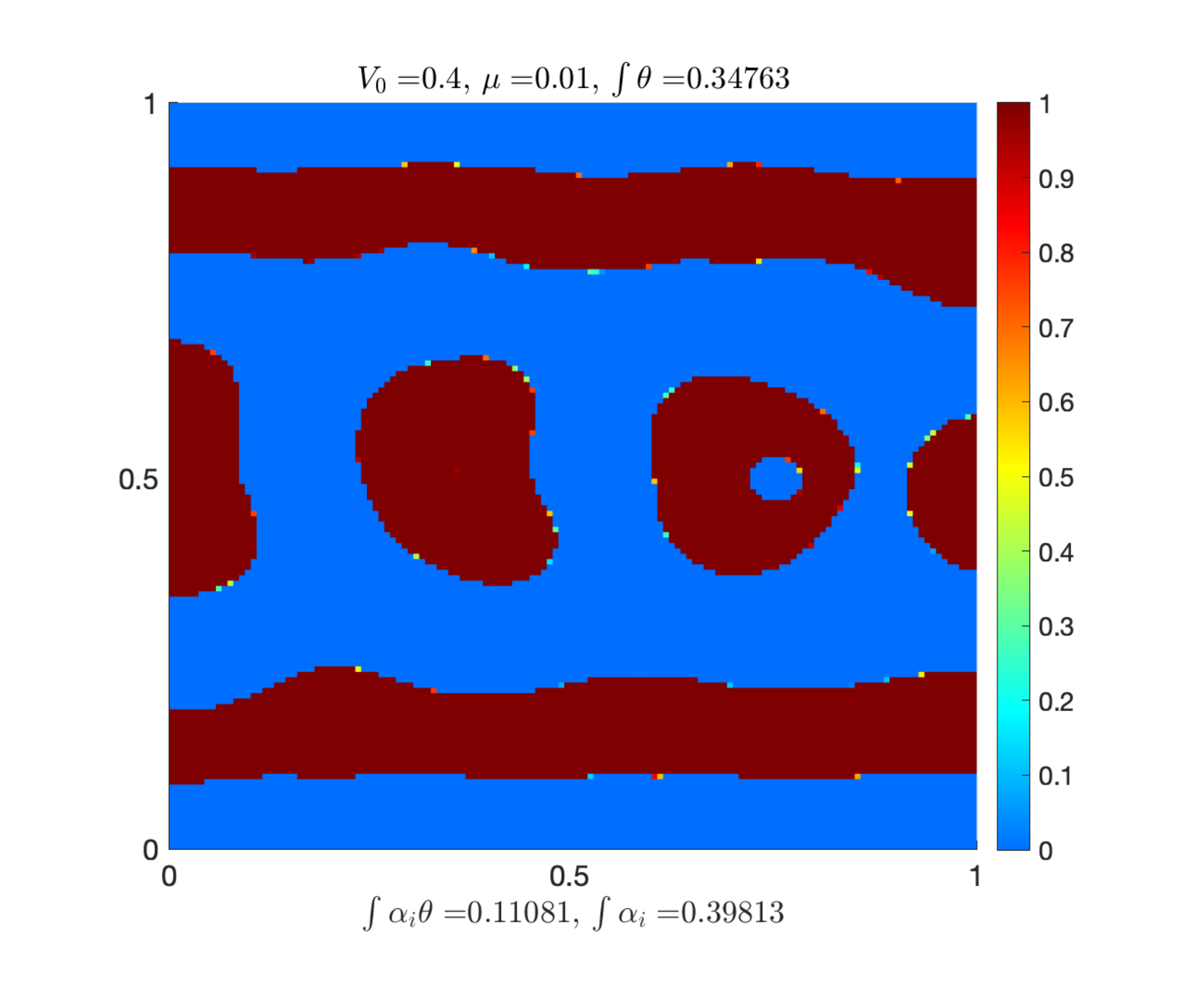}
\includegraphics[scale=0.22]{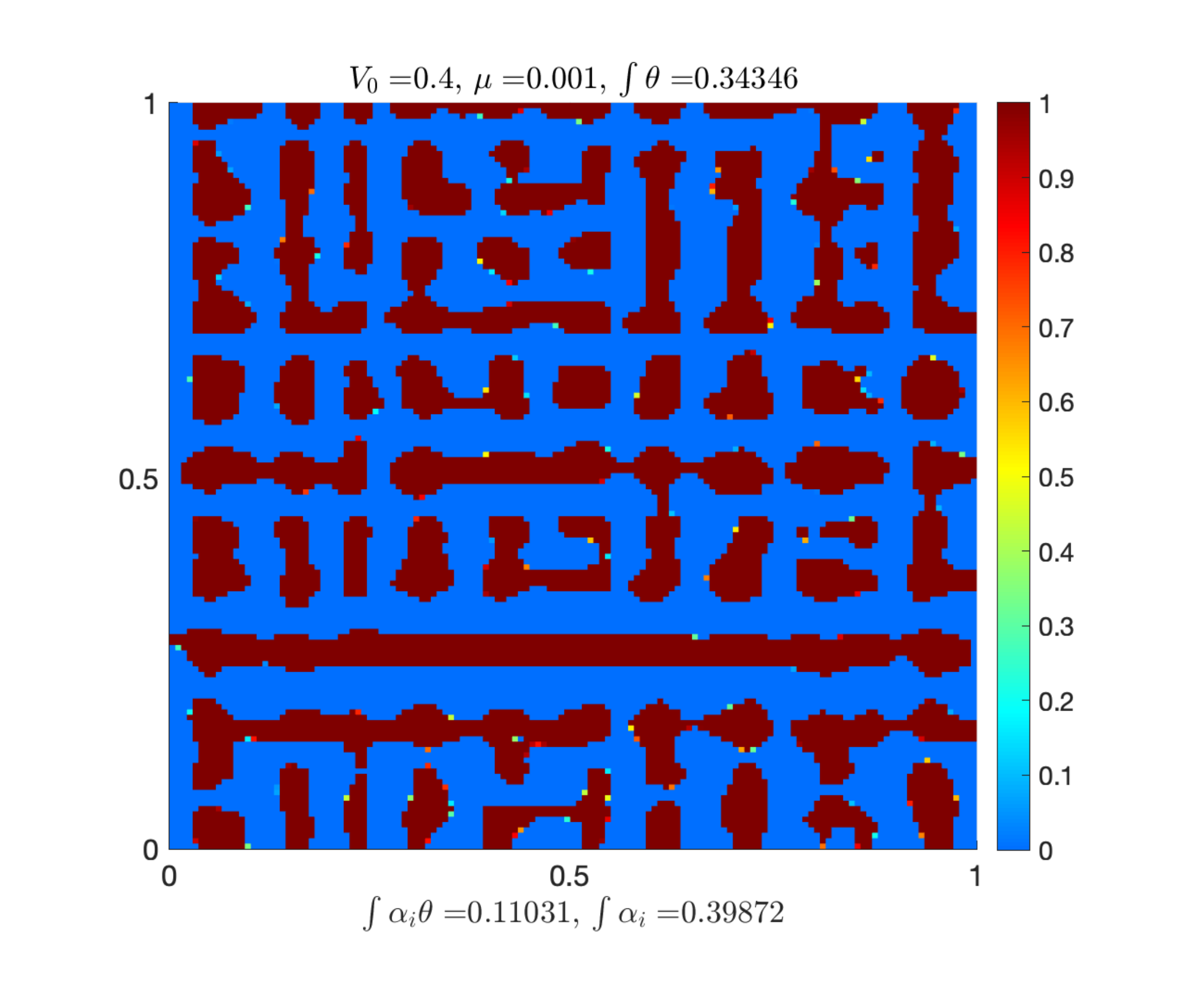}

\includegraphics[scale=0.22]{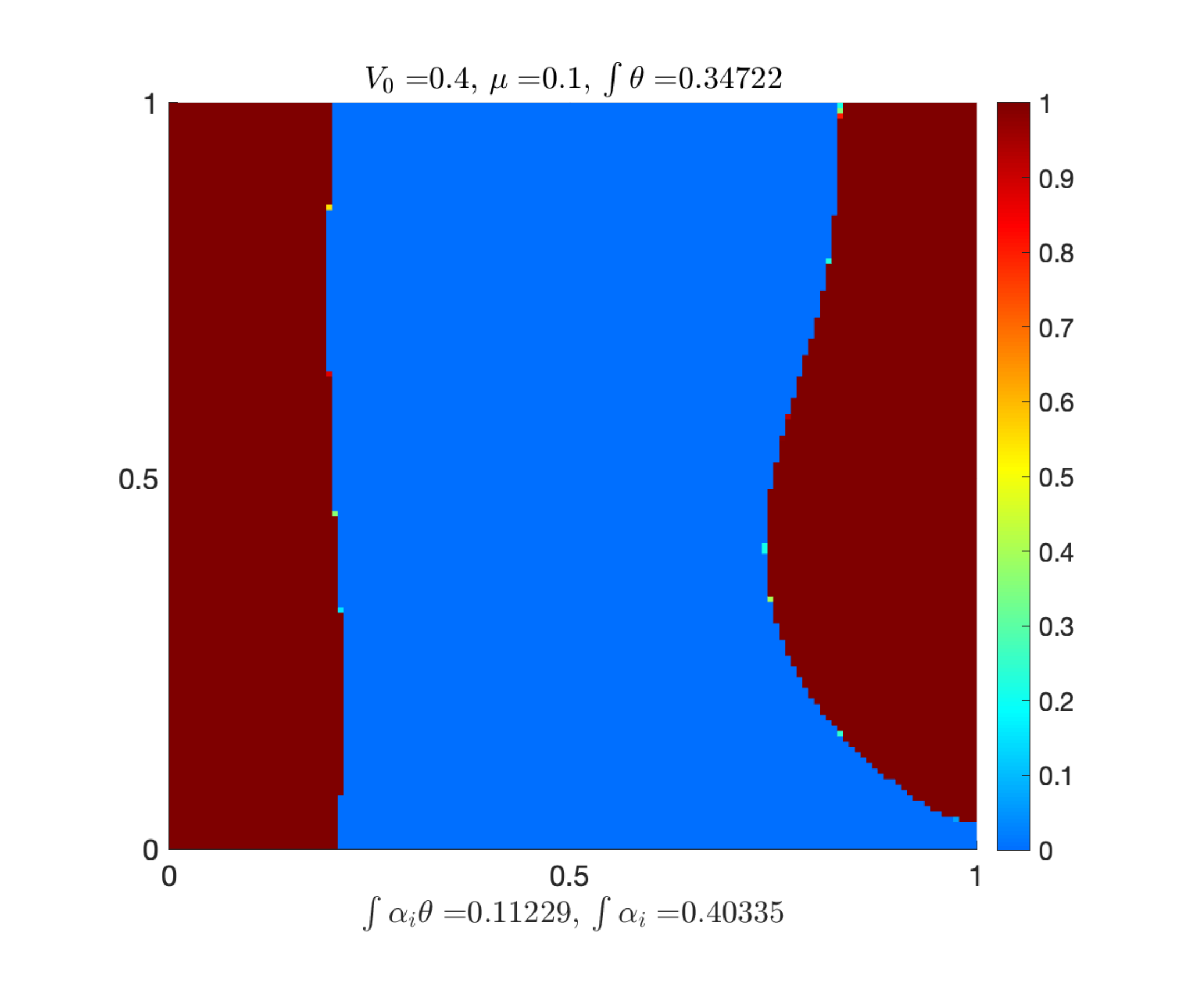}
\includegraphics[scale=0.22]{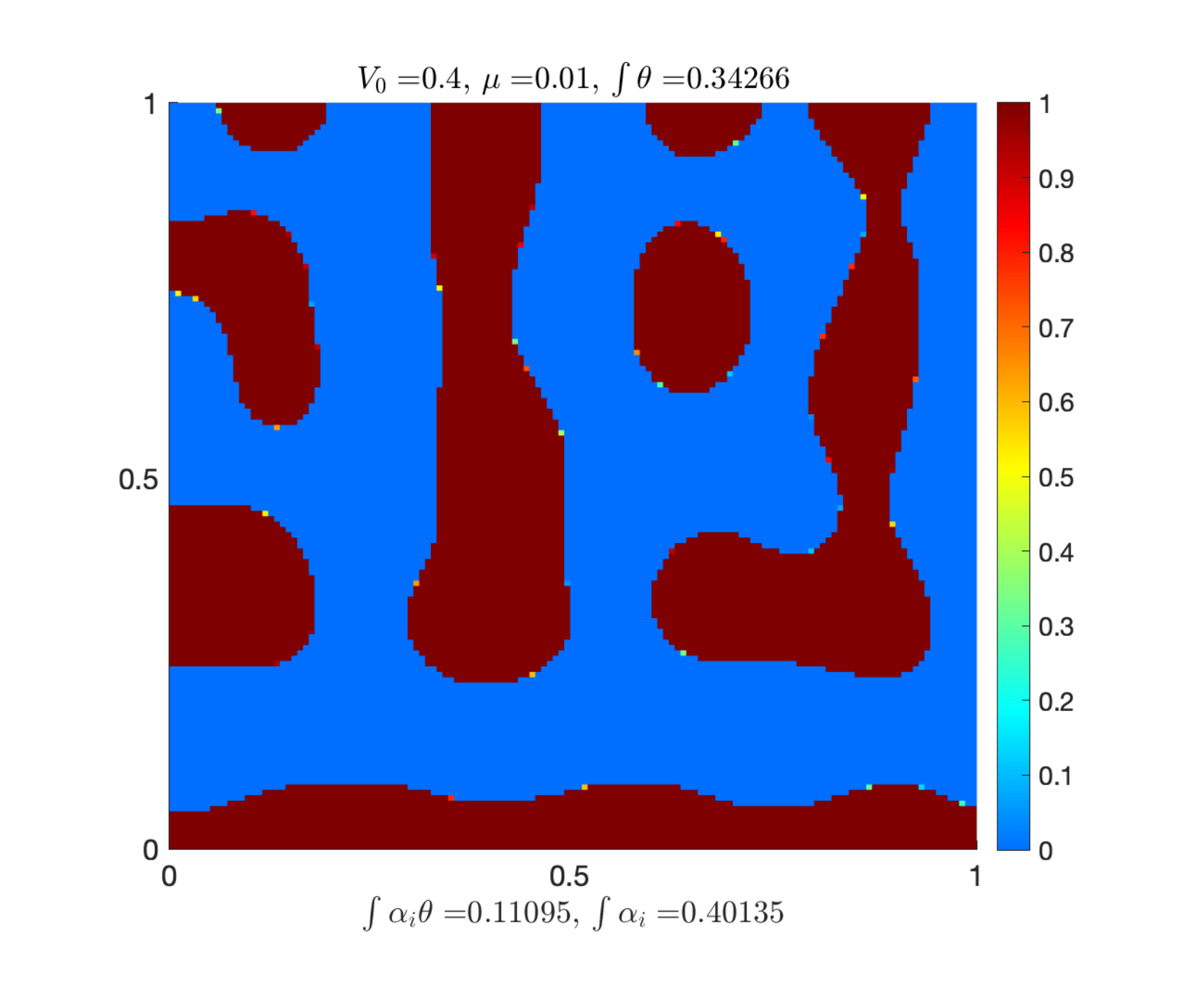}
\includegraphics[scale=0.22]{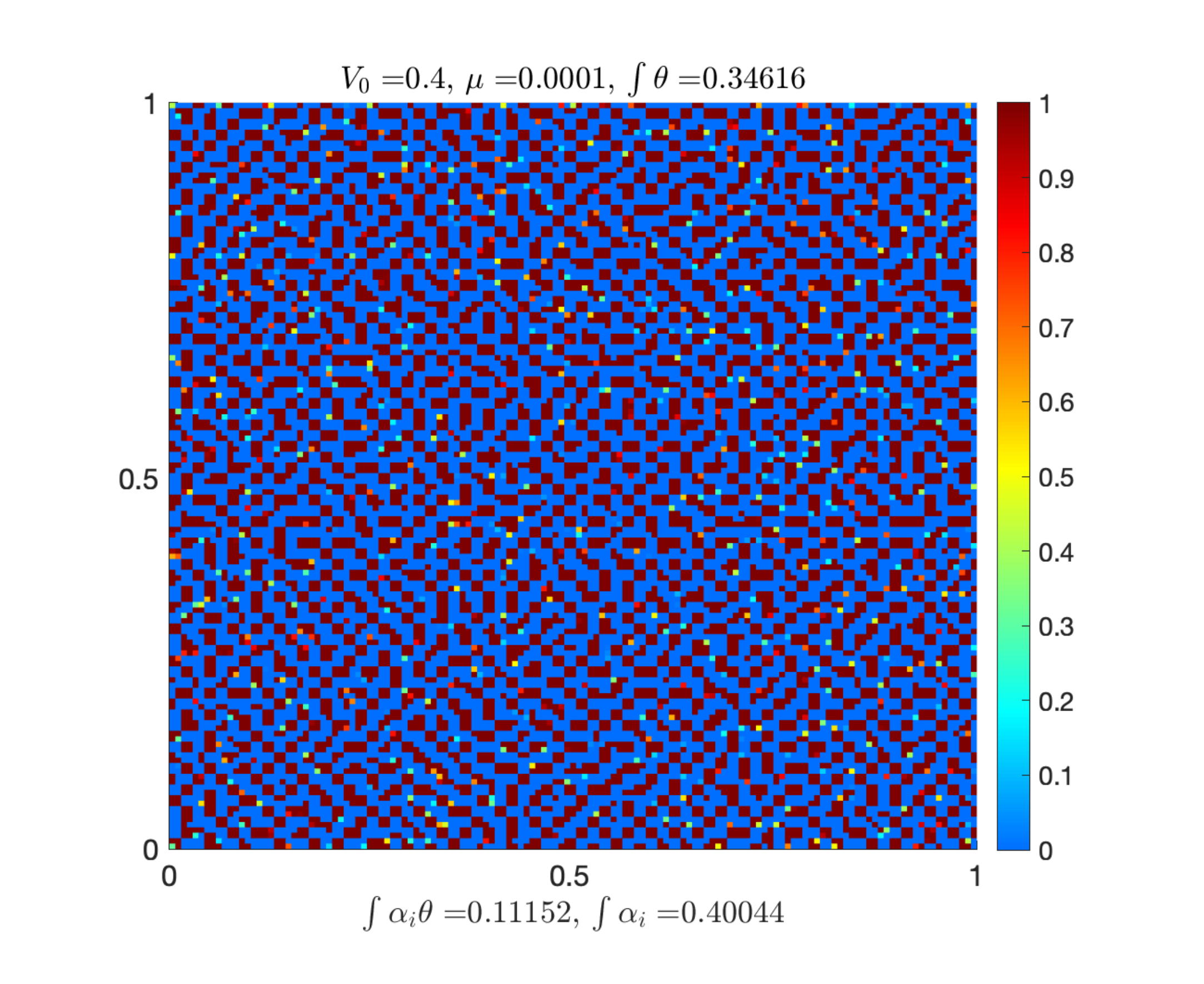}

\end{center}
\caption{In this figure, several simulations following the fixed point Algorithm \ref{algo1} have been performed. In all the simulations, both players play the same strategy, and hence, only one strategy is depicted. $K(x)=1$ was choosen for all simulations.}\label{Fig:2dNah}
\end{figure}

{

In this subsection we perform simulations and discuss the case in which both players have the same fishing capacity, \emph{i.e.} the game is symmetric. After looking at Figures \ref{Fig:1dNash} and \ref{Fig:2dNah}, one should observe five things:
\begin{enumerate}
\item All simulations have converged to an $\epsilon$-Nash equilibria in which the strategies of both players are the same. Therefore, they suggest that the search of such Nash equilibria can be phrased as finding fixed points for the map
$ \Lambda:L^\infty(\Omega)\longrightarrow L^\infty(\Omega)$
defined as
$$\Lambda(K)=1-\alpha^*_K $$
where by $\alpha^*_K$ is an element of the minimizers of the single player game. Of course, for such map to be well defined, we would need to ensure uniqueness for \eqref{Pv:1Ineq}.

\item When the integral bound $V_0$ is small, the Nash equilibria are not exhibiting a bang-bang structure and instead are constants (for $K=1$, Figure \ref{Fig:1dNash} left column). This is in the line with the concavity properties observed in Theorems \ref{Th:Concave} and \ref{Th:AsymptoticSingle} in this paper. 

\item When $V_0$ is big, the algorithm converges, for  both the one and two dimensional problems to a Nash equilibrium that is bang-bang. We have observed that, in the asymptotic regime there are two Nash equilibria for $V_0=\frac{1}{3}$, $(V_0,V_0)$ and $(\mathbbm{1}_{(0,V_0)},\mathbbm{1}_{(0,V_0)})$ (as a consequence of Theorem \ref{Th:NashLarge}). For every $\mu$,  $(V_0,V_0)$ is a Nash equilibria for the non-linear problem. An interesting question is to determine whether or not bang-bang symmetric Nash equilibria exist for general diffusivities.

\item Simulations in Figure \ref{Fig:2dNah}, also point that in the two dimensional case, there is no uniqueness. For the same diffusivity, and for the same integral bound, two different $\epsilon$-Nash equilibria were found (left and middle columns of Figure \ref{Fig:2dNah}).
\item An apparent fragmentation phenomena as observed in the simulations. The TV semi-norm of the strategies increases as $\mu\to 0^+$. This is a phenomenon observed in the maximisation of the total population size \cite{heo2021ratio,Mazari2021,MRBSIAP}. However, this phenomenon is quite surprising with respect to the previous studies since, in this problem, we are dealing with Nash-equilibria for a game whose pay-offs are different from maximizing the total population. 
\end{enumerate}
}

\subsubsection{Non-symmetric bounds}

{
In this subsection we introduce some asymmetry in the problem by considering different capacities for the players Figures \ref{Fig:nonsym1d} and \ref{Fig:nonsym2d}. We remark the following

\begin{enumerate}
\item As observed in the previous case, when the integral bound is low, the observed $\epsilon$-Nash equilibria consists of a pair of constants (left column in Figure \ref{Fig:nonsym1d}). As before, this is a manifestation of Theorem \ref{Th:Concave} and Theorem \ref{Th:NashExists}.

\item In contrast with the symmetric case, we no longer observe a full bang-bang strategy. Both in the one dimensional case in Figure \ref{Fig:nonsym1d} and in the two dimensional one in Figure \ref{Fig:nonsym2d}, we observe that the player with higher capacity adopts a bang-bang strategy while the player with less capacity is not showing this feature.

\item Figure \ref{Fig:nonsym1d} also shows that for high integral bound, the players do not necessarly share the supports of their strategies. In contrast, the simulation done in the two dimensional case, Figure \ref{Fig:nonsym2d} is not showing this particularity. There, it is observed that the player with less capacity fishes in the same area than the player with higher capacity but at a "lesser" intensity in some areas.

\item Note that all the comments made for \eqref{Pv:1Ineq} in the previous subsection regarding Figure \ref{Fig:optimalH} apply in the context of Nash equilibria for understanding a posteriori its geometrical properties.

\item Furthermore, we also observe an apparent fragmentation of the Nash equilibria shown for high capacity.

\end{enumerate}

}

\begin{figure}
\begin{center}
\includegraphics[scale=0.2]{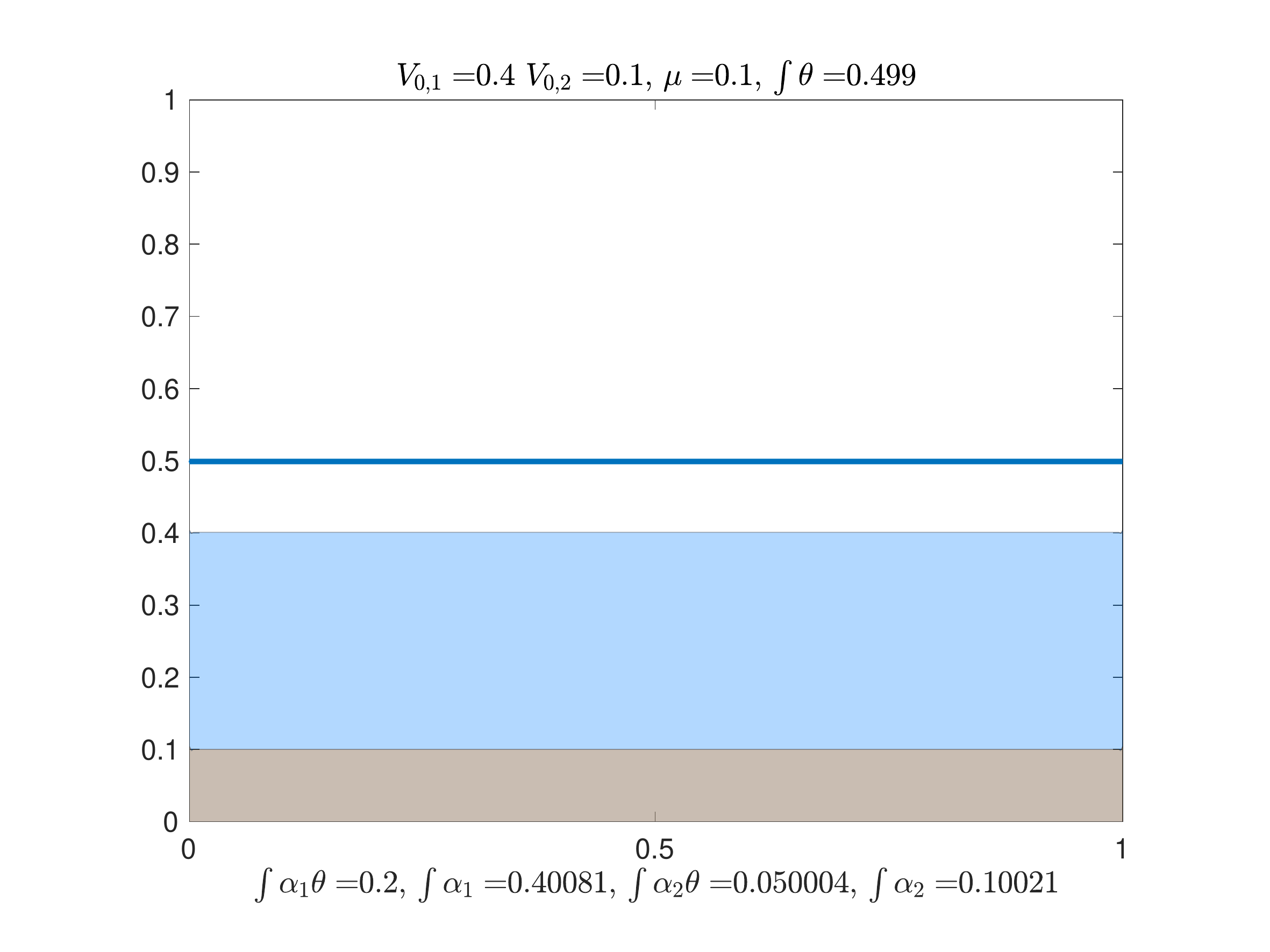}
\includegraphics[scale=0.2]{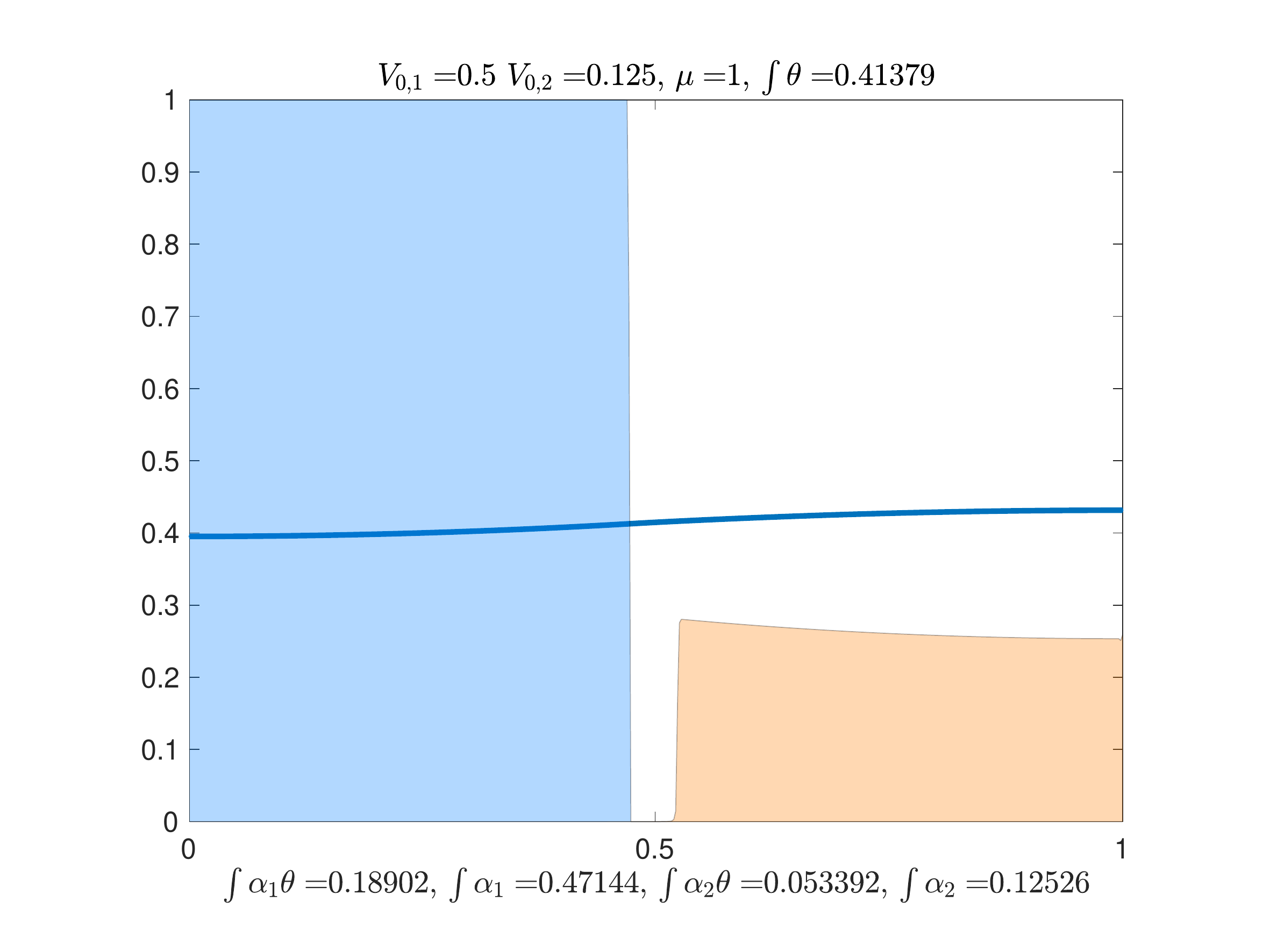}
\includegraphics[scale=0.2]{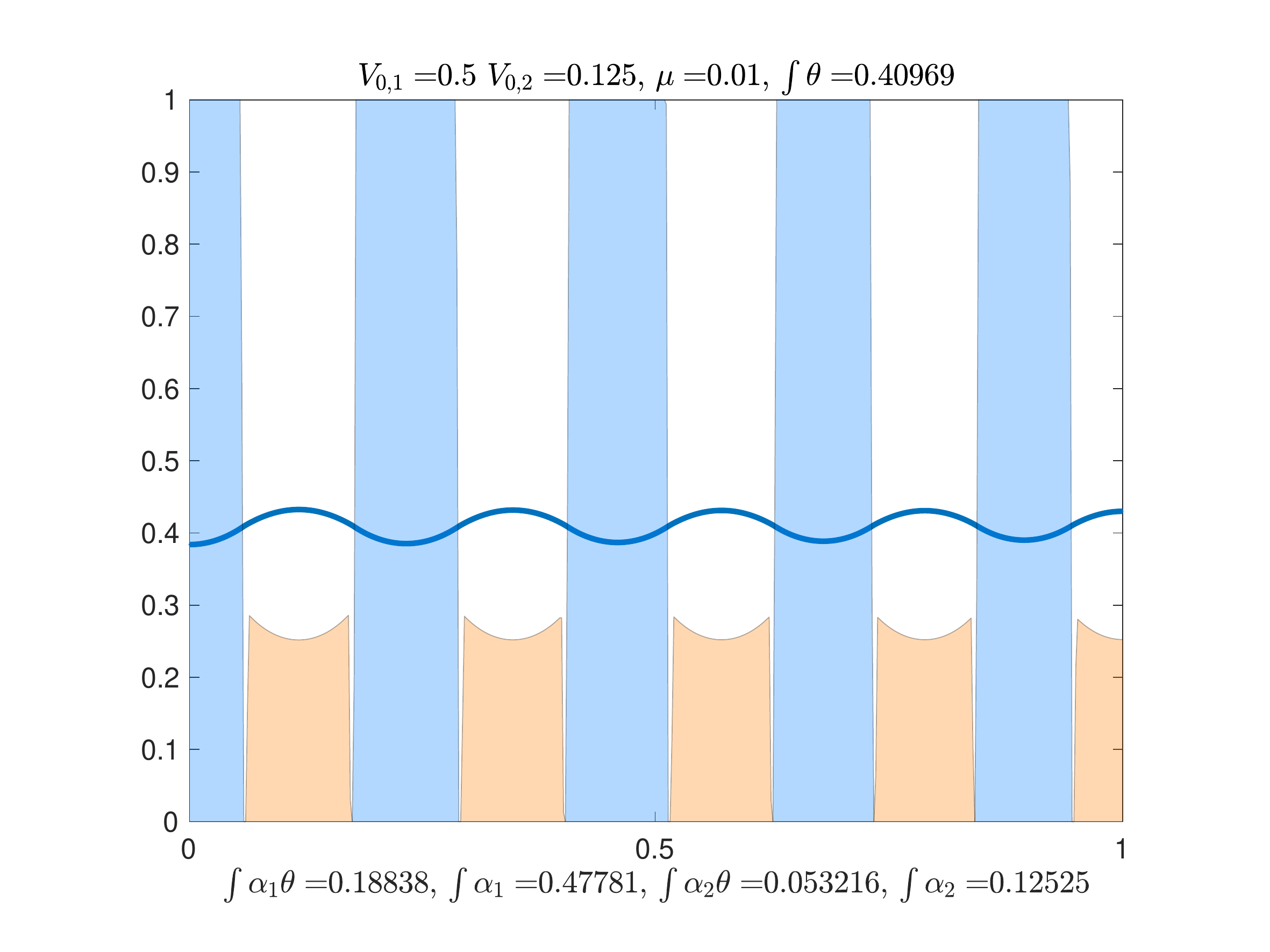}

\includegraphics[scale=0.2]{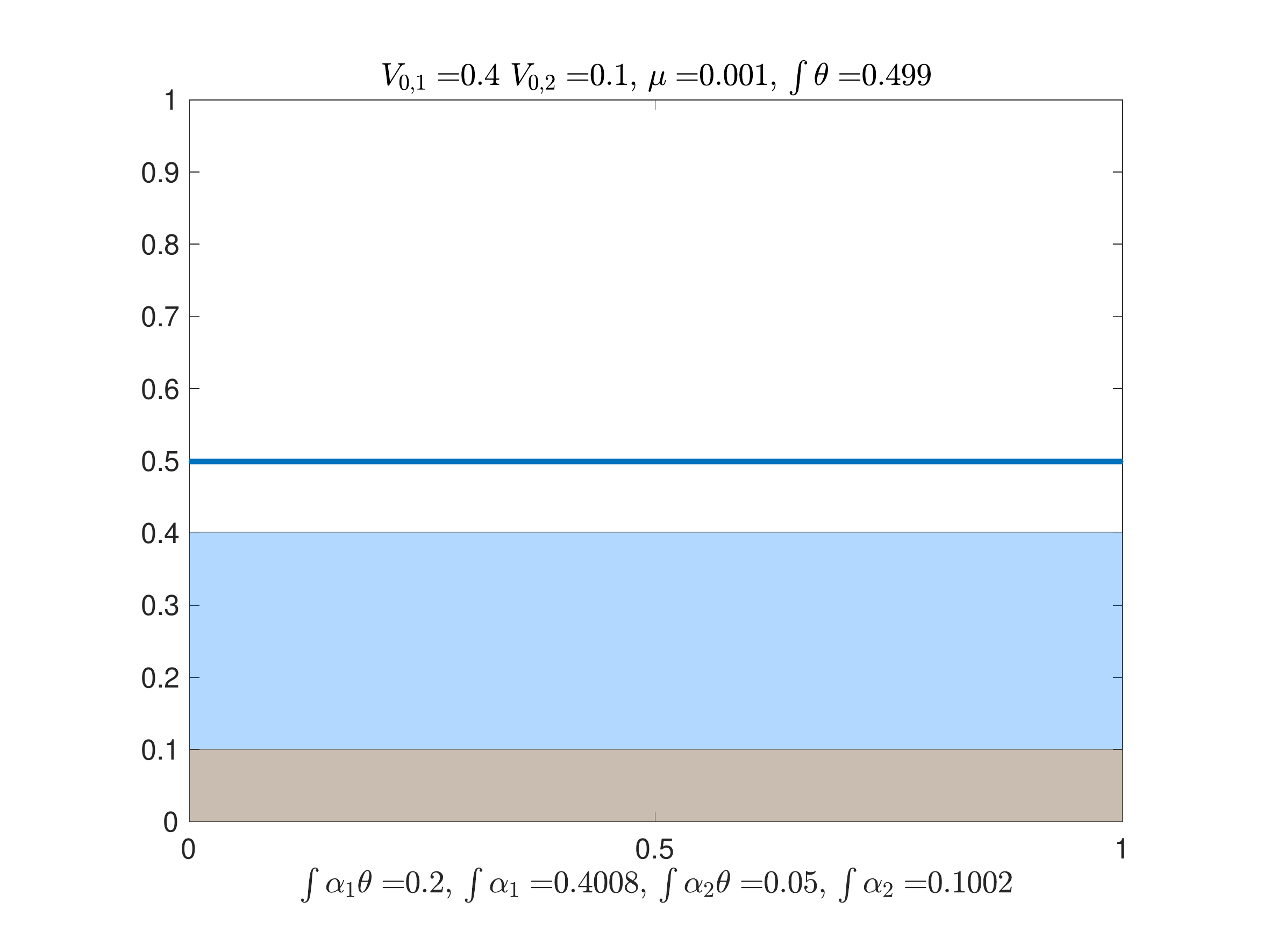}
\includegraphics[scale=0.2]{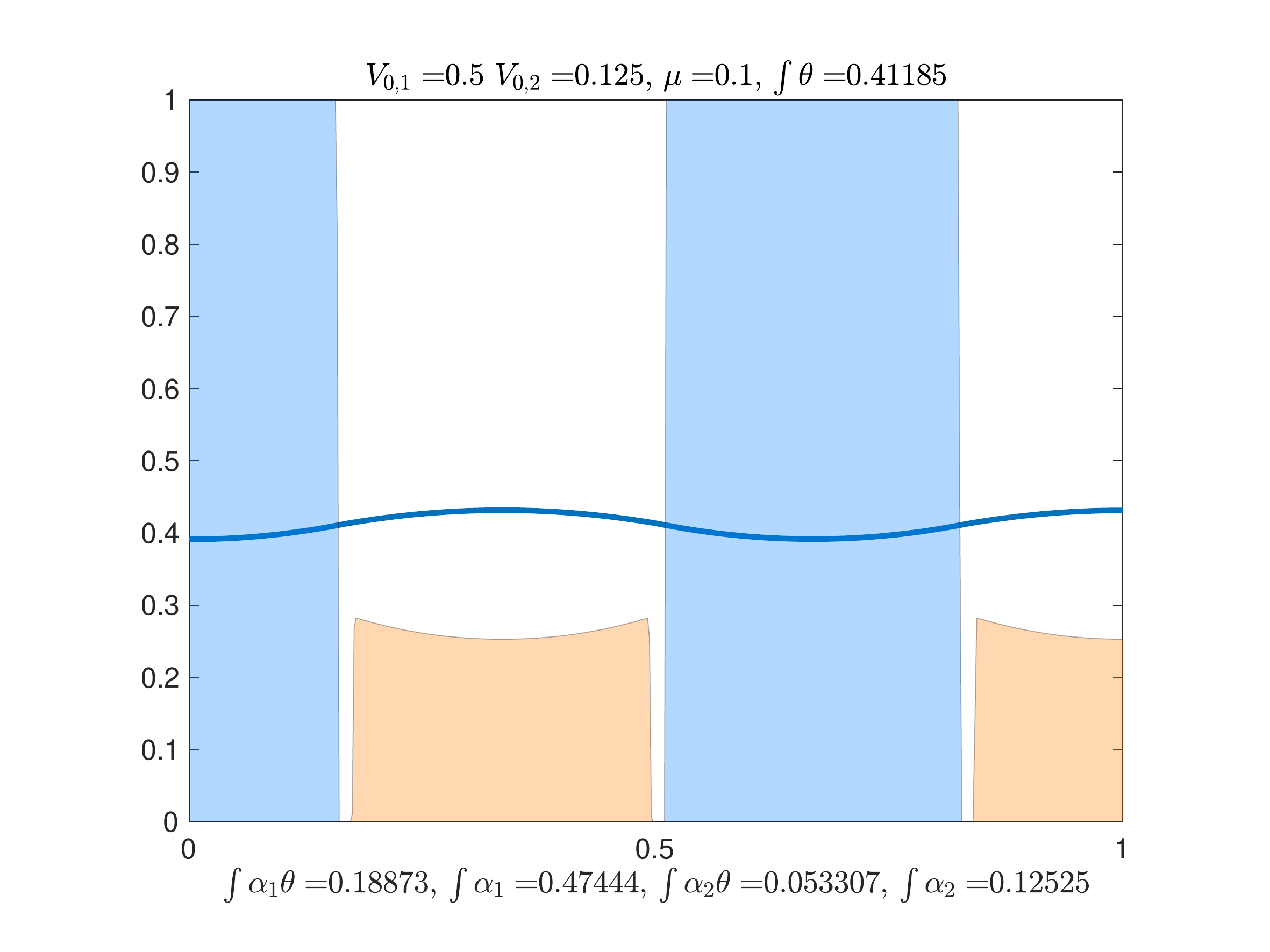}
\includegraphics[scale=0.2]{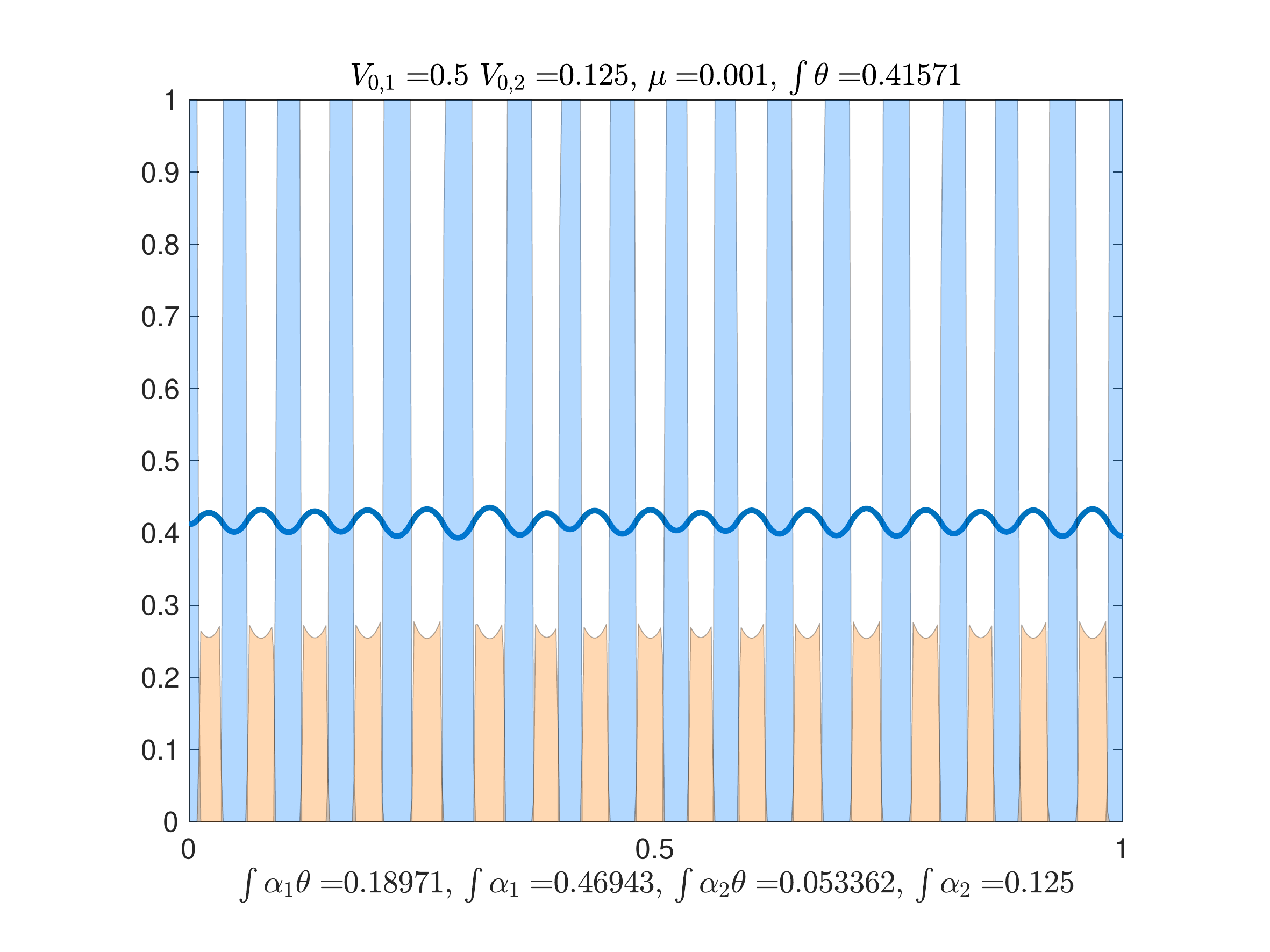}

\end{center}

\caption{The blue line represents the state $\theta_{\alpha_1,\alpha_2}$. The strategy of the first (and second) player, with higher (lower) capacity, has been depicted as a blue (orange) subgraph. The second player has a lower integral bound than the first player. In all these simulations, $K\equiv1$.}\label{Fig:nonsym1d}
\end{figure}

\begin{figure}
\begin{center}
\includegraphics[scale=0.3]{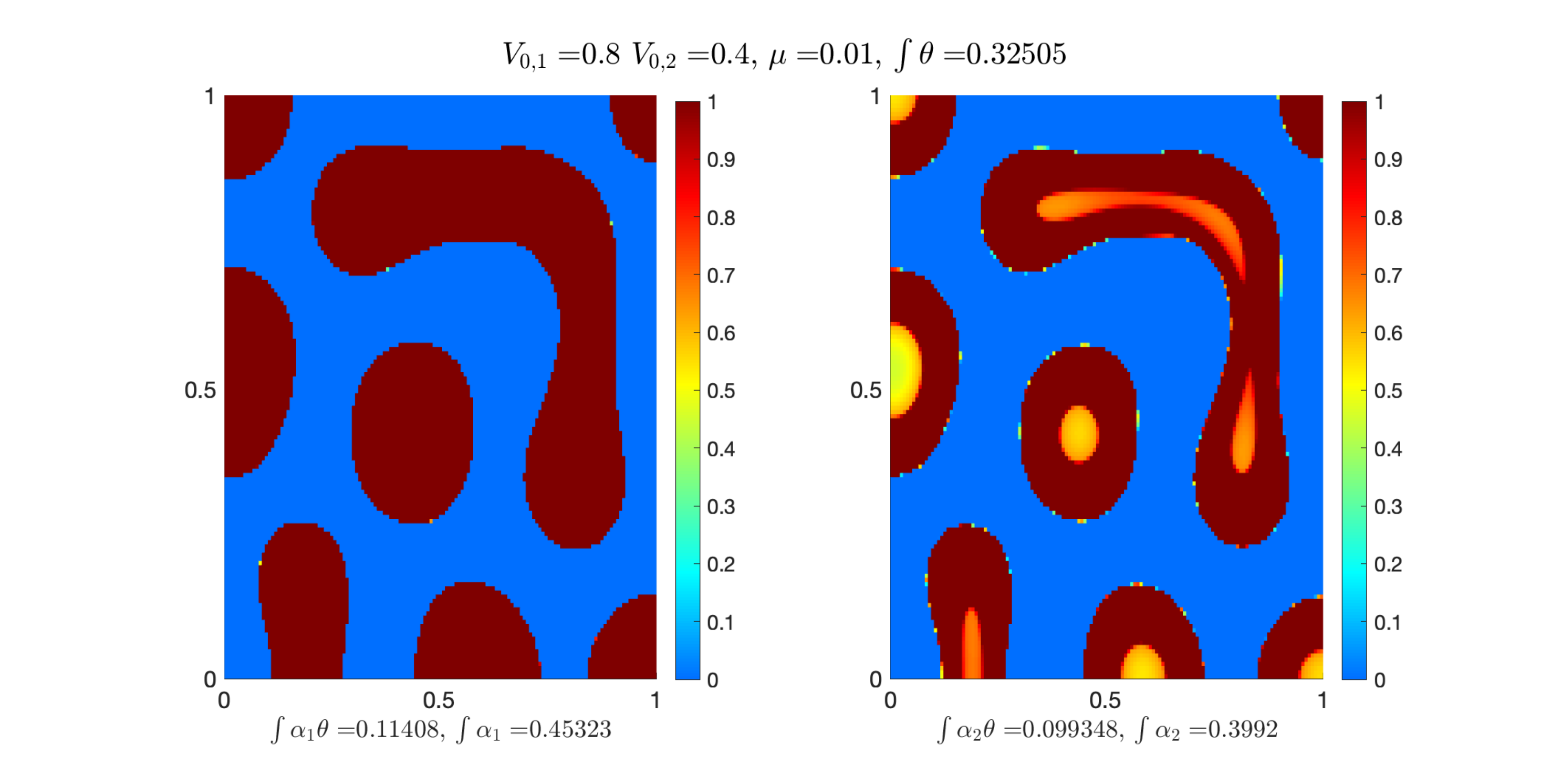}
\end{center}
\caption{At the left the strategy of the first player, at the right the strategy of the second player. The second player has a lower integral bound than the first player. $K(x)=1$.}\label{Fig:nonsym2d}
\end{figure}

\section{Open problems}
\paragraph{Concavity for low fishing abilities in higher dimensions}
One of the main drawbacks of Theorem \ref{Th:Concave} is the fact that it holds only in one-dimension or, in higher dimensions, if the resources distribution $K$ is close to a constant in the $L^1$ norm. As was seen during the proof, the main possibility to derive a result is so far to establish that
\begin{equation}\label{Eq:BHL}1-\frac\mu4\cdot\frac{|\n \overline\theta|^2}{\overline \theta^3}>0\end{equation} for any $K\in \mathcal K(\O)$, $\overline \theta$ being the solution of \eqref{Eq:LFAEq}. In the on-dimensional case this was obtained through an estimate of \cite{BaiHeLi}. In the higher-dimensional setting, however, it is quite likely that there is some serious difficulty in obtaining such an estimate for the following reason: in \cite{InoueKuto}, it is proved that, if we simply assume that $K\geq 0\,, K\neq 0$, the quantity $\sup_{\mu>0\,, K\,, K\neq 0\,, K\geq 0}\fint_\O \overline \theta/\fint_\O K$ is infinite. Should an estimate of the form \eqref{Eq:BHL}, such a result could not be true (as one could then apply the technique of \cite{BaiHeLi} and obtain  $\sup_{\mu,K\,, K\neq 0\,, K\geq 0}\fint_\O \overline \theta/\fint_\O K\leq 4,$ an obvious contradiction. Of course, in constructing a sequence such that the biomass to resources ratio diverges, the authors of \cite{InoueKuto} blow the $L^\infty$ bound up, but the fact that such phenomena occur in higher dimensions indicates the potentially very intricate nature of the problem.

\paragraph{The question of fragmentation for Nash equilibria}
We have also observed a clear fragmentation phenomenon of Nash equilibria in the low diffusivity limits. Building on \cite{heo2021ratio,MRBSIAP}, is it possible to prove a theorem of the form
\[ \lim_{\mu \to 0^+}\inf_{(\alpha_1^*,\alpha_2^*)\text{ Nash equilibria}}\min\left(\Vert \alpha_1^*\Vert_{BV},\Vert \alpha_2^*\Vert_{BV}\right)=\infty?\]
At this stage, it seems thoroughly unclear how one could approach that question, as this would require a very fine knowledge of the set of all Nash equilibria of the problem. We plan on tackling this question in future works.

\paragraph{Optimal Game Regulation Problem}

In this article we have studied several regimes for which Nash equilibria exist. Furthermore, we also illustrated how Nash equilibria lead to an under-performance of resources, in the sense that there are Nash equilibria for which the sum of the pay-offs of the players is strictly lower than what is optimal to fish. This also has been illustrated in the numerical simulations. Behind this lines, there is a relevant problem to be addressed. What is the optimal regulation so that we avoid overfishing as much as possible? 

In Figure \ref{nashvsbound} the total fish harvested is depicted with respect to the volume constraint. One can observe that, for the Nash equilibria found, there is an optimal volume constraint for maximising the total amount  harvested. This allows us to propose an optimal regulation problem for the harvesting problem. Let us first define  the set of all Nash equilibria given a volume constraint $V_0$
$$ \mathcal{N}(V_0):=\left\{(\alpha_1^*,\alpha_2^*)\in \mathcal{M}_{\leq}(1,V_0)\text{ such that } (\alpha_1^*,\alpha_2^*) \text{ is a Nash equilibria}\right\}.$$
Now, the \textit{optimal game regulation problem} for the harvesting game is the maximisation of \textit{the worst} Nash equilibria with respect to $V_0$, mathematically
$$ \max_{V_0} \min_{(\alpha_1,\alpha_2)\in \mathcal{N}(V_0)} \int_\Omega (\alpha_1(x)+\alpha_2(x))\theta_{\alpha_1,\alpha_2}(x) dx$$
where $\theta$ follows \eqref{Igame}. 
To address this problem, it is necessary to characterise all Nash equilibria given a volume constraint $V_0$. In Figure \ref{nashvsbound}, we only used the Nash equilibria found with Algorithm \ref{algo1}, but we do not know if there are other Nash equilibria. 
It is worth noting that, in the case of Figure \ref{nashvsbound} ($K\equiv 1$), it would be sufficient to prove that the unique Nash equilibria for $V_0=0.25$ is  $\alpha_1(x)=V_0$, $\alpha_2(x)=V_0$.

\begin{figure}
\centering
\includegraphics[scale=0.25]{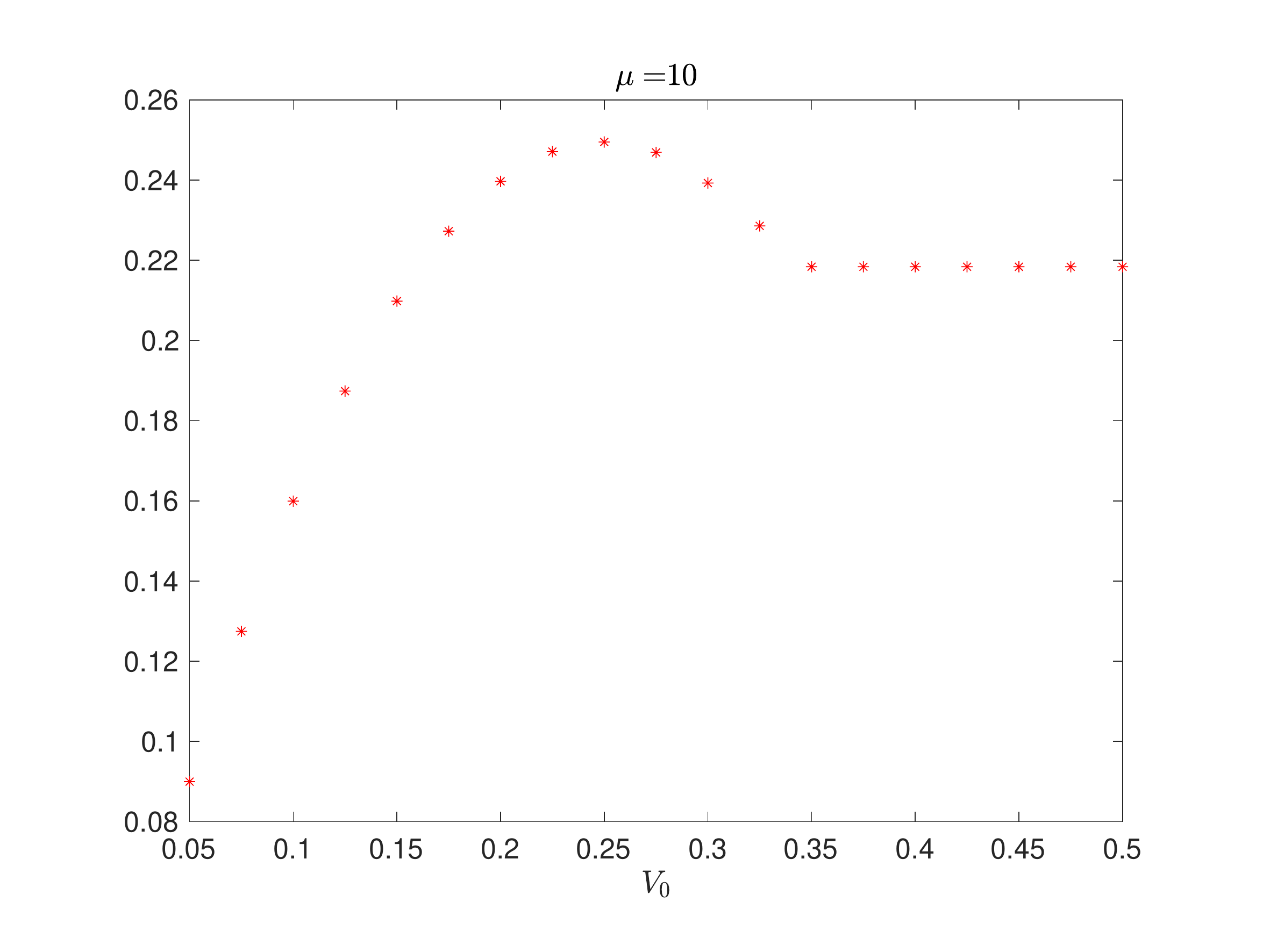}
\caption{In the vertical axis, the sum of the total fish harvested for both players at the Nash equilibria found, in the horizontal axis the volume constraint for both players. $K(x)=1$. }\label{nashvsbound}
\end{figure}

\bibliographystyle{abbrv}
\bibliography{BiblioNash}

\begin{thebibliography}{10}

\bibitem{Alvino1990}
A.~Alvino, P.-L. Lions, and G.~Trombetti.
\newblock Comparison results for elliptic and parabolic equations via {S}chwarz
  symmetrization.
\newblock {\em Annales de l'Institut Henri Poincare (C) Non Linear Analysis},
  7(2):37--65, Mar. 1990.

\bibitem{alvino1991}
A.~Alvino, P.-L. Lions, and G.~Trombetti.
\newblock Comparison results for elliptic and parabolic equations via
  symmetrization: a new approach.
\newblock {\em Differential Integral Equations}, 4(1):25--50, 1991.

\bibitem{AlvinoNitschTrombetti}
A.~Alvino, C.~Nitsch, and C.~Trombetti.
\newblock Comparison result for solutions to elliptic problems with robin
  boundary conditions.
\newblock {\em arXiv: Analysis of PDEs}, 2019.

\bibitem{Baernstein_1989}
A.~Baernstein.
\newblock Convolution and rearrangement on circle.
\newblock {\em Complex Variables, Theory and Application: An International
  Journal}, 12(1-4):33--37, oct 1989.

\bibitem{Baernstein}
A.~Baernstein~II.
\newblock Rearrangements.
\newblock In {\em Symmetrization in Analysis}, pages 16--53. Cambridge
  University Press, Mar. 2019.

\bibitem{BaiHeLi}
X.~Bai, X.~He, and F.~Li.
\newblock An optimization problem and its application in population dynamics.
\newblock {\em Proceedings of the American Mathematical Society},
  144(5):2161--2170, Oct. 2015.

\bibitem{Bandle}
C.~Bandle.
\newblock {\em Isoperimetric Inequalities and Applications}.
\newblock Monographs and studies in mathematics. Pitman, 1980.

\bibitem{BBC50}
BBC-News.
\newblock 'only 50 years left' for sea fish.
\newblock \url{http://news.bbc.co.uk/2/hi/science/nature/6108414.stm}, 2006.

\bibitem{BBCwaste}
BBC-News.
\newblock Fisheries waste 'costs billions'.
\newblock \url{http://news.bbc.co.uk/1/hi/sci/tech/7660011.stm}, 2008.

\bibitem{BBC2}
BBC-News.
\newblock Global fish stocks are exploited or depleted to such an extent that
  without urgent measures we may be the last generation to catch food from the
  oceans.
\newblock
  \url{https://www.bbc.com/future/article/20120920-are-we-running-out-of-fish},
  2012.

\bibitem{BHR}
H.~Berestycki, F.~Hamel, and L.~Roques.
\newblock Analysis of the periodically fragmented environment model : I --
  species persistence.
\newblock {\em Journal of Mathematical Biology}, 51(1):75--113, 2005.

\bibitem{Braverman2009}
E.~Braverman and L.~Braverman.
\newblock Optimal harvesting of diffusive models in a nonhomogeneous
  environment.
\newblock {\em Nonlinear Analysis: Theory, Methods {\&} Applications},
  71(12):e2173--e2181, Dec. 2009.

\bibitem{Cal_Campana_2020}
F.~C. Campana, G.~Ciaramella, and A.~Borz{\`{\i}}.
\newblock Nash equilibria and bargaining solutions of differential bilinear
  games.
\newblock {\em Dynamic Games and Applications}, 11(1):1--28, mar 2020.

\bibitem{CantrellCosner1}
R.~S. Cantrell and C.~Cosner.
\newblock Diffusive logistic equations with indefinite weights: Population
  models in disrupted environments {II}.
\newblock {\em {SIAM} Journal on Mathematical Analysis}, 22(4):1043--1064, jul
  1991.

\bibitem{MR1105497}
R.~S. Cantrell and C.~Cosner.
\newblock The effects of spatial heterogeneity in population dynamics.
\newblock {\em J. Math. Biol.}, 29(4):315--338, 1991.

\bibitem{CantrellCosner}
R.~S. Cantrell and C.~Cosner.
\newblock {\em Spatial Ecology via Reaction-Diffusion Equations}.
\newblock John Wiley \& Sons, 2003.

\bibitem{31806255e7f648d5b65ff02c30c4f539}
R.~S. Cantrell, C.~Cosner, and V.~Hutson.
\newblock Permanence in ecological systems with spatial heterogeneity.
\newblock {\em Proceedings of the Royal Society of Edinburgh Section A:
  Mathematics}, 123(3):533--559, 1993.

\bibitem{Carvalho_2018}
P.~P. Carvalho and E.~Fern{\'{a}}ndez-Cara.
\newblock On the computation of nash and pareto equilibria for some
  bi-objective control problems.
\newblock {\em Journal of Scientific Computing}, 78(1):246--273, jun 2018.

\bibitem{Cooke1986}
K.~L. Cooke and M.~Witten.
\newblock One-dimensional linear and logistic harvesting models.
\newblock {\em Mathematical Modelling}, 7(2-3):301--340, 1986.

\bibitem{costello2012status}
C.~Costello, D.~Ovando, R.~Hilborn, S.~D. Gaines, O.~Deschenes, and S.~E.
  Lester.
\newblock Status and solutions for the world's unassessed fisheries.
\newblock {\em Science}, 338(6106):517--520, 2012.

\bibitem{davies2012extinction}
T.~D. Davies and J.~K. Baum.
\newblock Extinction risk and overfishing: reconciling conservation and
  fisheries perspectives on the status of marine fishes.
\newblock {\em Scientific reports}, 2(1):1--9, 2012.

\bibitem{DeAngelis2020}
D.~DeAngelis, B.~Zhang, W.-M. Ni, and Y.~Wang.
\newblock Carrying capacity of a population diffusing in a heterogeneous
  environment.
\newblock {\em Mathematics}, 8(1):49, Jan. 2020.

\bibitem{Ding2010}
W.~Ding, H.~Finotti, S.~Lenhart, Y.~Lou, and Q.~Ye.
\newblock Optimal control of growth coefficient on a steady-state population
  model.
\newblock {\em Nonlinear Analysis: Real World Applications}, 11(2):688--704,
  Apr. 2010.

\bibitem{DockeryHutsonMischaikowPernarowskiEvolution}
J.~Dockery, V.~Hutson, K.~Mischaikow, and M.~Pernarowski.
\newblock The evolution of slow dispersal rates: a reaction diffusion model.
\newblock {\em Journal of Mathematical Biology}, 37(1):61--83, 1998.

\bibitem{Fern_ndez_Cara_2021}
E.~Fern{\'{a}}ndez-Cara and I.~Mar{\'{\i}}n-Gayte.
\newblock Bi-objective optimal control of some {PDEs}: Nash equilibria and
  quasi-equilibria.
\newblock {\em {ESAIM}: Control, Optimisation and Calculus of Variations},
  27:50, 2021.

\bibitem{Ferone2003}
V.-R. Ferone, Adele.
\newblock Minimal rearrangements of sobolev functions : a new proof.
\newblock {\em Annales de l'I.H.P. Analyse non lin{\'e}aire}, 20(2):333--339,
  2003.

\bibitem{Fisher}
R.~A. Fisher.
\newblock The wave of advances of advantageous genes.
\newblock {\em Annals of Eugenics}, 7(4):355--369, 1937.

\bibitem{Glicksberg_1952}
I.~L. Glicksberg.
\newblock A further generalization of the kakutani fixed point theorem, with
  application to nash equilibrium points.
\newblock {\em Proceedings of the American Mathematical Society}, 3(1):170, feb
  1952.

\bibitem{Gonz_lez_D_az_2010}
J.~Gonz{\'{a}}lez-D{\'{\i}}az, I.~Garc{\'{\i}}a-Jurado, and M.~G.
  Fiestras-Janeiro.
\newblock {\em An Introductory Course on Mathematical Game Theory}.
\newblock American Mathematical Society, may 2010.

\bibitem{hamilton2001outport}
L.~C. Hamilton and M.~J. Butler.
\newblock Outport adaptations: Social indicators through newfoundland's cod
  crisis.
\newblock {\em Human Ecology Review}, pages 1--11, 2001.

\bibitem{hardin2009tragedy}
G.~Hardin.
\newblock The tragedy of the commons.
\newblock {\em Journal of Natural Resources Policy Research}, 1(3):243--253,
  2009.

\bibitem{HN}
X.~He and W.-M. Ni.
\newblock Global dynamics of the lotka-volterra competition-diffusion system:
  Diffusion and spatial heterogeneity {I}.
\newblock {\em Communications on Pure and Applied Mathematics},
  69(5):981--1014, 2015.

\bibitem{HENIII}
X.~He and W.-M. Ni.
\newblock Global dynamics of the lotka--volterra competition--diffusion system
  with equal amount of total resources, {II}.
\newblock {\em Calculus of Variations and Partial Differential Equations}, 55,
  04 2016.

\bibitem{HeNi}
X.~He and W.-M. Ni.
\newblock Global dynamics of the {L}otka--{V}olterra competition--diffusion
  system with equal amount of total resources, {III}.
\newblock {\em Calc. Var. Partial Differential Equations}, 56(5):56:132, 2017.

\bibitem{henrot2006}
A.~Henrot.
\newblock {\em Extremum problems for eigenvalues of elliptic operators}.
\newblock Frontiers in Mathematics. Birkh\"auser Verlag, Basel, 2006.

\bibitem{heo2021ratio}
J.~Heo and Y.~Kim.
\newblock On the ratio of biomass to total carrying capacity in high
  dimensions.
\newblock {\em Journal of the Korean Mathematical Society}, 58(5):1227--1237,
  2021.

\bibitem{InoueKuto}
J.~Inoue, , and K.~Kuto.
\newblock On the unboundedness of the ratio of species and resources for the
  diffusive logistic equation.
\newblock {\em Discrete {\&} Continuous Dynamical Systems - B}, 22(11):0--0,
  2017.

\bibitem{Johari}
R.~Johari.
\newblock The price of anarchy and the design of scalable resource allocation
  mechanisms.
\newblock In N.~Nisan, T.~Roughgarden, E.~Tardos, and V.~V. Vazirani, editors,
  {\em Algorithmic Game Theory}, pages 543--568. Cambridge University Press.

\bibitem{KaoLouYanagida}
C.-Y. Kao, Y.~Lou, and E.~Yanagida.
\newblock Principal eigenvalue for an elliptic problem with indefinite weight
  on cylindrical domains.
\newblock {\em Math. Biosci. Eng.}, 5(2):315--335, 2008.

\bibitem{Kawohl}
B.~Kawohl.
\newblock {\em Rearrangements and Convexity of Level Sets in {PDE}}.
\newblock Springer Berlin Heidelberg, 1985.

\bibitem{Kesavan}
S.~Kesavan.
\newblock {\em Symmetrization and Applications}.
\newblock {WORLD} {SCIENTIFIC}, Apr. 2006.

\bibitem{KPP}
A.~Kolmogorov, I.~Pretrovski, and N.~Piskounov.
\newblock \'etude de l'\'equation de la diffusion avec croissance de la
  quantit\'e de mati\`ere et son application \`a un probl\`eme biologique.
\newblock {\em Moscow University Bulletin of Mathematics}, 1:1--25, 1937.

\bibitem{LamLiuLou}
K.-Y. Lam, S.~Liu, and Y.~Lou.
\newblock Selected topics on reaction-diffusion-advection models from spatial
  ecology.
\newblock {\em Math. Appl. Sci. Eng.}, 1(2):91--206, 2020.

\bibitem{LamboleyLaurainNadinPrivat}
J.~Lamboley, A.~Laurain, G.~Nadin, and Y.~Privat.
\newblock {Properties of optimizers of the principal eigenvalue with indefinite
  weight and Robin conditions}.
\newblock {\em {Calculus of Variations and Partial Differential Equations}},
  55(6), Dec. 2016.

\bibitem{Langford}
J.~Langford.
\newblock {\em Comparison Theorems in Elliptic Partial Differential Equations
  with Neumann Boundary Conditions}.
\newblock PhD thesis, Washington University, 2012.

\bibitem{lenhart2007optimal}
S.~Lenhart and J.~Workman.
\newblock {\em Optimal control applied to biological models}.
\newblock Chapman \& Hall/CRC, Boca Raton, 2007.

\bibitem{LiangLou}
S.~Liang and Y.~Lou.
\newblock On the dependence of population size upon random dispersal rate.
\newblock {\em Discrete and Continuous Dynamical Systems - Series B},
  17(8):2771--2788, July 2012.

\bibitem{LiangZhang}
X.~Liang, , and L.~Z. and.
\newblock The optimal distribution of resources and rate of migration
  maximizing the population size in logistic model with identical migration.
\newblock {\em Discrete {\&} Continuous Dynamical Systems - B}, 22(11):0--0,
  2017.

\bibitem{LouInfluence}
Y.~Lou.
\newblock On the effects of migration and spatial heterogeneity on single and
  multiple species.
\newblock {\em Journal of Differential Equations}, 223(2):400--426, Apr. 2006.

\bibitem{Lou2008}
Y.~Lou.
\newblock Some challenging mathematical problems in evolution of dispersal and
  population dynamics.
\newblock In {\em Lecture Notes in Mathematics}, pages 171--205. Springer
  Berlin Heidelberg, 2008.

\bibitem{LouNagaharaYanagida}
Y.~Lou, K.~Nagahara, and E.~Yanagida.
\newblock Maximizing the total population with logistic growth in a patchy
  environment.
\newblock {\em Submitted}, 2020.

\bibitem{LouYanagida}
Y.~Lou and E.~Yanagida.
\newblock Minimization of the principal eigenvalue for an elliptic boundary
  value problem with indefinite weight, and applications to population
  dynamics.
\newblock {\em Japan J. Indust. Appl. Math.}, 23(3):275--292, 10 2006.

\bibitem{MazariThese}
I.~Mazari.
\newblock {\em Shape optimization and spatial heterogeneity in
  reaction-diffusion equations}.
\newblock PhD thesis, Paris-Sorbonne Universit\'e, Laboratoire Jacques-Louis
  Lions, 2020.

\bibitem{M2021}
I.~Mazari.
\newblock The bang-bang property in some parabolic bilinear optimal control
  problems \emph{via} two-scale asymptotic expansions.
\newblock {\em Submitted}, 2021.

\bibitem{Mazari2022}
I.~Mazari.
\newblock Quantitative estimates for parabolic optimal control problems under
  l$\infty$ and l1 constraints in the ball: Quantifying parabolic isoperimetric
  inequalities.
\newblock {\em Nonlinear Analysis}, 215:112649, Feb. 2022.

\bibitem{MazariNadinPrivat}
I.~Mazari, G.~Nadin, and Y.~Privat.
\newblock Optimization of a two-phase, weighted eigenvalue with dirichlet
  boundary conditions.
\newblock Preprint, 2019.

\bibitem{MNPChapter}
I.~Mazari, G.~Nadin, and Y.~Privat.
\newblock {\em Handbook of optimal control and numerical analysis}, chapter
  Some challenging optimisation problems for logistic diffusive equations and
  numerical issues.
\newblock 2020.

\bibitem{Mazari2020}
I.~Mazari, G.~Nadin, and Y.~Privat.
\newblock Optimal location of resources maximizing the total population size in
  logistic models.
\newblock {\em Journal de Math{\'{e}}matiques Pures et Appliqu{\'{e}}es},
  134:1--35, Feb. 2020.

\bibitem{Mazari2021}
I.~Mazari, G.~Nadin, and Y.~Privat.
\newblock Optimisation of the total population size for logistic diffusive
  equations: bang-bang property and fragmentation rate.
\newblock {\em Communications in Partial Differential Equations}, pages 1--32,
  Dec. 2021.

\bibitem{MPRobin}
I.~Mazari and Y.~Privat.
\newblock Qualitative analysis of optimisation problems with respect to
  non-constant {R}obin coefficients.
\newblock {\em Submitted}, 2021.

\bibitem{MRBSIAP}
I.~Mazari and D.~Ruiz-Balet.
\newblock A fragmentation phenomenon for a non-energetic optimal control
  problem: optimisation of the total population size in logistic diffusive
  models.
\newblock {\em {SIAM Journal on Applied Mathematics}}, Accepted for
  publication.

\bibitem{RakotosonMossino}
J.~Mossino and J.~M. Rakotoson.
\newblock Isoperimetric inequalities in parabolic equations.
\newblock {\em Annali della Scuola Normale Superiore di Pisa - Classe di
  Scienze}, Ser. 4, 13(1):51--73, 1986.

\bibitem{NagaharaYanagida}
K.~Nagahara and E.~Yanagida.
\newblock Maximization of the total population in a reaction--diffusion model
  with logistic growth.
\newblock {\em Calculus of Variations and Partial Differential Equations},
  57(3):80, Apr 2018.

\bibitem{Nash1951}
J.~Nash.
\newblock Non-cooperative games.
\newblock {\em The Annals of Mathematics}, 54(2):286, Sept. 1951.

\bibitem{pikitch2012risks}
E.~K. Pikitch.
\newblock The risks of overfishing.
\newblock {\em Science}, 338(6106):474--475, 2012.

\bibitem{pinsky2011unexpected}
M.~L. Pinsky, O.~P. Jensen, D.~Ricard, and S.~R. Palumbi.
\newblock Unexpected patterns of fisheries collapse in the world's oceans.
\newblock {\em Proceedings of the National Academy of Sciences},
  108(20):8317--8322, 2011.

\bibitem{Rakotoson}
J.-M. Rakotoson.
\newblock {\em R{\'{e}}arrangement Relatif}.
\newblock Springer Berlin Heidelberg, 2008.

\bibitem{Roughgarden}
T.~Roughgarden and {\'{E}}.~Tardos.
\newblock Introduction to the inefficiency of equilibria.
\newblock In N.~Nisan, T.~Roughgarden, E.~Tardos, and V.~V. Vazirani, editors,
  {\em Algorithmic Game Theory}, pages 443--460. Cambridge University Press.

\bibitem{Sannipoli2022}
R.~Sannipoli.
\newblock Comparison results for solutions to the anisotropic laplacian with
  robin boundary conditions.
\newblock {\em Nonlinear Analysis}, 214:112615, Jan. 2022.

\bibitem{ShigesadaKawaski}
N.~Shigesada and K.~Kawasaki.
\newblock {\em Biological Invasions: Theory and Practice}.
\newblock Oxford University Press, 1997.

\bibitem{SuTongYang}
Y.-H. Su, , W.-T. Li, and F.-Y. Yang.
\newblock Effects of nonlocal dispersal and spatial heterogeneity on total
  biomass.
\newblock {\em Discrete {\&} Continuous Dynamical Systems - B}, 22(11):1--8,
  2017.

\bibitem{Talenti}
G.~Talenti.
\newblock Elliptic equations and rearrangements.
\newblock {\em Annali della Scuola Normale Superiore di Pisa - Classe di
  Scienze}, Ser. 4, 3(4):697--718, 1976.

\bibitem{worm2012future}
B.~Worm and T.~A. Branch.
\newblock The future of fish.
\newblock {\em Trends in ecology \& evolution}, 27(11):594--599, 2012.

\end{thebibliography}


\begin{thebibliography}{10}

\bibitem{BBC50}
BBC-News.
\newblock 'only 50 years left' for sea fish.
\newblock \url{http://news.bbc.co.uk/2/hi/science/nature/6108414.stm}, 2006.

\bibitem{BBCwaste}
BBC-News.
\newblock Fisheries waste 'costs billions'.
\newblock \url{http://news.bbc.co.uk/1/hi/sci/tech/7660011.stm}, 2008.

\bibitem{BBC2}
BBC-News.
\newblock Global fish stocks are exploited or depleted to such an extent that
  without urgent measures we may be the last generation to catch food from the
  oceans.
\newblock
  \url{https://www.bbc.com/future/article/20120920-are-we-running-out-of-fish},
  2012.

\bibitem{costello2012status}
C.~Costello, D.~Ovando, R.~Hilborn, S.~D. Gaines, O.~Deschenes, and S.~E.
  Lester.
\newblock Status and solutions for the world’s unassessed fisheries.
\newblock {\em Science}, 338(6106):517--520, 2012.

\bibitem{davies2012extinction}
T.~D. Davies and J.~K. Baum.
\newblock Extinction risk and overfishing: reconciling conservation and
  fisheries perspectives on the status of marine fishes.
\newblock {\em Scientific reports}, 2(1):1--9, 2012.

\bibitem{Peck}
M.~Dickey-Collas, R.~D. Nash, T.~Brunel, C.~J. Van~Damme, C.~T. Marshall, M.~R.
  Payne, A.~Corten, A.~J. Geffen, M.~A. Peck, E.~M. Hatfield, et~al.
\newblock Lessons learned from stock collapse and recovery of north sea
  herring: a review.
\newblock {\em ICES Journal of Marine Science}, 67(9):1875--1886, 2010.

\bibitem{EUfish}
{European Council}.
\newblock Management of the eu's fish stocks.
\newblock \url{https://www.consilium.europa.eu/en/policies/eu-fish-stocks/},
  2021.

\bibitem{hamilton2001outport}
L.~C. Hamilton and M.~J. Butler.
\newblock Outport adaptations: Social indicators through newfoundland's cod
  crisis.
\newblock {\em Human Ecology Review}, pages 1--11, 2001.

\bibitem{hardin2009tragedy}
G.~Hardin.
\newblock The tragedy of the commons.
\newblock {\em Journal of Natural Resources Policy Research}, 1(3):243--253,
  2009.

\bibitem{MNP}
I.~Mazari, G.~Nadin, and Y.~Privat.
\newblock Optimal location of resources maximizing the total population size in
  logistic models.
\newblock {\em Journal de Math{\'{e}}matiques Pures et Appliqu{\'{e}}es}, 2020.

\bibitem{mazari2021fragmentation}
I.~Mazari and D.~Ruiz-Balet.
\newblock A fragmentation phenomenon for a nonenergetic optimal control
  problem: Optimization of the total population size in logistic diffusive
  models.
\newblock {\em SIAM Journal on Applied Mathematics}, 81(1):153--172, 2021.

\bibitem{pikitch2012risks}
E.~K. Pikitch.
\newblock The risks of overfishing.
\newblock {\em Science}, 338(6106):474--475, 2012.

\bibitem{pinsky2011unexpected}
M.~L. Pinsky, O.~P. Jensen, D.~Ricard, and S.~R. Palumbi.
\newblock Unexpected patterns of fisheries collapse in the world's oceans.
\newblock {\em Proceedings of the National Academy of Sciences},
  108(20):8317--8322, 2011.

\bibitem{smith1994scaling}
T.~D. Smith and T.~D. Smith.
\newblock {\em Scaling fisheries: the science of measuring the effects of
  fishing, 1855-1955}.
\newblock Cambridge University Press, 1994.

\bibitem{worm2012future}
B.~Worm and T.~A. Branch.
\newblock The future of fish.
\newblock {\em Trends in ecology \& evolution}, 27(11):594--599, 2012.

\bibitem{dickey2010lessons}
B.~Worm and T.~A. Branch.
\newblock The future of fish.
\newblock {\em Trends in ecology \& evolution}, 27(11):594--599, 2012.

\end{thebibliography}

\begin{appendix}
\section{Proof of technical results}
\subsection{Proof of \eqref{Eq:Ode}}\label{Ap:Ode}
\begin{proof}[Proof of \eqref{Eq:Ode}]
We argue by contradiction. Thus there exists $\eta>0$, a sequence $\{V_k\}_{k\in \N}$ converging to 0, and, for any $k\in \N$, there exists $\alpha_k\in \mathcal M_\leq(\kappa,V_k)$ such that 
\[ \left|\lambda(K-\alpha_k-2\theta_{\alpha_k})-\lambda(K-2\overline \theta)\right|\geq \eta.\] For the sake of simplicity, define 
\[ \lambda_k:=\lambda(K-\alpha_k-2\theta_{\alpha_k}).\]
For any $k\in \N$, define $\p_k$ as  the principal eigenfunction of $-\mu \Delta-(K-\alpha_k-2\theta_{\alpha_k})$. Up to renormalisation we may thus assume that $\p_k$ satisfies
\begin{equation}\label{Eq:PhiK}
\begin{cases}
-\mu \Delta \p_k=(K-\alpha_k-2\theta_{\alpha_k})\p_k+\lambda_k\p_k&\text{ in }\O\,, 
\\ \frac{\partial \p_k}{\partial \nu}=0&\text{ on }\partial \O\,, 
\\ \p_k>0\text{ in }\O\,, \int_\O \p_k^2=1.\end{cases}\end{equation} Let $V_k=K-\alpha_k-2\theta_{\alpha_k}.$ By the maximum principle there exists $M_0$ such that
\[ \forall k\in \N\,, \Vert V_k\Vert_{L^\infty(\O)}\leq M_0\] whence we derive that there exists $M_1\in \R$ such that
\[ \sup_{k\in \N}|\lambda_k|\leq M_1.\]
In particular, by the weak formulation of \eqref{Eq:PhiK} there exists $M_2\in \R$ such that 
\[ \sup_{k\in \N}\Vert \p_k\Vert_{W^{1,2}(\O)}\leq M_2.\]Let $\lambda_\infty$ be a closure point of $\{\lambda_k\}_{k\in \N}$ and $\p_\infty$ be a (weak $W^{1,2}$, strong $L^2$) closure point of $\{\p_k\}_{k\in \N}$. As 
\[V_k\underset{k\to \infty}{\overset{L^2(\O)}\longrightarrow }K-2\overline \theta.\] Passing to the limit in the weak formulation of \eqref{Eq:PhiK}, as well as in the normalisation conditions, we obtain, on $\p_\infty$, the equation 
\begin{equation*}
\begin{cases}
-\mu \Delta =(K-2\overline \theta)\p_\infty+\lambda_\infty\p_\infty&\text{ in }\O\,, 
\\ \frac{\partial \p_\infty}{\partial\nu}=0&\text{ on }\partial \O\,, 
\\ \p_\infty\geq 0\text{ in }\O\,, \int_\O \p_\infty^2=1.\end{cases}\end{equation*} It thus appears that $\p_\infty$ is a constant-sign eigenfunction of the operator $-\mu \Delta-(K-2\overline \theta)$. As the first eigenfunction of an operator is the only one having constant sign we deduce that $\p_\infty$ is a first eigenfunction of $-\mu \Delta-(K-2\overline \theta)$, so that $\lambda_\infty=\lambda(K-2\overline \theta)$. As $\lambda_\infty=\lim_{k\to \infty}\lambda(K-\alpha_k-2\theta_{\alpha_k})$, this is a contradiction.

\end{proof}
\subsection{Proof of \eqref{Eq:CvEi1}}\label{Ap:ConvEigen}
\begin{proof}[Proof of \eqref{Eq:CvEi1}]
We argue by contradiction and assume that \eqref{Eq:CvEi1} does not hold. In particular there exists $\eta>0$, a sequence $\{V_{0,k}\}_{k\in \N}$ converging to zero and, for any $k\in \N$, $\alpha_k\in \mathcal M_\leq(\kappa,V_{0,k})$ such that 
\[\forall k\in \N\,, \left|\xi(\alpha_k)-\overline\xi \right|\geq \eta>0.\]
As $W_\alpha$ is uniformly bounded in $L^\infty(\O)$ for $V_0$ small enough from \eqref{Eq:CvUni}-\eqref{Bloodborne}, the sequence $\{\xi(\alpha_k)\}_{k\in \N}$ is uniformly bounded, say by a constant $M_0>0$:
\[ \forall k \in \N\,, \vert \xi(\alpha_k)\vert\leq M_0\] and thus, up to extracting a subsequence, it converges to some $\xi^*$.

In turn this implies that, if we define, for any $k\in \N$, the normalised eigenfunction $\p_k$ as the solution of 
\[
\begin{cases}
-\mu\nabla \cdot\left(\frac{|\n \p_k|}{\theta_{\alpha_k}}\right)-W_{\alpha_k}\p_k=\xi(\alpha_k)\p_k&\text{ in }\O\,, 
\\ \p_k\geq 0\,, \fint_\O \p_k^2=1\,, 
\\ \frac{\partial \p_k}{\partial \nu}=0&\text{ on }\partial \O,
\end{cases}
\]
the estimate \eqref{Eq:CvUni} and the fact that $\inf_{\overline \O}\overline\theta>0$ yield the existence of a constant $M_1$ such that 
\[ \forall k\in \N\,, \Vert \p_k\Vert_{W^{1,2}(\O)}\leq M_1.\]
We can thus extract a converging  (weakly in $W^{1,2}(\O)$, strongly in $L^2(\O)$) subsequence of $\{\p_k\}_{k\in \N}$, and, without relabelling we assume that the entire sequence thus converges to a $\p^*\in W^{1,2}(\O)$. Passing to the limit in the normalisation conditions provides us with 
\[ \p^*\geq 0\text{ in }\O\,, \fint_\O (\p^*)^2=1.\] Since $W_{\alpha_k}\underset{k\to \infty}\rightarrow \overline W$ strongly (in particular, weakly) in $L^2(\O)$ we finally obtain, passing to the limit in the eigenequation, that $\p^*$ solves
\[
\begin{cases}
-\mu \n\cdot\left(\frac{\n \p^*}{\overline \theta}\right)-\overline W\p^*=\xi^*\p^*&\text{ in }\O\,, 
\\ \frac{\partial \p^*}{\partial \nu}=0&\text{ on }\partial \O\,, 
\\ \p^*\geq 0\,, \fint_\O (\p^*)^2=1.
\end{cases}
\]
In particular, $\p^*$ is a positive eigenfunction of $\overline L$. As an eigenfunction of $\overline L$ has a constant sign if, and only if, it corresponds to the first eigenvalue, we deduce that $\xi^*=\overline \xi$, a contradiction.
\end{proof}

\subsection{Proof of convergence to an $\epsilon$-Nash equilibria}
{\begin{proposition}\label{prop-enash}
Algorithm \ref{algo1} with \texttt{tol}$=\epsilon>0$, in case of convergence it converges to an $\epsilon^{1/3}$-Nash equilibria.
\end{proposition}

\begin{proof}
Assume that one has set the tolerance of the algorithm up to $\epsilon>0$ and that the algorithm has converged. Then one has that
\begin{align*} 
I_1(\alpha_1^{k+1},\alpha_2^k) -&I_1(\alpha_1^k,\alpha_2^k)=\int_{\Omega} \alpha_1^{k+1}\theta_{\alpha_1^{k+1},\alpha_2^{k}}dx-\int_\Omega \alpha_1^{k} \theta_{\alpha_1^{k},\alpha_2^{k}}dx \\
&=\int_{\Omega} \alpha_1^{k+1}\theta_{\alpha_1^{k+1},\alpha_2^{k}}dx-\int_{\Omega} \alpha_1^{k+1}\theta_{\alpha_1^{k},\alpha_2^{k}}dx+\int_{\Omega} \alpha_1^{k+1}\theta_{\alpha_1^{k},\alpha_2^{k}}dx-\int_\Omega \alpha_1^{k} \theta_{\alpha_1^{k},\alpha_2^{k}}dx\\
&=\int_\Omega \alpha^{k+1}\left(\theta_{\alpha_1^{k+1},\alpha_2^{k}}-\theta_{\alpha_1^{k},\alpha_2^{k}}\right)dx+\int_\Omega \theta_{\alpha_1^{k},\alpha_2^{k}}\epsilon(x)dx
\end{align*}
where $\epsilon(x)$ is a function in $L^1$ such that $\|\epsilon(x)\|_{L^1}\leq \epsilon$. Then, by adapting the estimate obtained in \cite[Equation 2.4]{LouInfluence} one has the follow estimate
$$\forall (\alpha,\alpha')\in \mathcal{M}_\leq(\kappa,V_0)^2\quad \|\theta_{\alpha',\alpha_2^{k}}-\theta_{\alpha,\alpha_{2}^k}\|_{L^1}\leq C \| \alpha-\alpha'\|_{L^1}^{1/3} $$
where $C$ is independent of $\mu$ and $\alpha$. This allows us to state
$$I_1(\alpha_1^{k+1},\alpha_2^k) -I_1(\alpha_1^k,\alpha_2^k)\leq C\epsilon^{1/3}$$
The same argument applies for the second player and with that the proposition is proved.
\end{proof}
}
\end{appendix}
\newpage
\tableofcontents

\end{document}